\documentclass[10pt,twoside, a4paper, english, reqno]{amsart}
\usepackage[dvips]{epsfig}
\usepackage{a4wide}
\usepackage{amscd}
\usepackage{amssymb}
\usepackage{amsthm}
\usepackage{amsmath}
\usepackage{latexsym}
\usepackage{mathtools,dsfont}
\usepackage{graphicx,enumitem,dsfont}

\usepackage{upref}
\usepackage[colorlinks,backref]{hyperref}
\usepackage{color}
\usepackage[usenames,dvipsnames]{xcolor}

\usepackage{mathrsfs}
\def\h{{\Phi}}

\theoremstyle{plain}
\newtheorem{thm}{Theorem}[section]
\theoremstyle{plain}
\newtheorem{lem}[thm]{Lemma}

\newtheorem{cor}[thm]{Corollary}

\theoremstyle{definition}
\newtheorem{defi}{Definition}[section]
\newtheorem{rem}{Remark}[section]

\newtheorem*{maintheorem*}{Main Theorem}
\newtheorem*{maincorollary*}{Main Corollary}

\newenvironment{Assumptions}
{
\setcounter{enumi}{0}

\begin{enumerate}}
{\end{enumerate} }

\makeatletter
\newsavebox{\@brx}
\newcommand{\llangle}[1][]{\savebox{\@brx}{\(\m@th{#1\langle}\)}%
 \mathopen{\copy\@brx\mkern2mu\kern-0.9\wd\@brx\usebox{\@brx}}}
\newcommand{\rrangle}[1][]{\savebox{\@brx}{\(\m@th{#1\rangle}\)}%
 \mathclose{\copy\@brx\mkern2mu\kern-0.9\wd\@brx\usebox{\@brx}}}
\makeatother
\newcommand{\norm}[1]{\left\|#1\right\|}

\newcommand{\R}{\ensuremath{\mathbb{R}}}
\newcommand{\D}{{\ensuremath{\R^d}}}

\newcommand{\E}{\ensuremath{\mathbb{E}}}

\newcommand{\Div}{\mathrm{div}\,}
\newcommand{\sgn}{\mathrm{sign}}
\newcommand{\rd}{\ensuremath{\mathbb{R}^d}}
\newcommand{\supp}{\ensuremath{\mathrm{supp}\,}}
\newcommand{\goto}{\ensuremath{\rightarrow}}
\newcommand{\grad}{\ensuremath{\nabla}}
\newcommand{\eps}{\ensuremath{\varepsilon}}

\newcommand{\ep}{\eps}

\newcommand{\intrd}{\int_{\mathbb{R}^{d}}}

\newcommand{\Rd}{\mathbb{R}^d}
\newcommand{\fr}{\mathscr{L}_{\lambda}}

\def\dint{\displaystyle\int}
\def\d#1{\mathrm{ d#1}}

\numberwithin{equation}{section} \allowdisplaybreaks

\DeclareMathOperator{\supess}{sup\,ess}

\title[Stochastic degenerate fractional problem]
{A fractional degenerate parabolic-hyperbolic Cauchy \\problem with noise}



\author[Neeraj Bhauryal]{Neeraj Bhauryal}
\address[Neeraj Bhauryal]{\newline
	Centre for Applicable Mathematics,
	Tata Institute of Fundamental Research,
	P.O.\ Box 6503, GKVK Post Office,
	Bangalore 560065, India}
\email[]{neeraj@tifrbng.res.in}

\author[Ujjwal Koley]{Ujjwal Koley}
\address[Ujjwal Koley]{\newline
	Centre for Applicable Mathematics,
	Tata Institute of Fundamental Research,
	P.O.\ Box 6503, GKVK Post Office,
	Bangalore 560065, India}
\email[]{ujjwal@tifrbng.res.in}

\author[Guy Vallet]{Guy Vallet}
\address[Guy Vallet]{\newline 
	LMAP UMR- CNRS 5142, IPRA BP 1155, 64013 Pau Cedex, France}
\email[]{guy.vallet@univ-pau.fr}

\begin{document}
\begin{abstract}
We consider the Cauchy problem for a stochastic scalar parabolic-hyperbolic equation in any space dimension with nonlocal, nonlinear, and possibly degenerate diffusion terms. The equations are nonlocal because they involve fractional diffusion operators.  We adapt the notion of stochastic entropy solution and provide a new technical framework to prove the uniqueness. The existence proof relies on the vanishing viscosity method.  Moreover, using bounded variation (BV) estimates for vanishing viscosity approximations, we derive an explicit continuous dependence estimate on the nonlinearities and derive error estimate for the stochastic vanishing viscosity method. In addition, we develop uniqueness method ``\`{a} la Kru\v{z}kov" for more general equations where the noise coefficient may depends explicitly on the spatial variable.
\end{abstract}

\maketitle
\tableofcontents

\section{Introduction}
In this paper, we consider the following initial value problem for the stochastic nonlinear, nonlocal conservation law
\begin{equation}
 \label{eq:stoc_frac}
 \begin{cases} 
 du(t,x) + \Big[\mathscr{L}_{\lambda}[A(u(t,\cdot))](x)- \Div f(u(t,x))\Big]\,dt
 =\h(u(t,x)) \,dW(t), & \quad \text{in } Q_T, \\
 u(0,x) = u_0(x), & \quad \text{in } \D,
 \end{cases}
 \end{equation}
 where $Q_T:=\D \times (0,T)$ with $T>0$ fixed, $u_0$ is the given initial
 function, $f: \R \to \D$, $A:\R \to \R$ are  given (sufficiently smooth) functions 
(see Section~\ref{sec:tech} for the complete list of assumptions), and $\mathscr{L}_\lambda[u]$ denotes the fractional Laplace operator $(-\Delta)^\lambda[u]$ of order $\lambda \in (0,1)$, defined pointwise as follows
\begin{align*}
\fr[\varphi](x) := c_{\lambda}\, \text{P.V.}\, \int_{|z|>0} \frac{\varphi(x) -\varphi(x+z)}{|z|^{d + 2 \lambda}} \,dz,
\end{align*}
for some constants $c_{\lambda}>0$, and a sufficiently regular function $\varphi$.
Note that $A'$ is allowed to be zero on an interval so that, as in the local case (see \cite{Carrillo}), the problem may degenerate in a free set. Let $\big(\Omega, \mathcal{F}, \mathbb{P}, (\mathcal{F}_t )_{t\ge0} \big)$ be a stochastic basis, where $\big(\Omega, \mathcal{F}, \mathbb{P} \big)$ is a probability space and $(\mathcal{F}_t)_{t \ge 0}$ is a complete filtration with the usual assumptions. We assume that $W$ is a cylindrical Wiener process: $W(t)= \sum_{k\ge 1} e_k \beta_k(t)$ with $(\beta_k)_{k\ge 1}$ being mutually independent real valued standard Wiener processes, and $(e_k)_{k\ge 1}$ a complete orthonormal system in a separable Hilbert space $\mathbb{H}$. The map $ u \mapsto \h(u)$ is  an $\mathbb{H}$-valued function signifying the multiplicative nature of the noise. Moreover, for each $v$ in $L^2(\R^d)$, we consider the mapping $\h(v): \mathbb{H}\goto L^2(\R^d)$ defined by $\h(v)e_k= g_k(v(\cdot))$. In particular, we suppose that $g_k$ is Lipschitz-continuous and $\mathbb{G}^2(r):= \sum_{k\ge 1} g_k^2(r)$.

\subsection{Stochastic Entropy Formulation}
It is well-known that the nonlinearity of the flux function and a possible degeneracy of the diffusion term in \eqref{eq:stoc_frac} can lead to a loss of regularity in the solution, even with smooth initial data. Thus, weak solutions to \eqref{eq:stoc_frac} must be sought. However, weak solutions are not necessarily uniquely determined by their initial data. Consequently, an admissibility condition, so called {\em entropy condition}, must be imposed to single out the physically relevant solution. To describe the entropy framework for \eqref{eq:stoc_frac}, we first introduce the notion of entropy-entropy flux pair.

\begin{defi}[entropy-entropy flux pair]
	A pair $(\eta,\zeta) $ is called an entropy-entropy flux pair 
	if $ \eta \in C^2(\R) $, $\eta \ge0$ and 
	$\zeta = (\zeta_1,\zeta_2,....\zeta_d):\R \mapsto\rd $ is a vector field satisfying
	$\zeta'(r) = \eta'(r)f'(r)$, for all $r$. An entropy-entropy flux pair $(\eta,\zeta)$ is called 
	convex if $ \eta^{\prime\prime} \ge 0$.  
\end{defi}
With the help of  convex entropy-entropy flux pairs $(\eta,\zeta)$, we are ready to define the notion of stochastic entropy solution. To this end, let us first split the non-local operator $\fr$ into two terms: for each $r>0$, we write $\fr[\varphi] := \mathscr{L}_{\lambda, r}[\varphi] + \mathscr{L}_{\lambda}^{r}[\varphi]$, where
\begin{align*}
\mathscr{L}_{\lambda, r}[\varphi](x):= c_{\lambda}\, \text{P.V.}\, \int_{|z|\le r} \frac{\varphi(x) -\varphi(x+z)}{|z|^{d + 2 \lambda}} \,dz, \quad
\mathscr{L}_{\lambda}^{r}[\varphi](x):= c_{\lambda}\,\int_{|z|> r} \frac{\varphi(x) -\varphi(x+z)}{|z|^{d + 2 \lambda}} \,dz.
\end{align*}
 
\addtocounter{footnote}{1}

\begin{defi}[Stochastic Entropy Solution] 
\label{Defi_Entropy_formulation}
A square integrable $ L^2(\R^d )$-valued $\{\mathcal{F}_t: t\geq 0 \}$-predictable  process $u$\footnote{$u \in N_w^2(0,T,L^2(\R^d ))$ in the sequel.} is called a stochastic entropy solution of \eqref{eq:stoc_frac} with initial data $u_0$ if $u \in L^\infty(0,T,L^2(\Omega\times\R^d))$ and, given any non-negative test function $\varphi\in C_{c}^{1,2}([0,\infty )\times\R^d)$\footnote{$\varphi$ in $H^1(Q)\cap L^2(0,T,H^2(\R^d))$ is possible by a density argument.}, any convex entropy flux pair $(\eta,\zeta)$ and any positive $r$, the following inequality holds:
\begin{align*}
&0\leq  \int_\D\eta(u_0(x)-k)\varphi(0) \,dx + \int_{Q_T} \eta(u(t,x) - k)\partial_t\varphi(t,x) -\nabla\zeta(u(t,x))\cdot \varphi(t,x)\,dx\,dt\\
&+ \int_{Q_T} \eta'(u(t,x) -k) \h(u(t,x))\varphi(t,x) \,dx \,dW(t) + \frac{1}{2}\int_{Q_T}\mathbb{G}^2(u(t,x))\eta^{\prime\prime}(u(t,x) - k)\varphi(t,x) \,dx\,dt\\ 
&-\int_{Q_T} \Big[\mathscr{L}^r_\lambda[A(u(t,\cdot))](x) \varphi(t,x) \eta'(u(t,x)-k) + A^\eta_k(u(t,x)) \mathscr{L}_{\lambda,r}[\varphi(t,\cdot)](x)\Big]\,dx\,dt,\hspace{0.5cm} \mathbb{P}-a.s.
\end{align*}
\end{defi}
A formal derivation of the above entropy inequality can be found in Appendix~\ref{appendix_entropy}. Let us only remark here that for any positive $r$, assuming that $u$ belongs to $L^2(Q_T)$ is enough to give a sense to  
\begin{align*}
\int_{Q_T} \Big[\mathscr{L}^r_\lambda[A(u(t,\cdot))](x) \varphi(t,x) \eta'(u(t,x)-k) + A^\eta_k(u(t,x)) \mathscr{L}_{\lambda,r}[\varphi(t,\cdot)](x)\Big]\,dx\,dt,
\end{align*}
where $A^\eta_k(u)$ is defined in Section~\ref{sec:tech}.

\begin{rem}\label{weakcont}
Note that since $u \in L^\infty(0,T,L^2(\Omega\times\R^d))$ and $\partial_t [ u - \int_0^t \h(u) dW] \in L^2(\Omega\times(0,T),H^{-1}(\D))$, based on the properties of It\^o's integrals, one concludes that $u - \int_0^t \h(u) dW \in C_w([0,T],L^2(\Omega\times\R^d))$ first (\cite[p.262-263]{Temam}) and then $u \in C_w([0,T],L^2(\Omega\times\R^d))$.
\end{rem}

\subsection{Earlier works and outline of this paper}   

The equation \eqref{eq:stoc_frac} can be viewed as a stochastic perturbation of a nonlocal degenerate parabolic-hyperbolic equation. In the absence of nonlocal term along with $\h=0$, equation \eqref{eq:stoc_frac} becomes a standard conservation law in $\R^d$. For conservation laws, the question of existence and uniqueness of solutions was first settled in the pioneer papers of Kru\v{z}kov \cite{Kruzkov} and Vol'pert \cite{Volpert}. In the case $\h=0$, well-posedness of Cauchy problems was studied by Alibaud \cite{Alibaud}, Cifani \& Jakobsen \cite{CifaniJakobsen}. For the linear case i.e., $A(x)=x$, well posedness results for \eqref{eq:stoc_frac} has been recently developed in \cite{BhKoleyVa}.
\smallskip

The study of stochastic balance laws have a recent yet intense history. In fact, Kim \cite{KIm2005} extended  Kru\v{z}kov well-posedness theory to one dimensional balance laws that are driven by additive Brownian noise, and Vallet \& Wittbold \cite{Vallet_2009} to the multidimensional Dirichlet problem. On the other hand, Feng $\&$ Nualart introduced a notion of strong entropy solution in \cite{nualart:2008}, for which the uniqueness was established in the class of entropy solutions in any space dimension, in the multiplicative case. The existence was proven by using a stochastic version of the compensated compactness method and it was valid only for \emph{one} spatial dimension. To overcome this problem, Debussche $\&$ Vovelle \cite{Vovelle2010} introduced
the kinetic formulation of such problems and as a result they were able to establish the well-posedness
of multidimensional stochastic balance law via \emph{kinetic} approach. 
A number of authors have contributed since then, and we mention the works of 
Bauzet et al. \cite{BaVaWit_2012,BaVaWit_JFA}, Biswas et al. \cite{BisMajKarl_2014}. For degenerate
parabolic equations, we mention the works of Vallet \cite{Vallet_2005, Vallet_2008}, Debussche et al. \cite{martina}, Koley et al. \cite{Koley1, Koley2}.
We also mention works by Chen et al. \cite{Chen-karlsen_2012}, and Biswas et al. \cite{BisKoleyMaj}, where the well posedness of the problem for entropy solution is established in $L^p \cap BV$, via BV framework. 
Moreover, they were able to develop a continuous dependence theory for multidimensional balance laws and as a by product they derived an explicit \emph{convergence rate} of the approximate solutions to the underlying problem. 
\smallskip

In a nutshell, the main difficulty in above mentioned works is that, by virtue of It\^o's formula, it is not possible to use the usual Kru\v{z}kov's entropies and therefore adaptation of the deterministic ideas are quite involved. In fact, one has to work with smooth approximations of the absolute-value function and to deal with the consequences of this change. On the other hand, the (well posedness)  analysis of fractional conservation laws relies on an essential ingredient, namely the following (Kato's type of) inequality
\begin{equation}\label{imp_10}
\sgn\Big(u(x)-v(y)\Big) \Big(\fr[A(u)](x)- \fr[A(v)](y)\Big) \le \fr\Big[|A(u) -A(v)|\Big](x,y),
\end{equation}
where we have used the standard notation of fractional derivative for a function of two variables on the right hand side of the above inequality (for more details, see Cifani \& Jakobsen \cite[Section 3]{CifaniJakobsen}). We remark that the above inequality \eqref{imp_10} is true for the signum function, and does not hold (due to the presence of nonlinear function $A$) for a regularized version of the signum function which is required to establish the well posedness theory in the stochastic case. In view of the above discussions, it is clear that there is a gap between the stochastic theory and its deterministic counterpart for nonlinear fractional conservation laws. Incidentally, in the linear case \textit{i.e.}, $A(x)=x$, inequality \eqref{imp_10} holds for a regularized version of the signum function. Indeed, this has been exploited by the authors in \cite[Lemma 3.4]{BhKoleyVa} to establish the well posedness theory in the linear case.

The present proof of well posedness contains two new ingredients: 
\begin{itemize}
\item [(a)] A change in computing hierarchical limits with respect to various parameters involved, \textit{i.e.}, we pass to the limit in the parameter $\delta$ (related to the approximation of the absolute value function), before passing to the limit in the parameter $l$ (related to approximation of  the ``doubling of variable'' constant). This is a significant departure from the existing literature, and seems necessary to accommodate Kato's type of inequality \eqref{imp_10} in the stochastic setup. However, in view of Definition~\ref{Defi_Entropy_formulation}, this change immediately invites rudimentary problem due to the presence of the term consisting of $\eta_{\delta}''$. We overcome this difficulty by making use of an integration by parts formula (for details, see Lemma~\ref{lem:stochastic-terms}). Needless to mention that the above change in hierarchy enforces us to revisit all the terms involved in the entropy inequality.
\item [(b)] A typical approach to prove the existence of solutions for the regularized/viscous problem is often based on a semi-implicit time discretization (see \cite{BaVaWit_2012,BhKoleyVa}). However, due to the presence of the nonlinear fractional diffusion, it seems not possible to adapt such techniques here. Therefore, following \cite[Section 4]{martina}, we use a general method of constructing martingale solutions of SPDEs (see Section~\ref{viscous}), that does not rely on any kind of martingale representation theorem. This argument is based on a compactness method where one needs uniform estimates to demonstrate tightness results and this yields the convergence of the approximate sequence on another probability space and the existence of martingale solution follows. The existence of a pathwise solution is obtained by  Gy\"{o}ngy-Krylov's characterization of convergence of probability. Since we are working on the full space (not on a torus, as in \cite{martina}), we require weighted $L^2$-estimates for solutions to successfully demonstrate a compactness argument.
\end{itemize}
We also develop a \emph{continuous dependence theory} for stochastic entropy solution of \eqref{eq:stoc_frac}, which, in turn, is used to establish an \emph{error estimate} for the vanishing viscosity method. To that context, we first address the question of existence, uniqueness of stochastic BV entropy solution in $L^2(\R^d) \cap BV(\R^d)$ of the problem \eqref{eq:stoc_frac}. To display essential new ideas
in a simpler context, we only provide a continuous dependence estimate on the nonlinearities coming from the (fractional) nonlocal term. For the continuous dependence estimates on other nonlinearities present in the equation, one can follow \cite[Section 4]{BhKoleyVa}. Furthermore, making use of the crucial BV estimate, we derive an error estimate for the vanishing viscosity method provided that the initial data lies in $u_0 \in L^2(\R^d) \cap BV(\R^d)$. Finally, we turn our discussions to more general stochastic nonlocal degenerate problems driven by Brownian noise, namely when the coefficient of the Brownian noise $\Phi$ has an explicit dependency on the spatial position $x$ (cf. equation \eqref{eq:stoc_frac_001}) as well. In this case, the uniqueness proof again requires a change in order in computing limits with respect to various parameters. This technical hurdle compelled us to analyze the general equation with values of $\lambda$ in $[0,1/2)$ only.

We remark that our solution concept is completely different from the concept of \emph{random entropy solution} for fractional conservation laws incorporating randomness in the initial data and fluxes. Several results are available in that direction. For more details on the well-posedness theory of random entropy solution, we refer to \cite{ujjwal, koley2013multilevel,Koley3}.

The rest of the paper is organized as follows:  we describe the technical framework and state the main results in Section~\ref{sec:tech}. In Section~\ref{uniqueness}, we present a proof of the result of uniqueness by using a variant of Kru\v{z}kov's doubling of variable technique, and then derive stability results for \eqref{eq:stoc_frac}. Section~\ref{ContDep_NonL} is devoted to deriving the continuous dependence estimate on nonlinearities, while Section~ \ref{rateofconv} deals with the error estimates. Existence, uniqueness and several \textit{a priori} bounds of viscous solutions are presented in Section~\ref{viscous}. Finally, in Section~\ref{extension}, we establish a uniqueness argument for more general nonlocal stochastic problems, and a formal derivation of the entropy inequality is presented in Appendix~\ref{appendix_entropy}.

\section{Technical Framework and Statement of the Main Results}
\label{sec:tech}
Throughout this paper, we use the letter $C$ to denote various generic constants. There are situations where this constant may change from line to line, but the notation is kept unchanged so long as it does not impact the central idea. 
In general, if $G \subset \R^k$, $\mathcal{D}(G)$ denotes the restriction to $G$ of $\mathcal{D}(\R^k)$ functions $u$ such that \supp($u$)$\cap G$ is compact.  Then, $\mathcal{D}^+(G)$ will denote the subset of nonnegative elements of $\mathcal{D}(G)$.

For a given separable Banach space $X$, we denote by $N^2_w(0,T,X)$ the space of square integrable predictable $X$-valued processes (cf. \cite{daprato} p.94 or \cite{PrevotRockner} p.28 for example). Furthermore, we denote $BV(\R^d)$ as the set of integrable functions with bounded variation on $\R^d$ endowed with the norm $\|u\|_{BV(\R^d)}= \|u\|_{L^1(\R^d)} + TV_{x}(u)$, where $TV_{x}$ is the total variation of $u$ defined on $\R^d$.
\smallskip

We denote by $\mathcal{E}$ the set of nonnegative convex functions in $C^{2,1}(\R)$  approximating the absolute-value function, such that $\eta(0)=0$ and that there exists $\delta>0$ such that $\eta'(x)=1$ (resp. $-1$) if $x>\delta$  (resp. $x<-\delta$ ). Then, $\eta^{\prime\prime}$ has a compact support and $\eta$ and $\eta'$ are Lipschitz-continuous functions.
\smallskip

For convenience, denote by $\sgn(x)=\frac{x}{|x|}$ if $x\neq 0$ and $0$ otherwise;  $F(a,b)=\sgn(a-b)[f (a)-f (b)]$ and $F^\eta(a,b)=\int_b^a\eta'(\sigma - b) f'(\sigma)\,\d{\sigma}$. Note, in particular, that $F$ and $F^\eta$ are Lipschitz-continuous functions.
\\
Similarly, denote by $A^\eta_k(a)=\int_k^a\eta'(\sigma - k) A'(\sigma)\d{\sigma}$.
\smallskip

 Next, we write down some useful properties of the fractional operator which are used in the sequel, for a detailed description, consult \cite[Appendix B]{BhKoleyVa}. First note that 
\begin{align*}
\mathscr{L}_{\lambda}[\varphi](x)&= c_{\lambda}\, \text{P.V.}\, \int_{|z|\le r} \frac{\varphi(x) -\varphi(x+z)}{|z|^{d + 2 \lambda}} \,dz + c_{\lambda}\,\int_{|z|> r} \frac{\varphi(x) -\varphi(x+z)}{|z|^{d + 2 \lambda}} \,dz \\
& = c_{\lambda}\, \int_{|z|\le r} \frac{\varphi(x) -\varphi(x+z) + z \cdot \nabla \varphi(x)}{|z|^{d + 2 \lambda}} \,dz + c_{\lambda}\,\int_{|z|> r} \frac{\varphi(x) -\varphi(x+z)}{|z|^{d + 2 \lambda}} \,dz,
\end{align*}
for some constants $c_{\lambda}$, $\lambda \in (0,1)$, and a sufficiently regular function $\varphi$.
Moreover, for all $u,v \in H^{\lambda}(\D)$, denoting the convolution operator by $\star$, we have
\begin{align*}
u \star \mathscr{L}_\lambda [v] &= v \star \mathscr{L}_\lambda [u], \\
\langle\mathscr{L}_\lambda [u],v\rangle &=\frac{c_{\lambda}}{2}\int_{\D}\int_{\D} \frac{\big(u(x)-u(y)\big) \big(v(x)-v(y)\big)}{|x-y|^{d+2\lambda}}\,dx\,dy =\int_{\D} \mathscr{L}_{\lambda/2} [u](x) \,\mathscr{L}_{\lambda/2} [v](x) \,dx, \\
\langle \mathscr{L}_{\lambda,r}[u],v \rangle &= \frac{c_{\lambda}}{2} \int_\D \int_{|z| \le r} \frac{(u(x)- u(x+z))(v(x)-v(x+z))}{|z|^{d+2\lambda}} \,dz\,dx.
\end{align*}
The primary objective of this paper is to settle the problem of the existence and uniqueness of a solution for the Cauchy problem \eqref{eq:stoc_frac}, and we do so under the following assumptions:
 \begin{Assumptions}
 \label{aasumptions}
 	\item \label{A1} The initial function $u_0$ is a deterministic function in $L^2(\R^d)$\footnote{Note that $\|u\|_{BV} <\infty$ will be assumed for stability analysis.}.
 	\item \label{A2}  $ f=(f_1,f_2,\cdots, f_d):\R\rightarrow \R^d$ is a Lipschitz continuous function with $f_k(0)=0$, for 
 	all $1\le k\le d$.
 	\item \label{A3} $A: \R \to \R$  is non-decreasing Lipschitz continuous function with $A(0)=0$. 
 	\item \label{A4} 
We assume that $g_k:\R \to \R$ satisfies $g_k(0)=0$ for all $k\ge 1$. Moreover, there exists a positive constant $K > 0$  such that, for all $u,v \in \R$,  
 	\begin{align*} 
 	\sum_{k\ge 1}\big| g_k(u)-g_k(v)\big|^2  \leq K |u-v|^2 \quad \text{and}\quad \mathbb{G}^2(u)= \sum_{k\ge 1} g_k^2(u)\le K\,|u|^2.
 	\end{align*}  
 \end{Assumptions}

\begin{rem}
 In view of Assumption \ref{A4}, for any $v\in L^2(\R^d)$ (resp. $H^1(\R^d)$), $\h(v)$ is a Hilbert-Schmidt operator from the separable Hilbert space $\mathbb{H}$ to $L^2(\R^d)$ (resp. $H^1(\R^d)$). Therefore, for a given predictable process 
 $v\in L^2(\Omega;L^2(0,T;L^2(\R^d)))$, the stochastic integral $t\mapsto \int_0^t \h(v)dW(s)$ is a well-defined process taking values in the Hilbert space $L^2(\R^d)$. Moreover, the trajectories 
 of $W$ are $\mathbb{P}$- a.s. continuous in $\mathbb{H}_0 \supset \mathbb{H}$, where 
 \begin{align*}
  \mathbb{H}_0:= \Big\{ v= \sum_{k\ge 1} v_k e_k:\,\, \sum_{k\ge 1} \frac{v_k^2}{k^2} < + \infty \Big\} 
 \end{align*}
 endowed with the norm $\|v\|_{\mathbb{H}_0}^2= \sum_{k\ge 1} \frac{v_k^2}{k^2} $ where $v= \sum_{k\ge 1} v_k e_k$. Furthermore, the embedding $\mathbb{H}\hookrightarrow \mathbb{H}_0$ is Hilbert-Schmidt (see \cite{daprato}).
\end{rem}


Like its deterministic counterpart, the result of existence of entropy solutions is largely related to the study of associated viscous problems. For a small positive number $\eps>0$, consider the parabolic perturbation
\begin{align}
\label{eq:viscous-Brown} 
du_\eps(t,x) -[\eps \Delta u_\eps(t,x)+ \Div f(u_\ep(t,x))]\,dt & + \mathscr{L}_{\lambda}[A(u_\eps(t, \cdot))](x)\,dt = \h(u_\eps(t,x))\,dW(t) 
\end{align}
of \eqref{eq:stoc_frac} with initial the data $u_{\eps}(0,x)=u_0^{\eps}(x)\in H^1(\R^d)\cap L^1(\R^d)$, where $u_0^{\eps}$ is a suitable approximation of the initial condition satisfying: $u_0^{\eps}$ converges to $u_0$ in $L^2(\R^d)$, $\sqrt{\eps}u_0^{\eps}$ is bounded in $H^1(\R^d)$ by $C\|u_0\|_{L^2(\R^d)}$; if moreover $u_0\in L^1(\R^d)$ then $\|u_0^{\eps}\|_{L^1(\R^d)} \leq \|u_0\|_{L^1(\R^d)}$ and if $u_0\in BV(\R^d)$ then $TV(u_0^{\eps}) \leq TV(u_0)$. We first propose a result of existence of the weak solution\footnote{Weak is understood here in the sense of PDE.} to the regularized problem \eqref{eq:viscous-Brown} by adapting the argument of Gy\"{o}ngy and Krylov (see \cite{GyongyKrylov}) based on a result from Yamada and Watanabe (see \cite{YamadaWatanabe}) in Section \ref{viscous}.

\begin{thm}[Existence and Uniqueness of Viscous Solution]
\label{prop:vanishing viscosity-solution} \hfill \\
Let assumptions \ref{A1}-\ref{A4} hold and $u_0^{\eps}\in H^1(\R^d)\cap L^1(\R^d)$ as presented above. Then for any $\eps>0$, there exists a unique solution $u_\eps \in N_w^2(0,T,H^1(\R^d))$, pathwise continuous in $L^2(\R^d)$ such that $\partial_t \big(u_\eps - \int_0^t \h(u_\eps(s,\cdot))\, dW(s)\big)
\in L^2(\Omega\times(0,T),H^{-1}(\R^d))$, to the problem \eqref{eq:viscous-Brown}. Moreover, the solution 
$u_\eps \in C([0,T];L^2(\Omega \times\R^d))$ and there exists a constant $C>0$, independent of $\eps$, such that
\begin{align}
\sup_{0\le t\le T} \E\Big[\big\|u_\eps(t)\big\|_{L^2(\R^d)}^2\Big]  + \eps \int_0^T \E\Big[\big\|\grad u_\eps(s)\big\|_{L^2(\R^d)}^2\Big]\,ds + \int_0^T \E\Big[\big\| A(u_\eps(s))\|_{H^{\lambda}(\R^d)}^2\Big]\,ds \le C.
\end{align} 
Assuming that $u_0 \in BV(\R^d)$, there exists a constant $C>0$, independent of $\eps$, such that for any time $t>0$, 
\begin{align*}
\sup_{\eps>0} \E\Big[ \|u_\eps(t)\|_{L^1(\R^d)}\Big] \le  C \, \|u_0\|_{L^1(\R^d)}, \quad
&\sup_{\eps>0} \E\Big[ TV_x(u_\eps(t))\Big] \le TV_x(u_0).
 \end{align*}
\end{thm}
The estimates of the above theorem are obtained by classical arguments and are given in Subsection \ref{sec:apriori+existence}.  
With the above results at hand, we are now in a position to state the existence and regularity part of the main results of this paper. The uniqueness and stability one is proved in Section \ref{uniqueness}.
\begin{thm}[Existence and Uniqueness]
\label{uniqueness_new}
Let the assumptions \ref{A1}-\ref{A4} be true. Then there exists a unique stochastic entropy solution for the Cauchy problem \eqref{eq:stoc_frac} in the sense of Definition~\ref{Defi_Entropy_formulation}.
Furthermore,  $A(u) \in L^2(\Omega\times(0,T),H^\lambda(\R^d))$ and, assuming that $u_0 -v_0 \in L^1(\R^d)$, we have for every $t$ in $(0,T)$,
\begin{align*}
\E \Big[\int_{\D}  |u(t,x) -v(t,x)| \,dx\Big] \leq& \int_\D |u_0-v_0|\,dx.
\end{align*}
Assuming moreover that $u_0 \in BV(\R^d)$ implies that $u$ is a BV-entropy solution in the sense that for any $t>0$,
\begin{align*}
   \E \Big[\|u(t)\|_{BV(\R^d)} \Big] \le \|u_0\|_{BV(\R^d)}.
\end{align*}
\end{thm}

\begin{thm}[Continuous Dependence Estimate]
\label{continuous-dependence}
Consider two sets of given data $(u_0, f, A,\h, \lambda)$ and $(v_0,f,B,\h,\lambda)$ satisfying the assumptions \ref{A1}-\ref{A4} and assume moreover that the initial data $u_0$ is in $L^1(\R^d)$, that $v_0$ is in $BV(\R^d)$ and $f^{\prime\prime}\in L^\infty(\R^d)$.
Denote by $u$ and $v$ the corresponding solutions of \eqref{eq:stoc_frac}, which are also BV entropy solutions. Then there exists a constant $C$, only depending on $ \|u_0\|_{L^1(\R^d)}$, and $\|v_0\|_{BV(\R^d)}$, such that for all $0<t<T<+\infty$, 
\begin{align*}
  \E & \Big[ \int_{\D} \big| u(t,x)-v(t,x)\big|\,dx \Big] 
 \le   \norm{u_0-v_0}_{L^1(\R^d)} + C\Big(\norm{u_0}_{L^1(\R^d)},\norm{v_0}_{BV(\R^d)}\Big)\Big[T\|A'-B'\|_\infty\Big]^{\frac1{1+\lambda}}. 
\end{align*}
\end{thm}
\begin{cor}[Rate of Convergence]
\label{rate}
Let the assumptions \ref{A1}-\ref{A4} hold, $u_0 \in BV(\R^d)$, $f^{\prime\prime}\in L^\infty$ and let $u$ be the entropy solution of \eqref{eq:stoc_frac} and $u_\eps$ be a weak solution to the problem \eqref{eq:viscous-Brown}. Then there exists a constant $C$, depending only on 
$|u_0|_{BV(\R^d)}$, such that for all $t>0$, 
\begin{align*}
 \E\Big[\|u_\eps(t,\cdot)-u(t,\cdot)\|_{L^1(\R^d)}\Big] \le C \eps^\frac{1}{2},
\end{align*}
provided the initial error $\|u_0^\eps-u_0\|_{L^1(\R^d)} \approx \mathcal{O}(\eps^\frac{1}{2})$.

\end{cor}

\section{Proof of Theorem \ref{uniqueness_new}: Uniqueness of Entropy Solutions}
\label{uniqueness}
\subsection{Kato's inequality}
\label{Kato}
We follow the usual strategy of adapting a variant of the classical Kru\v{z}kov's doubling of variables approach. Note that, the main difficulty lies in ``doubling'' the time variable which gives rise to stochastic integrands that are anticipative and hence can not be interpreted in the usual It\^{o} sense. To get around this problem, we shall compare first two weakly converging sequences of viscous approximations: let $u(t,x,\alpha)$ and $v(s,y,\beta)$, $\alpha, \beta \in (0,1)$, be Young measure-valued limit processes associated to the sequences  
$u_\theta(t,x)$ and $u_\eps(s,y)$ of weak solutions of \eqref{eq:viscous-Brown} with regularized initial data $v^{\theta}_0(x)$ and $u^{\eps}_0(y)$ respectively. 

For technical reasons, as observed in \cite[Section 3]{BaVaWitParab} concerning the local diffusion case, the regularity $u_\ep \in H^1(\D)$ is not sufficient to prove Kato's inequality (cf. term $J_3$ in Lemma \ref{lem:stochastic-terms}). Therefore, we need to regularize $u_\ep$ by a space convolution. 
Let $\{\rho_\gamma\}_{\gamma}$ be a given mollifier-sequence in $\R^d$, by using the test function $\varphi*\rho_\gamma$ in the equation satisfied by $u_\eps$ for any $\varphi \in \mathcal{D}(\R^{d+1})$, one gets that $u_\eps * \rho_\gamma$ is a solution to a perturbed stochastic problem.  The way to obtain the local terms are detailed in \cite[Section 3]{BaVaWitParab}, so we only need to focus on the nonlocal one. 

Thanks to the regularity of the functions $A(u_\eps)(\omega,t) \in H^1(\D)$ and $\varphi(t)\in \mathcal{D}(\D)$, one gets that 
\begin{align*}
\int_\D \mathscr{L}_{\lambda/2}[A(u_\eps)] \mathscr{L}_{\lambda/2}[\varphi*\rho_\gamma] \,dx =& \int_\D A(u_\eps) \mathscr{L}_{\lambda}[\varphi*\rho_\gamma] \,dx = \int_\D A(u_\eps) \mathscr{L}_{\lambda}[\varphi]*\rho_\gamma \,dx
\\=&
 \int_\D (A(u_\eps)*\rho_\gamma) \mathscr{L}_{\lambda}[\varphi ]\,dx= \int_\D \mathscr{L}_{\lambda}[A(u_\eps)*\rho_\gamma]\, \varphi \,dx.
\end{align*}
Thus,  $u_\eps *\rho_\gamma(t=0) = u_{\eps}(0)*\rho_\gamma$ and  the above mentioned perturbed stochastic problem is 
\begin{eqnarray*}
\partial_t \bigg[u_\eps*\rho_\gamma - \dint_0^t \h(u_\eps)*\rho_\gamma dW\bigg] - \big[\eps \Delta (u_\eps *\rho_\gamma) - \mathscr{L}_\lambda [A(u_\eps)*\rho_\gamma] + \Div f (u_\eps)*\rho_\gamma \big] = 0.
\end{eqnarray*}
Observe that the  problem is posed in $L^2(\D)$ and, for convenience, let us denote in the sequel $v*\rho_\gamma$ by $v^\gamma$, for any generic $v$ in $L^2(\D)$.

Let $\rho_n$ and $\rho_m$ be standard nonnegative 
mollifiers on $\R$ and $\R^d$ respectively such that 
$\supp(\rho_n) \subset [-2/n,0]$ and $\supp(\rho_m) \subset \overline{B_{1/m}}(0)$.  Given a  nonnegative test 
function $\varphi\in C_c^{1,2}([0,\infty)\times \rd)$, we define 
\begin{align}
\label{test_function}
\psi (t,x, s,y) := \rho_n(t-s) 
\,\rho_m(x-y) \,\varphi(t,x).
\end{align}
Clearly $ \rho_n(t-s) \neq 0$ only 
if $s-2/n \le t\le s$ and hence $\psi(t,x, s,y)= 0$ outside  $s-2/n \le t\le s$. 
 
Moreover, let $\rho_l$ be the standard symmetric 
nonnegative mollifier on $\R$ with support in $[-l,l]$. For simplicity, we also use the generic $\eta$ for the function $\eta_\delta$.  
Applying It\^{o}'s formula to $\eta( u^\gamma_\ep(s,y)-k) \psi(t,x,s,y)$, multiplying by $\rho_l( u_\theta(t,x)-k)$, taking the expectation, and integrating with respect to $s,y,k$, an application of Fubini's theorem yields 

\begin{align}
 0\leq & 
\E\bigg[\dint_{Q_T\times\R\times \D}\eta(u^\gamma_\eps(0)-k)\psi(t,x,0,y)  \rho_l(u_\theta(t,x)-k) \,dy\,dk\,dx\,dt\bigg] \notag
\\ & 
+ \E \bigg[\dint_{\R\times Q^2_T} \eta(u^\gamma_\eps(s,y)-k) \partial_s \psi(t,x,s,y) \rho_l(u_\theta(t,x)-k)\,dy\,ds\,dk\,dx\,dt\bigg]\notag\\
&+ \E\bigg[\sum_{j \ge 1} \dint_{\R\times Q^2_T}  \eta'(u^\gamma_\eps(s,y)-k) \psi(t,x,s,y) g_j(u_\eps)^\gamma(s,y) \rho_l(u_\theta(t,x)-k)\, d\beta_j(s)\,dy \,dk\,dx\,dt\bigg] \notag
\\&
+ \frac{1}{2}\E\bigg[\dint_{\R\times Q^2_T} \eta^{\prime\prime}(u^\gamma_\eps(s,y)-k)  (\mathbb{G}(u_\eps)^\gamma(s,y))^2 \psi(t,x,s,y)\rho_l(u_\theta(t,x)-k) \,ds\,dy\,dk\,dx\,dt\bigg]\notag\\
  & - \E\bigg[\dint_{\R\times Q^2_T} F^\eta( u^\gamma_\ep(s,y),k)\cdot \nabla_y \psi(t,x,s,y)\rho_l( u_\theta(t,x)-k)\,  dy\,ds\,dk\,dx\,dt\bigg]\notag \\
&- \eps \E\bigg[\dint_{\R\times Q^2_T}\eta'(u^\gamma_\eps(s,y)-k) \nabla u^\gamma_\eps(s,y).\nabla_y\psi(t,x,s,y)\rho_l(u_\theta(t,x)-k) \,dy\,ds\,dk\,dx\,dt\bigg]\notag
\\
& -\E \bigg[\int_{\R \times Q^2_T} \mathscr{L}^r_\lambda[A(u_\eps)^\gamma(s,\cdot)](y) \psi(t,x,s,y)\, \eta^\prime( u^\gamma_\eps(s,y)-k) \rho_l(u_\theta(t,x)-k) \,dy \,ds\,dk\,dx\,dt\bigg]\notag\\
& -\E \bigg[\int_{\R \times Q_T}\int^T_0 \mathscr{L}_{\lambda,r}[A( u_\eps)^\gamma(s,\cdot)](y)  \psi(t,x,s,y) \eta^\prime( u^\gamma_\eps(s,y)-k)\rho_l(u_\theta(t,x)-k) \,ds\,dk\,dx\,dt\bigg]\notag\\
& =:  I_1 + I_2 + I_3 +I_4 + I_5 + I_6 + I_7 + I_8.
\label{kato_one_01}
\end{align}

We now write  It\^o's formula  for $\eta(u_\theta(t,x)-k)\varphi$ then multiply by $\rho_l[u^\gamma_\eps(s,y)-k]$ and  integrate with respect to $k,y,s$ and then take expectation to get
\begin{align}
 0\leq & 
\E\bigg[\dint_{Q_T\times\R\times \D}\eta(u_\theta(0)-k)\psi(0,x,s,y)  \rho_l(u^\gamma_\eps(s,y)-k) \,dy\,dk\,dx\,ds\bigg] \notag
\\ & 
+ \E \bigg[\dint_{\R\times Q^2_T} \eta(u_\theta(t,x)-k) \partial_t \psi(t,x,s,y) \rho_l(u^\gamma_\eps(s,y)-k)\,dy\,ds\,dk\,dx\,dt\bigg]\notag\\
& +\E\bigg[\sum_{j \ge1}\dint_{\R\times Q^2_T} \eta'(u_\theta(t,x)-k) \psi(t,x,s,y) g_j(u_\theta(t,x)) \, d\beta_j(t)\rho_l(u^\gamma_\eps(s,y)-k)\,dy \,dk\,dx\,ds\bigg] \notag
\\&
+ \frac{1}{2}\E\bigg[\dint_{\R\times Q^2_T} \eta^{\prime\prime}(u_\theta(t,x)-k)  \mathbb{G}^2(u_\theta(t,x)) \psi(t,x,s,y)\rho_l(u^\gamma_\eps(s,y)-k) \,ds\,dy\,dk\,dx\,dt\bigg]\notag\\
&-\E\bigg[\dint_{\R\times Q^2_T} F^\eta(u_\theta(t,x),k)\cdot \nabla_x \psi(t,x,s,y) \rho_l(u^\gamma_\eps(s,y)-k) \,dy\,ds\,dk\,dx\,dt\bigg] \notag
\\
&- \theta \E\bigg[\dint_{\R\times Q^2_T}\eta'(u_\theta(t,x)-k) \nabla_x u_\theta(t,x). \nabla_x\psi(t,x,s,y) \rho_l(u^\gamma_\eps(s,y)-k) \,dy\,ds\,dk\,dx\,dt\bigg]\notag
\\
& -\E \bigg[\int_{\R \times Q^2_T} \mathscr{L}^r_\lambda[A(u_\theta)(t,\cdot)](x) \psi(t,x,s,y)\, \eta^\prime( u_\theta(t,x)-k) \rho_l(u^\gamma_\eps(s,y)-k) \,dy \,ds\,dk\,dx\,dt\bigg]\notag\\
& -\E \bigg[\int_{\R \times Q^2_T} A^\eta_k(u_\theta(t,x))\mathscr{L}_{\lambda,r}[\psi(t,\cdot,s,y)](x) \rho_l(u^\gamma_\eps(s,y)-k) \,ds\,dy\,dk\,dx\,dt\bigg]\notag\\
& =:  J_1 + J_2 + J_3 +J_4 + J_5 + J_6 + J_7 + J_8.
\label{Kato_second_01}
\end{align}

We shall now add \eqref{kato_one_01} and \eqref{Kato_second_01} and pass to the limit with respect to the parameters in a fixed order: first $n$, then $\gamma, \delta, l ,\theta,\eps, r$ and finally $m \to \infty$. Some of the local terms can be handled as in 
\cite{BaVaWit_2012, BaVaWitParab},  therefore, we focus on providing details for the terms coming from the nonlocal part. To begin with, we recall well-known results for local terms.
\begin{lem}[\cite{BaVaWit_2012,BaVaWitParab}]\label{lem:initial+time-terms} \hfil \\ 
It holds that $I_1=0$ and since $u_0^{\eps} \underset{\eps\to0} \longrightarrow u_0$ in $L^2(\D)$ and $v_0^{\theta} \underset{\theta\to0} \longrightarrow v_0$ in $L^2(\D)$, we have
\begin{align*}
\lim_{m \goto \infty}\lim_{\theta \goto 0}\,\lim_{\eps \goto 0}\,\lim_{l \to 0}\,\lim_{\delta\goto 0} \,\lim_{\gamma\to 0}\,\lim_{n\goto \infty} \big(I_1 + J_1\big)& =   \int_{\R^d} |u_0(x)-v_0(x)| \,\varphi(0,x) \,dx.
\\
\lim_{m \goto \infty}\lim_{\theta \goto 0}\,\lim_{\eps \goto 0}\,\lim_{l \to 0}\,\lim_{\delta\goto 0} \,\lim_{\gamma\to 0}\,\lim_{n\goto \infty}  \big(I_2 + J_2\big) 
& =\E \Big[\int_{0}^T \int_{\R^d}\int_{0}^1 \int_{0}^1 
  |u(t,x,\alpha)-v(t,x,\beta)| \partial_t\varphi(t,x)
  d\alpha \,d\beta\,dx\,dt\Big].
\end{align*}
\end{lem}
The way to pass to the limits is mainly based on Lebesgue's Theorem and the properties of convolution. Concerning the proof of next lemma, this  is different from \cite{BaVaWit_2012,BaVaWitParab} and we develop the arguments of a new proof.
\begin{lem}\label{lem:stochastic-terms}
It holds that 
\begin{align*}
\lim_{\delta \to 0}\,& \lim_{\gamma\to 0}\,\lim_{n\goto \infty}  \big(I_4 + J_4 \big)\\
&=  -\frac{1}{2}\E \Big[ \int_{Q_T\times \R^d \times \R} \sgn( u_\theta(t,x)-k) \mathbb{G}^2(u_\theta(t,x)) \varphi(t,x) \rho'_l(u_\ep(t,y)-k)\rho_m(x-y) \,dy\,dk\,dx\,dt\Big]  \\
& \quad -\frac{1}{2}\E \Big[ \int_{Q_T\times \R^d \times \R} \sgn( u_\ep(t,y)-k) \mathbb{G}^2(u_\ep(t,y)) \varphi(t,x) \rho'_l(u_\theta(t,x)-k)\rho_m(x-y) \,dy\,dk\,dx\,dt\Big]  \\
& =-\frac{1}{2}\E \Big[ \int_{Q_T\times \R^d \times \R} \sgn( u_\theta(t,x)-u_\ep(t,y)+k) \mathbb{G}^2(u_\theta(t,x)) \varphi(t,x) \rho'_l(k)\rho_m(x-y) \,dy\,dk\,dx\,dt\Big]  \\
&\quad-\frac{1}{2}\E \Big[ \int_{Q_T\times \R^d \times \R} \sgn( u_\ep(t,y)-u_\theta(t,x)+k) \mathbb{G}^2(u_\ep(t,y)) \varphi(t,x) \rho'_l(k)\rho_m(x-y) \,dy\,dk\,dx\,dt\Big]\\
&=\E \Big[ \int_{Q_T\times \R^d} \mathbb{G}^2(u_\theta(t,x)) \varphi(t,x) \rho_l(u_\ep(t,y)-u_\theta(t,x)) \rho_m(x-y)\,dy\,dk\,dx\,dt\Big]  \\
& \quad+\E \Big[ \int_{Q_T\times \R^d} \mathbb{G}^2(u_\ep(t,y)) \varphi(t,x) \rho_l(u_\theta(t,x)-u_\ep(t,y)) \rho_m(x-y)\,dy\,dk\,dx\,dt\Big].
\end{align*}
Moreover, $I_3$=0 and 
\begin{align*}
&\lim_{\delta \to 0} \lim_{\gamma \to 0}\lim_{n \to \infty}J_3
\\ =
& \E\bigg[\sum_{j \ge1}\dint_{\R\times Q_T\times \R^d}\hspace{-1cm} \sgn(u_\theta(t,x)-u_\ep(t,y))+k) \varphi(t,x) g_j(u_\theta(t,x)) g_j(u_\ep(t,y)) 
\times \rho'_l(k)\rho_m(x-y)\,dy \,dk\,dx\,dt\bigg]
\\
=&-2 \E \bigg[\sum_{j \ge1}\dint_{Q_T\times \R^d}g_j(u_\theta(t,x)) g_j(u_\ep(t,y))\varphi(t,x)\rho_m(x-y)\rho_l(u_\ep(t,y)-u_\theta(t,x))\,dy\,dx\,dt\bigg]
\end{align*}
Finally, it follows that
$$\lim_{l \to 0}\lim_{\delta \to 0} \lim_{ \gamma \to 0} \lim_{ n \to \infty}(I_3+I_4+J_3+J_4)=0.$$
\end{lem}
\begin{proof}
Inspired by the proof of \cite[Sec. 4]{BaVaWit_2012}, we would like to pass the limits in various parameters involved. However, contrary to arguments depicted in \cite{BaVaWit_2012}, we would like to pass to the limit with respect to the parameter $\delta$ before passing to the limits in the parameter $l$. To begin with, an argument based on Lebesgue's theorem and classical convolution properties reveals that
\begin{align*}
\lim_{\gamma \to 0} \lim_{n \to \infty} I_4 &=\frac{1}{2}\E \bigg[ \int_{Q_T\times \R^d \times \R} \eta''( u_\ep(t,y)-k) \mathbb{G}^2(u_\ep(t,y)) \varphi(t,x) \rho_l(u_\theta(t,x)-k)\rho_m(x-y) \,dy\,dk\,dx\,dt\Big].  
\end{align*}
Applying an integration by parts formula on the variable $k$ and then using a change of variables we get 
\begin{align*}
\lim_{\gamma \to 0} \lim_{n \to \infty} I_4 &=-\frac{1}{2}\E \bigg[ \int_{Q_T\times \R^d \times \R} \eta'( u_\ep(t,y)-u_\theta(t,x)+k) \mathbb{G}^2(u_\ep(t,y)) \varphi(t,x) \rho'_l(k)\rho_m(x-y) \,dy\,dk\,dx\,dt\Big].  
\end{align*}
We shall now consider the passage to the limit as $\delta$ goes to 0.\\
First observe that $\mathbb{G}^2(u_\ep(t,y)) \varphi(t,x) \rho'_l(k)\rho_m(x-y) \in L^1( \Omega \times Q_T \times \R^d \times \R)$. Indeed we have
\begin{align*}
&\E\Big[ \int_{Q_T \times \R^d \times \R} \mathbb{G}^2(u_\ep(t,y)) \varphi(t,x) |\rho'_l(k)| \rho_m(x-y) \,dy\,dk\,dx\,dt\Big] \\
&\hspace{4cm} \le C(\varphi)\E\Big[ \int_{Q_T  \times \R}  |u_\ep(t,y)|^2 |\rho'_l(k)| \,dk\,dt\Big]< +\infty.
\end{align*}
We can apply dominated convergence theorem, using the fact that $|\eta'| \le 1$, to conclude
\begin{align*}
\lim_{\delta \to 0}\,\lim_{\gamma\to 0}\,\lim_{n\goto \infty}  I_4 = & -\frac{1}{2}\E \Big[ \int_{Q_T\times \R^d \times \R} \hspace{-0.2cm}\sgn( u_\ep(t,y)-u_\theta(t,x)+k) \mathbb{G}^2(u_\ep(t,y)) \varphi(t,x) \rho'_l(k)\rho_m(x-y) \,dy\,dk\,dx\,dt\Big]
\\ = & \E \Big[ \int_{Q_T\times \R^d } \hspace{-0.2cm} \mathbb{G}^2(u_\ep(t,y)) \varphi(t,x) \rho_l(u_\theta(t,x)-u_\ep(t,y))\rho_m(x-y) \,dy\,dx\,dt\Big]
\end{align*}
by noticing that
\begin{align*}
\int_\R \sgn( u_\ep(t,y)-u_\theta(t,x)+k) \rho'_l(k)\,dk= -2\rho_l(u_\theta(t,x)-u_\ep(t,y)).
\end{align*}
A similar argument yields
\begin{align*}
\lim_{\delta \to 0} \,\lim_{\gamma \to 0}\, \lim_{n \to \infty} J_4=&-\frac{1}{2}\E \Big[ \int_{Q_T\times \R^d \times \R} \hspace{-0.2cm} \sgn(u_\theta(t,x)-u_\ep(t,y)+k) \mathbb{G}^2(u_\theta(t,x)) \varphi(t,x) \rho'_l(k)\rho_m(x-y) \,dy\,dk\,dx\,dt\Big]
\\ =&\E \Big[ \int_{Q_T\times \R^d } \hspace{-0.2cm}  \mathbb{G}^2(u_\theta(t,x)) \varphi(t,x) \rho_l (u_\ep(t,y)-u_\theta(t,x))\rho_m(x-y) \,dy\,dk\,dx\,dt\Big].
\end{align*}
By an argument of conditional independence,  $I_3=0$. For $J_3$ we adapt for the first steps the technique used in \cite[Sec. 4]{BaVaWit_2012} where a simple application of It\^o's formula reveals that
\begin{align*}
&\rho_l(u^\gamma_\ep(s,y)-k) -\rho_l(u^\gamma_\ep(s-2/n,y)-k)\\
&\qquad = \int^s_{s-2/n} \rho'_l( u^\gamma_\ep (\sigma,y)-k) A_\ep(\sigma ,y) \,d\sigma+\sum_{j \ge 1}
\int^s_{s-2/n} \rho'_l( u^\gamma_\ep(\sigma,y)-k)g_j(u_\ep(\sigma,y)) \,d\beta_{j}(\sigma)\\
&\qquad \qquad + \frac{1}{2}\int^s_{s-2/n} \rho''_l( u^\gamma_\ep(\sigma,y)-k) \mathbb{G}^2(u^\gamma_\ep(\sigma,y)) \, d\sigma\\
&\qquad = -\frac{\partial}{\partial k}\bigg[\int^s_{s-2/n} \rho_l( u^\gamma_\ep (\sigma,y)-k) A_\ep(\sigma ,y)+\frac12 \rho'_l( u^\gamma_\ep(\sigma,y)-k) \mathbb{G}^2(u^\gamma_\ep(\sigma,y)) \, d\sigma\bigg]\\ 
&\qquad \qquad  +\sum_{j \ge 1}\int^s_{s-2/n} \rho'_l( u^\gamma_\ep(\sigma,y)-k)g_j(u_\ep(\sigma,y)) \,d\beta_{j}(\sigma)
\end{align*}
where $A_\epsilon=\eps \Delta u^\gamma_\eps+ \Div [f(u_\ep)]^\gamma -  \mathscr{L}_{\lambda}\Big([A(u_\eps))]^\gamma\Big)$ is square integrable. Hence, we recast the term $J_3$ as
\begin{align*}
J_3=& \E\bigg[\sum_{j \ge1}\dint_{\R\times Q^2_T} \eta'(u_\theta(t,x)-k) \psi(t,x,s,y) g_j(u_\theta(t,x)) \, d\beta_j(t)\\
& \hspace{2cm}\times (\rho_l(u^\gamma_\eps(s,y)-k)-\rho_l(u^\gamma_\eps(s-2/n,y)-k))\,dy \,dk\,dx\,ds\bigg] \\
&=\E \bigg[\sum_{j \ge1}\dint_{\R\times Q^2_T}\eta'(u_\theta(t,x)-k) \psi(t,x,s,y) g_j(u_\theta(t,x)) \, d\beta_j(t)\\
& \hspace{1cm} \times \frac{\partial}{\partial k}\bigg(\int^s_{s-2/n} \rho_l( u^\gamma_\ep (\sigma,y)-k) A_\ep(\sigma ,y)+1/2 \rho'_l( u^\gamma_\ep(\sigma,y)-k) \mathbb{G}^2(u^\gamma_\ep(\sigma,y)) \, d\sigma\bigg) \,dy\,dk\,dx\,ds\bigg]\\
&\quad +\E \bigg[\int_{\R \times Q_T^2} \bigg(\sum_{ j \ge 1}\eta'(u_\theta(t,x)-k) \psi(t,x,s,y) g_j(u_\theta(t,x)) \, d\beta_j(t)\bigg)\\
&\hspace{3cm} \times \bigg( \sum_{j \ge 1}\int^s_{s-2/n} \rho'_l( u^\gamma_\ep(\sigma,y)-k)g_j(u_\ep(\sigma,y)) \,d\beta_{j}(\sigma)\bigg)\,dy\,dk\,dx\,ds\bigg].
\end{align*}
The first term in the above expression goes to $0$ as $n$ goes to infinity (arguments similar to  \cite[Sec 4, Pg 689-690]{BaVaWit_2012}), and it remains now to consider the second one. In fact, using the fact that $\{e_j\}_{j\ge 1}$ is a complete orthonormal system and the product rule of It\^{o}'s integral, we get
\begin{align*}
\lim_{\gamma \to 0}& \lim_{n \to \infty}J_3\\
&= \E\bigg[\sum_{j \ge1}\dint_{\R\times Q_T\times \R^d} \eta'(u_\theta(t,x)-k) \varphi(t,x) g_j(u_\theta(t,x)) g_j(u_\ep(t,y)) \rho'_l(u_\eps(t,y)-k)\rho_m(x-y)\,dy \,dk\,dx\,dt\bigg]\\
&=\E\bigg[\sum_{j \ge1}\dint_{\R\times Q_T\times \R^d} \eta'(u_\theta(t,x)-u_\ep(t,y)+k) \varphi(t,x) g_j(u_\theta(t,x)) g_j(u_\ep(t,y)) \rho'_l(k)\rho_m(x-y)\,dy \,dk\,dx\,dt\bigg].
\end{align*}
We shall now show the passage to the limit as $\delta$ goes to 0. Again notice that $$\sum_{j \ge 1} g_j(u_\theta(t,x))g_j(u_\ep(t,y)) \varphi(t,x) \rho'_l(k) \rho_m(x-y) \in L^1( Q_T \times \R^d \times \R). $$
Indeed, we have
\begin{align*}
&\E \bigg[\int_{Q_T \times \R^d \times \R} \big| \sum_{j \ge 1} g_j(u_\theta(t,x))g_j(u_\ep(t,y))\big| \varphi(t,x) |\rho'_l(k)| \rho_m(x-y) \,dy\,dx\,dk\,dt \bigg] \\
& \hspace{1cm}\le \frac{1}{2} \E \bigg[ \int_{Q_T \times \R^d \times \R} \sum_{j \ge 1} |g_j(u_\theta(t,x))|^2\varphi(t,x) |\rho'_l(k)| \rho_m(x-y)\,dy\,dx\,dk\,dt \bigg] \\
& \hspace{1cm} + \frac{1}{2} \E \bigg[ \int_{Q_T \times \R^d \times \R} \sum_{k \ge 1} |g_j(u_\ep(t,y))|^2\varphi(t,x) |\rho'_l(k)| \rho_m(x-y)\,dy\,dx\,dk\,dt \bigg] \\
&\hspace{1cm}\le   C(\| \varphi\|_\infty,\rho'_l)\E \bigg[\int_{Q_T} | u_\theta(t,x)|^2 \,dx\,dt+ \int_{Q_T } |u_\ep(t,y)|^2\,dy\,dt\bigg] < +\infty.
\end{align*}
Thus using the fact that $|\eta'| \le 1$, and applying dominated convergence theorem we get
\begin{align*}
&\lim_{\delta \to 0} \lim_{\gamma \to 0}\lim_{n \to \infty}J_3
\\=& \E\bigg[\sum_{j \ge1}\dint_{\R\times Q_T\times \R^d} \hspace{-1cm} \sgn(u_\theta(t,x)-u_\ep(t,y)+k) \varphi(t,x) g_j(u_\theta(t,x)) g_j(u_\ep(t,y))  \rho'_l(k)\rho_m(x-y)\,dy \,dk\,dx\,dt\bigg]
\\=& -2\E\bigg[\sum_{j \ge1}\dint_{Q_T\times \R^d} \hspace{-1cm}  \varphi(t,x) g_j(u_\theta(t,x)) g_j(u_\ep(t,y))  \rho_l(u_\ep(t,x)-u_\theta(t,y))\rho_m(x-y)\,dy \,dk\,dx\,dt\bigg],
\end{align*}
and this implies, by symmetry  of $\rho_l$,
\begin{align*}
&\lim_{\delta \to 0} \lim_{ \gamma \to 0} \lim_{ n \to \infty}(I_4+J_3+J_4)
\\ =&  \E \bigg[\sum_{j \ge1}\dint_{Q_T\times \R^d} ( g_j(u_\theta(t,x)-g_j(u_\ep(t,y))^2 \varphi(t,x) \rho_m(x-y)  \rho_l(u_\ep(t,y)- u_\theta(t,x))\,dy\,dx\,dt\bigg].
\end{align*}
We now use the fact that the support of $\rho_l$ lies in $[-l,l],\,\rho_l = \mathcal{O}(1/l)$, and Assumption~\ref{A4} to conclude that 
$$  \sum_j |g_j(u_\theta(t,x))-g_j(u_\ep(t,y))|^2 \varphi(t,x) \rho_m(x-y)
\rho_l(u_\ep(t,y)- u_\theta(t,x)) \le C\,l\varphi(t,x) \rho_m(x-y),$$
and thus we get
$$\lim_{l \to 0}\lim_{\delta \to 0} \lim_{ \gamma \to 0} \lim_{ n \to \infty}(I_3+I_4+J_3+J_4)=0.$$
\end{proof}
\begin{lem}[\cite{BaVaWit_2012,BaVaWitParab}]\label{lem:flux-terms}
The followings hold:
\begin{equation*}
\begin{aligned}
& \lim_{m \goto \infty} \lim_{\theta \goto 0}\,\lim_{\eps \goto 0}\,\lim_{l \to 0}\,\lim_{\delta\goto 0} \,\lim_{\gamma\to 0}\,\lim_{n\goto \infty}  \big(I_5 + J_5 \big)  
= -\E \Big[\int_{Q_T} \int_{0}^1 \int_{0}^1 
F(u(t,x,\alpha),v(t,x,\beta)) \cdot \grad_x \varphi (t,x)
 \,d\alpha\,d\beta\,dx\,dt\Big]. \\[2mm]  
& \lim_{m \goto \infty} \lim_{\theta \goto 0}\,\lim_{\eps \goto 0}\,\lim_{l \to 0}\,\lim_{\delta \goto 0} \,\lim_{\gamma\to 0}\,\lim_{n\goto \infty}  (I_6 +J_6) =0. 
\end{aligned}
\end{equation*}
\end{lem}
We now add terms coming from the fractional operator, and compute the limits with respect to the various parameters involved.
\begin{lem}
\label{kato_lemma1}
It holds that
\begin{align*}
I_7+J_7 & \underset{n \goto \infty} \longrightarrow   -\E \bigg[\int_{Q_T\times \R \times \R^d} \mathscr{L}^r_\lambda[A(u_\ep)^\gamma(t,\cdot)](y) \varphi(t,x)\rho_m(x-y)\, \eta^\prime( u^\gamma_\eps(t,y)-k) \rho_l(u_\theta(t,x)-k) \,dy \,dk\,dx\,dt\bigg]\\
&\hspace{1cm}-\E \bigg[\int_{Q_T\times \R \times \R^d} \mathscr{L}^r_\lambda[A(u_\theta(t,\cdot))](x) \varphi(t,x)\rho_m(x-y)\, \eta^\prime( u_\theta(t,x)-k) \rho_l(u^\gamma_\ep(t,y)-k) \,dy \,dk\,dx\,dt\bigg]\\
& \underset{\gamma \goto 0} \longrightarrow -\E \bigg[\int_{Q_T\times \R \times \R^d} \mathscr{L}^r_\lambda[A(u_\eps(t,\cdot))](y) \varphi(t,x)\rho_m(x-y)\, \eta^\prime( u_\eps(t,y)-k) \rho_l(u_\theta(t,x)-k) \,dy \,dk\,dx\,dt\bigg]\\
&\hspace{1cm}-\E \bigg[\int_{Q_T\times \R \times \R^d} \mathscr{L}^r_\lambda[A(u_\theta(t,\cdot))](x) \varphi(t,x)\rho_m(x-y)\, \eta^\prime( u_\theta(t,x)-k) \rho_l(u_\ep(t,y)-k) \,dy \,dk\,dx\,dt\bigg]\\
& \underset{\delta \goto 0} \longrightarrow - \E\bigg[ \int_{Q_T \times \R^d \times \R} \mathscr{L}^r_\lambda[A(u_\ep(t,\cdot))](y) \varphi(t,x)\rho_m(x-y)\, \sgn ( u_\ep(t,y) - k) \rho_l(u_\theta(t,x)-k) \,dy \,dk \,dx\,dt\bigg]\\
&\hspace{1cm}-\E\bigg[ \int_{Q_T \times \R^d \times \R} \mathscr{L}^r_\lambda[A(u_\theta(t,\cdot))](x) \varphi(t,x)\rho_m(x-y)\, \sgn( u_\theta(t,x)- k)\rho_l(u_\ep(t,y)-k)  \,dy\,dk \,dx\,dt\bigg]\\
& \underset{ l \to 0}\longrightarrow - \E \bigg[\int_{Q_T \times \R^d} \mathscr{L}^r_\lambda[A(u_\ep(t,\cdot))](y) \varphi(t,x)\rho_m(x-y)\, \sgn( u_\ep(t,y) - u_\theta(t,x)) \,dy \,dx\,dt\bigg]\\
&\hspace{1cm}-\E \bigg[\int_{Q_T \times \R^d} \mathscr{L}^r_\lambda[A(u_\theta(t,\cdot))](x) \varphi(t,x)\rho_m(x-y)\, \sgn( u_\theta(t,x)- u_\ep(t,y))  \,dy \,dx\,dt\bigg].
\end{align*}
Moreover, we have
\begin{align*}
&\limsup_{\theta \to 0}\,\limsup_{\ep \to 0}\, \lim_{l \to 0}\,\lim_{\delta\to 0}\,\lim_{\gamma \to 0} \, \lim_{n \to \infty} (I_7+J_7) \\
& \qquad \le -\E \bigg[\int_{Q_T \times \R^d\times (0,1)^2} | A(u(t,x,\alpha)) - A( v(t,y,\beta)) |\mathscr{L}^r_\lambda[ \varphi(t,\cdot)](x) \rho_m(x-y) \,d\alpha\,d\beta\,dx\,dy\,dt\bigg].
\end{align*}
\end{lem}

\begin{proof} We shall prove the lemma with the regular part of the non-local operator in several steps.\\
\noindent {\bf Step 1} (Passing to the limit as $ n \to \infty $): As $\varphi$ is compactly supported in space and $u^{\gamma}_\ep,A(u_\ep)^{\gamma} \in L^2(\Omega \times Q_T)$, a simple application of dominated convergence theorem reveals that
\begin{align*}
&-\E \bigg[\int_{\R \times Q^2_T} \mathscr{L}^r_\lambda[A(u_\ep)^\gamma(s,\cdot)](y) \psi(t,x,s,y)\, \eta^\prime( u^\gamma_\ep(s,y)-k) \rho_l(u_\theta(t,x)-k) \,dy \,ds\,dk\,dx\,dt\bigg]\\
& \underset{n \goto \infty} \longrightarrow -\E \bigg[\int_{Q_T\times \R \times \R^d} \mathscr{L}^r_\lambda[A(u_\ep)^\gamma(t,\cdot)](y) \varphi(t,x)\rho_m(x-y)\, \eta^\prime( u^\gamma_\ep(t,y)-k) \rho_l(u_\theta(t,x)-k) \,dy \,dk\,dx\,dt\bigg],
\end{align*}
and similarly, $A(u_\theta)\in L^2(\Omega \times Q_T)$ and 
\begin{align*}
&-\E\bigg[ \int_{\R \times Q^2_T} \mathscr{L}^r_\lambda[A(u_\theta)(t,\cdot)](x) \psi(t,x,s,y)\, \eta^\prime( u_\theta(t,x)-k) \rho_l(u^\gamma_\ep(s,y)-k) \,dy \,ds\,dk\,dx\,dt\bigg]\\
& \underset{n \goto \infty} \longrightarrow -\E \bigg[\int_{Q_T\times \R \times \R^d} \mathscr{L}^r_\lambda[A(u_\theta(t,\cdot)](x) \varphi(t,x)\rho_m(x-y)\, \eta^\prime( u_\theta(t,x)-k) \rho_l(u^\gamma_\ep(t,y)-k) \,dy \,dk\,dx\,dt\bigg].
\end{align*}

\noindent {\bf Step 2} (Passing to the limit as $\gamma \goto 0$): 
Let us consider the first term 
\begin{align*}
-\E \bigg[\int_{Q_T\times \R \times \R^d} \fr^r[A(u_\ep)^\gamma(t,\cdot)](y) \varphi(t,x)\rho_m(x-y)\, \eta^\prime( u^\gamma_\ep(t,y)-k) \rho_l(u_\theta(t,x)-k) \,dy \,dk\,dx\,dt\bigg].
\end{align*}
{\bf Claim:} The following convergences hold in $ L^2( \Omega \times Q_T \times \R^d \times \R)$
\begin{align*}
(a)\qquad & \fr^r[A(u_\ep)^\gamma(t,\cdot)](y) \varphi(t,x)\sqrt{\rho_m(x-y)} \sqrt{ \rho_l( u_\theta(t,x)-k)}\\
& \qquad \qquad \qquad \qquad \underset{\gamma \to 0} \longrightarrow \fr^r[A(u_\ep)(t,\cdot)](y) \varphi(t,x)\sqrt{\rho_m(x-y)} \sqrt{ \rho_l( u_\theta(t,x)-k)} \\
(b) \qquad & \eta'(u^\gamma_\ep(t,y)-k) \sqrt{\rho_m(x-y)} \sqrt{\rho_l(u_\theta(t,x)-k)}\\
& \qquad \qquad \qquad \qquad \underset{\gamma \to 0} \longrightarrow \eta' ( u_\ep(t,y)-k) \sqrt{\rho_m(x-y)} \sqrt{\rho_l(u_\theta(t,x)-k)} 
\end{align*}
\begin{proof}(of the claim)
To prove the first one, let us consider 
\begin{align*}
 &\E \bigg[\int_{ Q_T \times \R^d \times \R} | \fr^r[A(u_\ep)^\gamma(t,\cdot)](y)- \fr^r[ A(u_\ep)(t,\cdot)](y) |^2 \rho_m(x-y) \varphi^2(t,x) \rho_l(u_\theta(t,x)-k) \,dk\,dy\,dx\,dt\bigg]\\
 & =\E \bigg[\int_{ Q_T \times \R^d \times \R} \Bigg| \int_{|z| \ge r} \frac{A(u_\ep)^\gamma(t,y)-A(u_\ep)^\gamma(t,y+z)}{|z|^{d+2\lambda}}\,dz- \int_{|z| \ge r} \frac{A(u_\ep)(t,y)-A(u_\ep)(t,y+z)}{|z|^{d+2\lambda}}\Bigg|^2 \\
 & \hspace{5cm} \times \rho_m(x-y) \varphi^2(t,x) \rho_l(u_\theta(t,x)-k)\,dk\,dy\,dx\,dt\bigg]\\
 &= \E \bigg[\int_{Q_T \times \R^d\times \R} \Bigg| \Bigg( \int_{|z| \ge r}\frac{dz}{|z|^{d+2\lambda}}\Bigg) \Big[ A(u_\ep)^\gamma(t,y)- A(u_\ep)(t,y)\Big] - \int_{|z| \ge r} \hspace{-0.3cm} \frac{A(u_\ep)^\gamma(t,y+z)-A(u_\ep)(t,y+z)}{|z|^{d+2\lambda}}\,dz\Bigg|^2\\
 &\hspace{5cm} \times \rho_m(x-y) \varphi^2(t,x) \rho_l(u_\theta(t,x)-k)\,dk\,dy\,dx\,dt\bigg]\\
 & \le c(r) \Bigg( \E \bigg[\int_{ Q_T \times \R^d\times \R}  | A(u_\ep)^\gamma(t,y)- A(u_\ep)(t,y)|^2\rho_m(x-y) \varphi^2(t,x) \rho_l(u_\theta(t,x)-k)\,dk\,dy\,dx\,dt\bigg]\\
 & \quad +\E \bigg[\int_{Q_T \times \R^d\times \R \times \{|z| \ge r\}} \hspace{-0.3cm} \frac{ |A(u_\ep)^\gamma(t,y+z)-A(u_\ep)(t,y+z)|^2}{|z|^{d+2\lambda}} \rho_m(x-y) \varphi^2(t,x) \rho_l(u_\theta(t,x)-k)\,dz\,dk\,dy\,dx\,dt\bigg]\Bigg) \\
&\le c(r,\varphi)\Bigg( \E \bigg[\int_{ Q_T }  | A(u_\ep)^\gamma(t,y)- A(u_\ep)(t,y)|^2 \,dy\,dt\bigg]\\
&\hspace{5cm}+\E \bigg[\int_{Q \times  \{|z| \ge r\}} \frac{|A(u_\ep)^\gamma(t,y+z)-A(u_\ep)(t,y+z)|^2}{|z|^{d+2\lambda}} \,dz\,dy\,dt\bigg]\Bigg)\\
&\le c(r,\varphi)\Bigg( \E\bigg[ \int_{ Q_T}  | A(u_\ep)^\gamma(t,y)- A(u_\ep)(t,y)|^2 \,dy\,dt\bigg]\\
&\hspace{5cm}+ \int_{\{|z| \ge r\}} \frac{1}{|z|^{d+2\lambda}}\,dz\,\| A(u_\ep)^\gamma-A(u_\ep)\|^2_{L^2(\Omega \times Q_T)} \Bigg) \underset{\gamma \to 0} \longrightarrow 0.
\end{align*}
Let us now consider the second part
\begin{align*}
 &\E \bigg[\int_{Q_T \times \R^d\times \R}|\eta'(u^\gamma_\ep(t,y)-k)- 
\eta' ( u_\ep(t,y)-k)|^2 \rho_m(x-y) \rho_l(u_\theta(t,x)-k)\,dk\,dy\,dx\,dt\bigg]\\
&\qquad \qquad \le C(\eta') \E \bigg[ \int_{Q_T}|u^\gamma_\ep(t,y)-  u_\ep(t,y)|^2\,dy\,dt\bigg] \underset{ \gamma \to 0} \longrightarrow 0.
\end{align*}
\end{proof}
In view of the above claims, we get the desired result. Next, to pass to the limit in the second term we consider
\begin{align*}
&\bigg| \E \bigg[\int_{Q_T\times \R \times \R^d} \mathscr{L}^r_\lambda[A(u_\theta(t,\cdot)](x) \varphi(t,x)\rho_m(x-y)\, \eta^\prime( u_\theta(t,x)-k) \rho_l(u^\gamma_\ep(t,y)-k) \,dy \,dk\,dx\,dt\bigg]\\
&\qquad -\E\bigg[ \int_{Q_T \times \R \times \R^d} \mathscr{L}^r_\lambda[A(u_\theta(t,\cdot)](x) \varphi(t,x)\rho_m(x-y)\, \eta^\prime( u_\theta(t,x)-k) \rho_l(u_\ep(t,y)-k) \,dy \,dk\,dx\,dt\bigg] \bigg|\\
&=\bigg| \E \bigg[\int_{Q_T \times \R \times \R^d} \mathscr{L}^r_\lambda[A(u_\theta(t,\cdot)](x) \varphi(t,x)\rho_m(x-y)\, \eta^\prime( u_\theta(t,x)-u^\gamma_\ep(t,y)+k)\rho_l(k) \,dk\,dy \,dx\,dt\bigg]\\
&\qquad -\E \bigg[\int_{Q_T\times \R \times \R^2} \mathscr{L}^r_\lambda[A(u_\theta(t,\cdot)](x) \varphi(t,x)\rho_m(x-y)\, \eta^\prime( u_\theta(t,x)-u_\ep(s,y)+k)\rho_l(k)\,dk \,dy \,dx\,dt\bigg] \bigg|\\
&\le C(\eta') \E \bigg[\int_{Q_T\times \R \times \R^d} |u^\gamma_\ep(t,y)- u_\ep(t,y)||\mathscr{L}^r_\lambda[A(u_\theta(t,\cdot)](x)| \varphi(t,x)\rho_m(x-y) \rho_l(k) \,dk\,dy\,dx\,dt\bigg]\\
&\le C(\eta',\varphi)\bigg(\E \bigg[\int_{Q_T\times \R^d} |u^\gamma_\ep(t,y)- u_\ep(t,y)|^2\rho_m(x-y) \,dy\,dx\,dt\bigg]\bigg)^{1/2}\\
& \hspace{3cm}\times \Big( \E \bigg[\int_{Q_T\times \R^d} |\mathscr{L}^r_\lambda[A(u_\theta(t,\cdot)](x)|^2 \rho_m(x-y) \,dy\,dx\,dt\bigg]\bigg)^{1/2}\quad \underset{\gamma \to 0} \longrightarrow 0.
\end{align*}

\noindent {\bf Step 3} (Passing to the limit as $\delta \to 0$):
First we shall consider the term $$ \E \bigg[\int_{Q_T \times \R^d \times \R} \mathscr{L}^r_\lambda [A(u_\ep(t,\cdot))](y) \eta' (u_\ep (t,y)- k) \varphi(t,x) \rho_m(x-y) \rho_l(u_\theta(t,x)-k) \,dk\,dy\,dx\,dt\bigg].$$
Observe that $\mathscr{L}^r_\lambda [A(u_\ep(t,\cdot))](y) \varphi(t,x) \rho_m(x-y) \rho_l(u_\theta(t,x)-k)\in L^1( \Omega \times Q_T \times \R^d\times \R)$. Indeed, since $\varphi$ has a compact support in space, we see
\begin{align*}
& \E\bigg[ \int_{Q_T \times \R^d \times \R }\big|\mathscr{L}^r_\lambda [A(u_\ep(t,\cdot))](y) \varphi(t,x) \rho_m(x-y)\rho_l(u_\theta(t,x)-k)\big|\,dy\,dx\,dt\bigg]\\
&\le \E \bigg[\int_{Q_T \times \R^d \times \{|z| \ge r\}} \frac{|A(u_\ep(t,y))|}{|z|^{d+2 \lambda}}\varphi(t,x) \rho_m(x-y) \,dz\,dy\,dx\,dt\bigg]\\
& \hspace{3cm} + \E \bigg[\int_{Q_T \times \R^d \times \{|z| \ge r\}}  \frac{|A(u_\ep(t,y+z))|}{|z|^{d+2 \lambda}}\varphi(t,x) \rho_m(x-y) \,dy\,dx\,dt\bigg]\\
& \le c(r) \E \bigg[\int_{Q_T \times \R^d} | A(u_\ep(t,y))| \varphi(t,x) \rho_m(x-y)\,dz\,dy\,dx\,dt\bigg]\\
& \hspace{3cm}+ d(r) \E \bigg[\int_{Q_T\times \R^d} \| A(u_\ep(t,\cdot))\|_{L^2(\R^d)} \varphi(t,x) \rho_m(x-y) \,dy\,dx\,dt\bigg]\\
& \le c(r) \Bigg( \E \bigg[\int_{\D \times Q_T} |u_\ep(t,y)|^2 \varphi(t,x)\rho_m(x-y)\,dt\,dy\,dx\bigg]\Bigg)^{1/2}\Bigg( \E \bigg[\int_{\D \times Q_T} \varphi(t,x)\rho_m(x-y)\,dt\,dy\,dx\bigg]\Bigg)^{1/2}\\
&+d(r) \Bigg( \E \bigg[\int_{\D \times Q_T} \hspace{-0.3cm} \| u_\ep(t,\cdot)\|^2_{L^2(\R^d)}\varphi(t,x) \rho_m(x-y) \,dy \,dt \,dx\bigg] \Bigg)^{1/2}\Bigg( \E \bigg[\int_{\D \times Q_T} \hspace{-0.3cm}\varphi(t,x)\rho_m(x-y) \,dy\,dt\,dx\bigg]\Bigg)^{1/2} \,\, \\
&< +\infty.
\end{align*}
Thus, using the fact that $| \eta'| \le 1$ and using Lebesgue's dominated convergence theorem, we get 
\begin{align*}
&-\E \bigg[\int_{Q_T \times \R^d \times \R} \mathscr{L}^r_\lambda [A(u_\ep(t,\cdot))](y) \eta' (u_\ep (t,y)- k) \varphi(t,x) \rho_m(x-y) \rho_l(u_\theta(t,x)-k)\,dy\,dk\,dx\,dt\bigg]\\
& \qquad \underset{ \delta \to 0} \longrightarrow -\E \bigg[\int_{Q_T \times \R^d \times \R} \mathscr{L}^r_\lambda [A(u_\ep(t,\cdot))](y) \sgn(u_\ep (t,y)- k) \varphi(t,x) \rho_l(u_\theta(t,x)-k)\rho_m(x-y) \,dy\,dk\,dx\,dt\bigg].
\end{align*}
Similarly, we can conclude 
\begin{align*}
&-\E \bigg[\int_{Q_T \times \R^d \times \R} \mathscr{L}^r_\lambda[A(u_\theta(t,\cdot)](x) \varphi(t,x)\rho_m(x-y)\, \eta^\prime( u_\theta(t,x)- k)\rho_l(u_\ep(t,y)-k)  \,dy\,dk \,dx\,dt\bigg]\\
& \qquad \underset{\delta \to 0} \longrightarrow -\E\bigg[ \int_{Q_T \times \R^d \times \R} \mathscr{L}^r_\lambda[A(u_\theta(t,\cdot)](x) \varphi(t,x)\rho_m(x-y)\, \sgn( u_\theta(t,x)- k)\rho_l(u_\ep(t,y)-k)  \,dy\,dk \,dx\,dt\bigg].
\end{align*}

\noindent {\bf Step 4} (Passing to the limit as $l \to 0$): 
Observe that after a change of variable and using the convolution in $L^1( \R ; L^1( \Omega \times Q_T \times \R^d))$  we have 
\begin{align*}
& -\E \bigg[\int_{Q_T\times \R \times \R^d} \mathscr{L}^r_\lambda[A(u_\ep)(t,\cdot)](y) \varphi(t,x)\rho_m(x-y)\, \sgn( u_\ep(t,y)-k) \rho_l(u_\theta(t,x)-k)\,dy\,dk\,dx\,dt\bigg]\\
&= -\E \bigg[\int_{Q_T \times \R^d} \mathscr{L}^r_\lambda[A(u_\ep)(t,\cdot)](y) \varphi(t,x)\rho_m(x-y)\, \bigg(\int_\R \sgn( u_\ep(t,y)-u_\theta(t,x)-k) \rho_l(k)\,dk\bigg)\,dy\,dx\,dt\bigg].\\
\end{align*}
Let $f_l(x,y,t)= \int_\R \sgn( u_\ep(t,y)-u_\theta(t,x)-k) \rho_l(k)\,dk$, then it is easy to see that $f_l(x,y,t) \overset {l \to 0} \longrightarrow \sgn( u_\ep(t,y)-u_\theta(t,x))$. Moreover, note that $|f_l|\le 1$ and $\mathscr{L}^r_\lambda[A(u_\ep(t,\cdot)](y) \varphi(t,x) \rho_m(x-y) \in L^1(\Omega \times Q_T)$(similar to Step 3 above). This allows us to use dominated convergence Theorem to get the desired result.

Similarly we conclude that 
\begin{align*}
& -\E \bigg[\int_{Q_T\times \R \times \R^d} \mathscr{L}^r_\lambda[A(u_\theta)(t,\cdot)](x) \varphi(t,x) \rho_m(x-y)\, \sgn( u_\theta(t,x)-k) \rho_l(u_\ep(t,y)-k)\,dy\,dk\,dx\,dt\bigg]\\
&\underset{ l \to 0} \longrightarrow -\E \bigg[\int_{Q_T \times \R^d} \mathscr{L}^r_\lambda[A(u_\theta)(t,\cdot)](x) \varphi(t,x)\rho_m(x-y)\, \sgn( u_\theta(t,y)-u_\ep(t,y))\,dy\,dx\,dt\bigg].\\
\end{align*}

\noindent {\bf Step 5} (Passing to the limit as $\ep,\theta \to 0$): Following \cite[Sec 4]{Alibaud}, we observe that 
\begin{align*}
& -\E \bigg[\int_{Q_T \times \R^d} \Big[\mathscr{L}^r_\lambda[A(u_\ep(t,\cdot)](y)-\mathscr{L}^r_\lambda[A(u_\theta(t,\cdot)](x)\Big] \varphi(t,x)\rho_m(x-y)\, \sgn(u_\ep(t,y)- u_\theta(t,x))  \,dy \,dx\,dt\bigg]\\
&= - \E \bigg[\int_{Q_T \times \R^d \times \{|z| \ge r\}} \Bigg(\frac{A(u_\ep(t,y))-A(u_\ep(t,y+z))}{|z|^{d+2\lambda}}-\frac{A(u_\theta(t,x))-A(u_\theta(t,x+z))}{|z|^{d+2\lambda}}\Bigg)\\
& \hspace{7cm}\times \sgn(u_\ep(t,y)-u_\theta(t,x)) \varphi(t,x) \rho_m(x-y) \,dz\,dy\,dx\,dt\bigg]\\
&= - \E \bigg[\int_{Q_T \times \R^d \times \{ |z| \ge r\}} \Bigg(\frac{[A(u_\ep(t,y))-A( u_\theta(t,x))]-[ A(u_\ep(t,y+z))- A(u_\theta(t,x+z))]}{|z|^{d+2\lambda}}\Bigg)\\
&\hspace{7cm} \times \sgn(u_\ep(t,y)-u_\theta(t,x)) \varphi(t,x) \rho_m(x-y) \,dz\,dy\,dx\,dt\bigg]\\
& \le  - \E \bigg[\int_{Q_T \times \R^d \times \{ |z| \ge r\}} \Bigg(\frac{|A(u_\ep(t,y))-A( u_\theta(t,x))| - | A(u_\ep(t,y+z))- A(u_\theta(t,x+z))|}{|z|^{d+2\lambda}}\Bigg)\\
&\hspace{10cm} \times\varphi(t,x) \rho_m(x-y) \,dz\,dy\,dx\,dt\bigg]\\
&= -\E \bigg[\int_{Q_T \times \R^d} | A( u_\ep(t,y))- A(u_\theta(t,x))|\mathscr{L}^r_{\lambda} [ \varphi(t,\cdot)](x) \rho_m(x-y) \,dx\,dy\,dt\bigg]
\end{align*} 
where to derive the penultimate inequality, we have used the fact that $A(a)-A(b)$ and $\sgn(a-b)$ have the same sign as $A$ is non-decreasing. For the last equality, we have performed a change of variable of coordinates for the first integral $x \to x+z, \,y\to y+z,\, z \to -z$. As we mentioned earlier, we have changed the order in which we pass to the limit in various parameters due to the presence of the non linearity inside the fractional operator. Note that we passed to the limit in the parameter $\delta$ before $\ep$ and $\theta$, which is a different order than in the uniqueness proof of \cite{BhKoleyVa}.

At this point, we first fix $\theta$ and pass to the limit in $\ep$ in the sense of Young measures. For that purpose, let 
us define 
$$G(t,y,\omega; \mu):= \intrd | A(u_\theta(t,x))-A(\mu)|\, \fr^r[ \varphi(t, \cdot)](x) \rho_m(x-y)\,dx.$$
Since $\mathscr{L}^r_{\lambda}[\varphi(t,\cdot)] \in L^p(\Rd)$ for any $t$ and any $p \in [1,\infty]$, $G$ is a Carath\'eodory function on $Q_T \times \Omega \times \R$. Note that $G(\cdot, u_\ep(\cdot))$ is uniformly bounded in $L^2( \Omega \times Q_T)$. Indeed,
\begin{align*}
\E \Big[\int_{Q_T} & |G(t,y;u_\ep)|^2 \,dy\,dt\Big]= \E \bigg[ \intrd \int^T_{t=0} \Big[\intrd |A(u_\theta(t,x))- A(u_\ep(t,y))| \fr^r[ \varphi(t, \cdot)](x) \rho_m(x-y) \,dx\Big]^2 \,dy \,dt \bigg]\\
&\le c(r,A',\varphi) \E \bigg[\intrd \intrd \int^T_{t=0} | u_\theta(t,x)-u_\ep(t,y)|^2 \rho_m(x-y)\,dt \,dx \,dy \bigg]\\
& \le c(r,A',\varphi)\E \bigg[\int_\D \int_\D \int^T_{t=0} ( | u_\theta(t,x)|^2 +| u_\ep(t,y)|^2 ) \rho_m(x-y) \,dt \,dx \,dy \bigg] \le C.
\end{align*}
To ensure that the family $G(\cdot, u_\ep(\cdot))$ is uniformly integrable, we need to check the equi-smallness property at infinity. For that purpose, set $\tilde \eps>0$ and note that for a given $R>r$ 
\begin{align*}
G(t,y,\omega;u_\ep) =& \int_{|z|>r} \frac{1}{|z|^{d+2\lambda}} \intrd |A(u_\theta(t,x))-A(u_\ep(t,y))| [\varphi(t, x)-\varphi(t,x+z)] \rho_m(x-y)\,dx\,dz
\\
=& \int_{|z|>R} \frac{1}{|z|^{d+2\lambda}} \intrd |A(u_\theta(t,x))-A(u_\ep(t,y))| [\varphi(t, x)-\varphi(t,x+z)] \rho_m(x-y)\,dx \,dz
\\&+ \int_{R>|z|>r} \frac{1}{|z|^{d+2\lambda}} \intrd |A(u_\theta(t,x))-A(u_\ep(t,y))| [\varphi(t, x)-\varphi(t,x+z)] \rho_m(x-y)\,dx\,dz.
\end{align*}
Firstly, set $R$ such that 
\begin{align*}
&\E \bigg[\int_{Q_T} \int_{|z|>R} \frac{1}{|z|^{d+2\lambda}} \intrd |A(u_\theta(t,x))-A(u_\ep(t,y)) | |\varphi(t, x)-\varphi(t,x+z)| \rho_m(x-y)\,dx\,dz\,dy\,dt \bigg]\\
\leq & 
C\int_{|z|>R} \frac{1}{|z|^{d+2\lambda}} \E \bigg[\int_{Q_T}  \intrd \big[|u_\theta(t,x)|+|u_\ep(t,y)|\big] \big[|\varphi(t, x)|+|\varphi(t,x+z)|\big] \rho_m(x-y)\,dx\,dy\,dt\bigg]\,dz \\
\leq & C\Big(\|\varphi\|_{L^2(Q_T)},\|u_\theta\|_{L^2(\Omega\times Q_T)},\|u_\eps\|_{L^2(\Omega\times Q_T)}\Big)\frac{1}{R^{2\lambda}} < \tilde \eps.
\end{align*}
Then, considering $M$ such that $M > K+1/m + R$, where supp$\varphi(t,.) \subset \bar B(0,K)$ for any $t$,
\begin{align*}
&\E \bigg[\int_{|y|>M} \int_{R>|z|>r} \frac{1}{|z|^{d+2\lambda}} \intrd |A(u_\theta(t,x))-A(u_\ep(t,y))| [\varphi(t, x)-\varphi(t,x+z)] \rho_m(x-y)\,dx\,dz\,dy\,dt \bigg]\\
=&\int_{R>|z|>r} \frac{1}{|z|^{d+2\lambda}} \E \bigg[ \int_{|y|>M}  \int_{|x-y|<\delta} \hspace{-0.3cm}|A(u_\theta(t,x))-A(u_\ep(t,y))| [\varphi(t, x)-\varphi(t,x+z)] \rho_m(x-y)\,dx\,dy\,dt \bigg]\,dz \\
&\hspace{-0.3cm}=0,
\end{align*}
since then $|x|>K$ and $|x+z|>K$. 
\\
Hence $G(\cdot;u_\ep)$ is uniformly integrable, and taking advantage of Young measures theory, we conclude that
\begin{align*}
&\lim_{\ep \to 0} \,\E\bigg[ \intrd \int_{Q_T}|A( u_\theta(t,x))-A(u_\ep(t,y))| \fr^r[ \varphi(t, \cdot)](x) \rho_m(x-y) \,dy\,dx\,dt \bigg]\\
& \qquad \qquad = \E\bigg[ \intrd \int_{Q_T} \int^1_0|A( u_\theta(t,x))-A(v(t,y,\beta))| \fr^r[ \varphi(t, \cdot)](x) \rho_m(x-y) \,d\beta \,dy\,dx\,dt \bigg].
\end{align*}
A verbatim copy of the above arguments yields
\begin{align*}
&\lim_{\theta \to 0}  \,\E\bigg[ \intrd \int_{Q_T} \int^1_0 | A(u_\theta(t,x))-A(v(t,y,\beta))| \fr^r[ \varphi(t, \cdot)](x) \rho_m(x-y) \,d\beta \,dy\,dx\,dt \bigg]\\
& \qquad \qquad = \E\bigg[ \intrd \int_{Q_T} \int^1_0\int^1_0 | A(u(t,x,\alpha))-A(v(t,y,\beta))| \fr^r[ \varphi(t, \cdot)](x) \rho_m(x-y) \,d\alpha\,d\beta \,dy\,dx\,dt\bigg].
\end{align*}
This finishes the proof.
\end{proof}

\begin{lem}
\label{kato_lemma2}
It holds that
\begin{align*}
I_8+J_8 &\underset{n \to \infty} \longrightarrow  -\E \bigg[\int_{Q_T \times \R \times \R^d} \mathscr{L}_{\lambda,r}[A( u_\eps)^\gamma(t,\cdot)](y)  \varphi(t,x) \rho_m(x-y) \eta^\prime( u^\gamma_\eps(t,y)-k)\rho_l(u_\theta(t,x)-k) \,dy\,dk\,dx\,dt\bigg]\\
&\hspace{1cm} -\E \bigg[\int_{Q_T\times\R\times \R^d} A^\eta_k (u_\theta(t,x))  \mathscr{L}_{\lambda,r}[\varphi(t,\cdot) \rho_m(\cdot -y)](x) \rho_l(u^\gamma_\eps(t,y)-k) \,dy\,dk\,dx\,dt\bigg]\\
&\underset{\gamma \to 0} \longrightarrow -\E \bigg[\int_{Q_T\times \R} \big \langle\mathscr{L}_{\lambda,r}[A( u_\eps)(t,\cdot)](y),   \rho_m(x-y) \eta^\prime( u_\eps(t,y)-k)\big \rangle \varphi(t,x) \rho_l(u_\theta(t,x)-k) \,dk\,dx\,dt\bigg] \\
&\hspace{1cm} -\E \bigg[\int_{Q_T\times\R\times \R^d}A^\eta_k (u_\theta(t,x))  \mathscr{L}_{\lambda,r}[\varphi(t,\cdot) \rho_m(\cdot -y)](x) \rho_l(u_\eps(t,y)-k) \,dy\,dk\,dx\,dt\bigg]\\
& \hspace{-1.5cm}\limsup_{\delta \to 0}\lim_{\gamma \to 0} \lim_{n \to \infty}(I_8+J_8) \le  -\E \bigg[\int_{Q_T \times \R^d\times \R} \hspace{-0.9cm}| A(u_\ep(t,y))-A(k)| \mathscr{L}_{\lambda,r}[ \rho_m(x-\cdot)](y)\varphi(t,x)\rho_l(u_\theta(t,x)-k) \,dy\,dk\,dx\,dt\bigg]\\
&\hspace{1cm}-\E \bigg[\int_{Q_T\times \R^d\times \R} | A(u_\theta(t,x))-A(k)| \mathscr{L}_{\lambda,r}[ \varphi(t,\cdot) \rho_m(\cdot-y)](x) \rho_l(u_\ep(t,y)-k) \,dy\,dk\,dx\,dt\bigg]\\
&\underset{l \to 0} \longrightarrow -\E \bigg[\int_{Q_T \times \R^d} | A(u_\ep(t,y))-A(u_\theta(t,x))| \mathscr{L}_{\lambda,r}[ \rho_m(x-\cdot)](y)\varphi(t,x) \,dy\,dx\,dt\bigg]\\
& \hspace{1cm}-\E \bigg[\int_{Q_T\times \R^d} | A(u_\theta(t,x))-A(u_\ep(t,y))| \mathscr{L}_{\lambda,r}[ \varphi(t,\cdot) \rho_m(\cdot-y)](x)  \,dy\,dx\,dt\bigg]\\
& \underset{ \ep \to 0} \longrightarrow -\E \bigg[\int_{Q \times \R^d \times (0,1)} | A(v(t,y,\beta))-A(u_\theta(t,x))| \mathscr{L}_{\lambda,r}[ \rho_m(x-\cdot)](y)\varphi(t,x) \,d\beta\,dy\,dx\,dt\bigg]\\
& \hspace{1cm}-\E \bigg[\int_{Q_T\times \R^d\times (0,1)} | A(u_\theta(t,x))-A(v(t,y,\beta))| \mathscr{L}_{\lambda,r}[ \varphi(t,\cdot) \rho_m(\cdot-y)](x)  \,d\beta\,dy\,dx\,dt\bigg]\\
& \underset{ \theta \to 0} \longrightarrow -\E \bigg[\int_{Q \times \R^d \times (0,1)^2} | A(v(t,y,\beta))-A(u(t,x,\alpha))| \mathscr{L}_{\lambda,r}[ \rho_m(x-\cdot)](y)\varphi(t,x) \,d\beta\,d\alpha\,dy\,dx\,dt\bigg]\\
& \hspace{1cm}-\E \bigg[\int_{Q_T\times \R^d\times (0,1)^2} | A(u(t,x,\alpha))-A(v(t,y,\beta))| \mathscr{L}_{\lambda,r}[ \varphi(t,\cdot) \rho_m(\cdot-y)](x)  \,d\beta\,d\alpha\,dy\,dx\,dt\bigg]
\end{align*}
\begin{proof}
We shall prove the lemma on the singular part of the non-local operator  in several steps.

\noindent {\bf Step 1} (Passing to the limit as $n \to \infty$):
Consider 
\begin{align*}
&\E \bigg[ \int_\R \int_{Q^2_T} \bigg[ \mathscr{L}_{\lambda,r}[A( u_\eps)^\gamma(s,\cdot)](y)\eta^\prime( u^\gamma_\eps(s,y)-k)- \mathscr{L}_{\lambda,r}[A( u_\eps)^\gamma(t,\cdot)](y)\eta^\prime( u^\gamma_\eps(t,y)-k) \bigg] \\
&\hspace{2cm} \times \varphi(t,x) \rho_m(x-y) \rho_l(u_\theta(t,x)-k) \rho_n(t-s) \,ds\,dt\,dx\,dy\,dk\bigg]\\
&=\E \bigg[ \int_\R \int_{Q^2_T} \bigg( \mathscr{L}_{\lambda,r}[A( u_\eps)^\gamma(s,\cdot)](y)-\mathscr{L}_{\lambda,r}[A( u_\eps)^\gamma(t,\cdot)](y)\bigg)\eta^\prime( u^\gamma_\eps(s,y)-k)\\
&\hspace{4cm} \times \varphi(t,x) \rho_m(x-y) \rho_l(u_\theta(t,x)-k) \rho_n(t-s)\,ds\,dt\,dx\,dy\,dk\bigg]\\
&+ \E\bigg[ \int_\R \int_{Q^2_T} \mathscr{L}_{\lambda,r}[A( u_\eps)^\gamma(t,\cdot)](y)\bigg(\eta^\prime( u^\gamma_\eps(s,y)-k)-\eta'(u^\gamma_\eps(t,y)-k)\bigg)\rho_n(t-s) \\
&\hspace{4cm}\times \varphi(t,x) \rho_m(x-y) \rho_l(u_\theta(t,x)-k) \,ds\,dt\,dx\,dy\,dk\bigg]:= I+J.
\end{align*}
Let us first compute $I$, we consider,
\begin{align*}
I &= \E\bigg[ \int_\R \int_{Q^2_T}  \int_{\{|z| \le r\}} \frac{[A(u_\ep(s,\cdot) -A(u_\ep(t,\cdot)] \ast \rho_\gamma(y)-[A(u_\ep(s,\cdot) -A(u_\ep(t,\cdot)] \ast \rho_\gamma(y+z)]}{|z|^{d+2\lambda}}\,dz\\
& \hspace{3cm}\times  \eta'( u^\gamma_\eps(s,y)-k)\varphi(t,x) \rho_m(x-y) \rho_l(u_\theta(t,x)-k) \rho_n(t-s)\,ds\,dt\,dx\,dy\,dk\bigg]\\
& (\text{ adding the zero term} \int_{\{|z| \le r\}}\hspace*{-0.7cm}\frac{\nabla[A(u_\ep(s,\cdot)) -A(u_\ep(t,\cdot))] \ast \rho_\gamma(y)].z}{|z|^{d+2\lambda}}dz \text{ and then using Taylor's series}), \\
& \le c(r) \E \bigg[ \int_{Q^2_T} \| (A(u_\ep(s,\cdot)) -A(u_\ep(t,\cdot))) \ast D^2\rho_\gamma\|_{L^\infty(B(y,r))} \varphi(t,x) \rho_m(x-y) \rho_n(t-s) \,dy\,dx\,ds\,dt\bigg]\\
& \text{ (using Young's inequality for convolution) }\\
&\le c(r) \E\bigg[ \int_{Q^2_T} \| A(u_\ep(s,\cdot))-A(u_\ep(t,\cdot))\|_{L^2} \|D^2 \rho_\gamma\|_{L^2}\, \varphi(t,x) \rho_m(x-y) \rho_n(t-s) \,dy\,dx\,ds\,dt\bigg]\\
& \text{ (using Cauchy-Schwarz's inequality)} \\
&\le c(r,\rho_\gamma,\varphi) \E \bigg[ \int^T_0 \int^T_0  \|A(u_\ep(s,\cdot))-A(u_\ep(t,\cdot))\|^2_{L^2}\, \rho_n(t-s) \,ds\,dt\bigg] \overset{ n \to \infty} \longrightarrow 0.
\end{align*}
To compute $J$, first notice that 
$$|\eta^\prime( u^\gamma_\eps(s,y)-k)-\eta'(u^\gamma_\eps(t,y)-k)| \le \text{min}\{2, \|\eta''\|_\infty ( u^\gamma_\eps(s,y)-u^\gamma_\eps(t,y))\}.$$
and thus using Lebesgue's dominated convergence theorem we conclude that $J \to 0$.
\\[0.2cm]
Now we shall consider the second term:
\begin{align*}
J_8=-\E \bigg[\int_{\R \times Q^2_T} A^\eta_k(u_\theta(t,x))\mathscr{L}_{\lambda,r}[\psi(t,\cdot,s,y)](x) \rho_l(u^\gamma_\eps(s,y)-k) \,ds\,dy\,dk\,dx\,dt\bigg].
\end{align*}
Note that $$|A^\eta_a(c)-A^\eta_b(c)| \le \norm{A'}_{\infty}\,|a-b|\Big(1 + \norm{\eta''}_{\infty}\,\max{\lbrace |c-b|, |c-a| \rbrace}\Big) $$
and that all the integrals below are over compact sets thanks to $\supp(\varphi), \supp(\rho_m)$ and the set $\{|z|\le r\}$. Indeed, if $K$ a compact set containing supp$\varphi(t,.)$ for any $t$ and $\tilde{K}=K+\overline{B(0,\frac 1m)}+\overline{B(0,r)}$, then $K$ and $\tilde{K}$ are the sets of integration of $x$ and $y$ respectively. Thus, 
\begin{align*}
&\bigg|\E \bigg[\int_{Q^2_T\times\R} A^\eta_k (u_\theta(t,x))  \mathscr{L}_{\lambda,r}[\psi(t,\cdot,s,y)](x) \rho_l(u^\gamma_\eps(s,y)-k) \,dy\,ds\,dk\,dx\,dt\bigg]\\
& \quad -\E \bigg[\int_{Q_T\times\R\times \R^d} A^\eta_k (u_\theta(t,x)) \mathscr{L}_{\lambda,r}[\varphi(t,\cdot) \rho_m(\cdot -y)](x) \rho_l(u^\gamma_\eps(t,y)-k) \,dy\,dk\,dx\,dt\bigg]\bigg|\\
&= \bigg| \E \bigg[ \int_{Q^2_T\times\R} \bigg( A^\eta_{u^\gamma_\ep(s,y)-k}(u_\theta(t,x)) - A^\eta_{u^\gamma_\ep(t,y)-k}(u_\theta(t,x)) \bigg) \\
& \hspace{5cm}\times \mathscr{L}_{\lambda,r}[ \varphi(t,\cdot) \rho_m(\cdot -y )](x) \rho_l(k) \rho_n(t-s) \,dy\,ds\,dk\,dx\,dt \bigg]\bigg|\\
&\le C \E \bigg[ \int_{Q^2_T\times\R} |u^\gamma_\ep(s,y)- u^\gamma_\ep(t,y)|\big( | u_\theta(t,x)-u^\gamma_\ep(t,y)+k| +1\big)\\
& \hspace{5cm} \times \mathscr{L}_{\lambda,r}[ \varphi(t,\cdot) \rho_m(\cdot -y )](x) \rho_l(k) \rho_n(t-s) \,dy\,ds\,dk\,dx\,dt \bigg]\\
&\le c(\varphi, \rho_m) \E \bigg[ \int_{(0,T)^2 \times \R \times K \times \tilde{K}} \hspace{-1.5cm}|u^\gamma_\ep(s,y)- u^\gamma_\ep(t,y)|\, (|u_\theta(t,x)-u^\gamma_\ep(t,y)+k| + 1) \rho_l(k) \rho_n(t-s) \,dk\,dy\,dx\,ds\,dt \bigg]\\
\le& c(\varphi, \rho_m) \bigg(\E \bigg[ \int_{(0,T)^2 \times K \times \tilde{K}} | u^\gamma_\ep(s,y)- u^\gamma_\ep(t,y)|^2\rho_n(t-s)\,dy\,dx\,ds\,dt\bigg] \bigg)^{1/2}\\
& \hspace{0.5cm} \times \bigg( \E \bigg[ \int_{(0,T)^2 \times \R \times K \times \tilde{K}} \Big(|u_\theta(t,x) -u^\gamma_\ep(t,y)|^2 +(l+1)^2\Big) \rho_l(k)\rho_n(t-s) \,dk\,dy\,dx\,ds\,dt \bigg] \bigg)^{1/2} \overset{n \to \infty} \longrightarrow 0.
\end{align*}

\noindent {\bf Step 2} (Passing to the limit as $\gamma \to 0$):
Consider 
\begin{align*}
& \bigg| \E \bigg[\int_{Q_T\times \D \times \R} \mathscr{L}_{\lambda,r}[A( u_\eps)^\gamma(t,\cdot)](y)  \varphi(t,x) \rho_m(x-y) \eta^\prime( u^\gamma_\eps(t,y)-k)\rho_l(u_\theta(t,x)-k) \,dy\,dk\,dx\,dt\bigg]
\\& 
- \E \bigg[\int_{\D \times (0,T)\times \R} \Big\langle \mathscr{L}_{\lambda,r}[A( u_\eps)(t,\cdot)](y), \rho_m(x-y) \eta^\prime( u_\eps(t,y)-k)\Big\rangle \varphi(t,x) \rho_l(u_\theta(t,x)-k) \,dk\,dx\,dt\bigg]\bigg|
\\&
\le \bigg| \E \bigg[\int_{\D \times (0,T)\times \R} \hspace{-0.6cm} \Big\langle \mathscr{L}_{\lambda,r} [ A( u_\eps)^\gamma(t,\cdot)-A( u_\eps)(t,\cdot)](y), \rho_m(x-y) \eta^\prime( u^\gamma_\eps(t,y)-k)\Big\rangle \varphi(t,x) \rho_l(u_\theta(t,x)-k) \,dk\,dx\,dt\bigg]\bigg|
\\&
+ \bigg|\E \bigg[\int_{\D \times (0,T)\times \R} \hspace{-1cm} \Big\langle \mathscr{L}_{\lambda,r}[A( u_\eps)(t,\cdot)](y), \rho_m(x-y)( \eta^\prime( u^\gamma_\eps(t,y)-k)-\eta'(u_\ep(t,y)-k))\Big\rangle \varphi(t,x) \rho_l(u_\theta(t,x)-k) \,dk\,dx\,dt\bigg]\bigg|
\\
& \le \E \bigg[\int_{\D \times (0,T) \times \R} \hspace{-0.9cm}\| A(u_\ep)^\gamma(t,\cdot)-A(u_\ep(t,\cdot))\|_{H^\lambda(\D)} \| \rho_m(x-\cdot) \eta'(u^\gamma_\ep(t,\cdot)-k)\|_{H^\lambda(\D)}\varphi(t,x) \rho_l(u_\theta(t,x)-k)\,dk\,dx\,dt\bigg]
\\
&+\E \bigg[\int_{\D \times (0,T)\times \D} \hspace{-0.9cm} \| A(u_\ep)(t,\cdot)\|_{H^\lambda(\D)} \| \rho_m(x-\cdot)(\eta'(u^\gamma_\ep(t,\cdot)-k)- \eta'(u_\ep(t,y)-k))\|_{H^\lambda(\D)} \varphi(t,x) \rho_l(u_\theta(t,x)-k)\,dx\,dt\bigg]\\
 &=: I + J, \mbox{where we used the variational representation of $\mathscr{L}_{\lambda,r}$.}
\end{align*}
We shall show that both $I$ and $J$ go to zero as $\gamma $ goes to 0. To see this, consider 
\begin{align*}
&I \le C \E \bigg[\int_{\D \times (0,T) \times \R} \hspace{-0.9cm} \| A(u_\ep)^\gamma(t,\cdot)-A(u_\ep(t,\cdot))\|_{H^1(\D)} \| \rho_m(x-\cdot) \eta'(u^\gamma_\ep(t,\cdot)-k)\|_{H^1(\D)}\varphi(t,x) \rho_l(u_\theta(t,x)-k)\,dk\,dx\,dt\bigg]
\\
& \le C\bigg(\E \bigg[\int_{\D \times (0,T) \times \R} \| A(u_\ep)^\gamma(t,\cdot)-A(u_\ep(t,\cdot))\|^2_{H^1(\D)} \varphi(t,x) \rho_l(u_\theta(t,x)-k)\,dx\,dk\,dt\bigg]\bigg)^{1/2}\\
&\hspace{1cm} \times \bigg(\E \bigg[\int_{\D \times (0,T) \times \R}  \| \rho_m(x-\cdot) \eta'(u^\gamma_\ep(t,\cdot)-k)\|^2_{H^1(\D)}\varphi(t,x) \rho_l(u_\theta(t,x)-k)\,dk\,dx\,dt\bigg]\bigg)^{1/2}
\\
&\le c(\varphi) \bigg(\E \bigg[\int^T_0\| A(u_\ep)^\gamma(t,\cdot)-A(u_\ep(t,\cdot))\|^2_{H^1(\D)} \,dt\bigg]\bigg)^{1/2} \\
& \hspace{0.5cm}\times \bigg(\E \bigg[\int_{\D \times (0,T) \times \R} \Big(\|\rho_m(x-\cdot) \eta'(u^\gamma_\ep(t,\cdot)-k)\|^2_{L^2(\D)}+\| \nabla \rho_m(x-y) \eta'(u^\gamma_\ep -k) \\
&\hspace{3cm} + \rho_m(x-y) \eta''(u^\gamma_\ep -k) \nabla u^\gamma_\ep\|^2_{L^2(\D)} \Big)\times \varphi(t,x) \rho_l( u_\theta(t,x)-k) \,dk\,dx\,dt \bigg] \bigg)^{1/2} \overset{ \gamma \to 0} \longrightarrow 0.
\end{align*}
Let us consider J, for convenience let $\chi_\gamma(t,\cdot)=\eta'(u^\gamma_\ep(t,\cdot)-k)) - \eta'(u_\ep(t,\cdot)-k)$.
\begin{align*}
& J \le \left| \E \left[ \int_{\D \times (0,T) \times \R} \| A(u_\ep)(t,\cdot)\|_{H^1(\D)}\| \rho_m(x-\cdot)\chi_\gamma(t,\cdot)\|_{H^1(\D)} \varphi(t,x) \rho_l(u_\theta(t,x)-k)\,dk\,dx\,dt \right]\right|\\
\le&  \left( \E \left[ \int_{\D \times (0,T) \times \R} \| A(u_\ep)(t,\cdot)\|^2_{H^1(\D)} \varphi(t,x) \rho_l(u_\theta(t,x)-k) \,dk\,dx\,dt \right] \right)^{1/2}\\
&\hspace{2cm}\times \left(\E \left[ \int_{\D \times (0,T) \times \R}\| \rho_m(x- \cdot) \chi_\gamma(t,\cdot)\|_{H^1(\D)} \varphi(t,x) \rho_l(u_\theta(t,x)-k)\,dk\,dx\,dt \right]\right)^{1/2}\\
\le& C( \varphi, \| A'\|_\infty) \| u_\ep \|_{L^2( \Omega \times (0,T); H^1(\D))}\\
&\times \bigg( \E \bigg[ \int_{\D \times (0,T) \times \R} \Big( \| \rho_m(x- \cdot) \chi_\gamma(t,\cdot)\|^2_{L^2(\D)}+\| \nabla \rho_m(x- \cdot)\chi_\gamma(t,\cdot)\|^2_{L^2(\D)} \\
& \hspace{3cm}+ \| \rho_m(x- \cdot)\nabla\chi_\gamma(t,\cdot)\|^2_{L^2(\D)}\Big) \varphi(t,x) \rho_l(u_\theta(t,x)-k)\,dk\,dx\,dt\bigg] \bigg)^{1/2}\\
\le&  C \bigg( \E \left[ \int_{Q_T \times \D }  \Big(|\rho_m(x-y)|^2+| \nabla \rho_m(x-y)|^2\Big)|u^\gamma_\ep(t,y)-u_\ep(t,y)|^2 \varphi(t,x)\,dy\,dx\,dt  \right]\\
& + \E \left[ \int_{Q_T \times \D \times \R} | \rho_m(x-y)\nabla\chi_\gamma(t,y) |^2\varphi(t,x)\rho_l(u_\theta(t,x)-k)\,dy\,dx\,dt\,dk \right]\bigg)^{1/2}\\
\le& C(m,\varphi, \eta'') \bigg( \| u^\gamma_\ep-u_\ep\|^2_{L^2(\Omega \times Q_T)}
+ \E \bigg[ \int_{Q_T \times \D } | \rho_m(x-y)|^2| \nabla u^\gamma_\ep(t,y)- \nabla u_\ep(t,y))|^2 \varphi(t,x) \,dy\,dx\,dt \bigg]\\
&+ \E \left[ \int_{Q_T \times \D \times \R} \hspace{-1cm}| \rho_m(x-y)|^2(  \eta''(u^\gamma_\ep(t,y)-k) - \eta''(u_\ep(t,y)-k) )^2 | \nabla u_\ep(t,y)|^2 \varphi(t,x) \rho_l(u_\theta(t,x)-k) \,dy\,dx\,dt\,dk \right]\bigg)^{1/2}\\
& \overset{ \gamma \to 0} \longrightarrow 0,
\end{align*}
where for the last integral we use the fact that $\eta''(u^\gamma_\ep-k)-\eta''(u_\ep-k)$ is uniformly bounded in $\gamma$ as $ 0 \le \eta''\le C/\delta$, and thus an application of dominated convergence theorem gives the result. 

\noindent Concerning the second term,  Lebesgue's Theorem yields
\begin{align*}
&\lim_{\gamma \to 0}\E \bigg[\int_{Q_T\times\R\times \R^d} A^\eta_k (u_\theta(t,x)) \mathscr{L}_{\lambda,r}[\varphi(t,\cdot) \rho_m(\cdot -y)](x) \rho_l(u^\gamma_\eps(t,y)-k) \,dy\,dk\,dx\,dt\bigg]
\\=&
\E \bigg[\int_{Q_T\times\R\times \R^d} A^\eta_k (u_\theta(t,x)) \mathscr{L}_{\lambda,r}[\varphi(t,\cdot) \rho_m(\cdot -y)](x) \rho_l(u_\eps(t,y)-k) \,dy\,dk\,dx\,dt\bigg].
\end{align*}

\noindent {\bf Step 3} (Passing to the limit as $ \delta \to 0$):
Note that, from Step 2 , we have
\begin{align*}
&\lim_{\gamma \to 0} \lim_{n \to \infty}(I_8+J_8)
\\ =& -\E \bigg[\int_{\D \times \R} \bigg[ \int^T_0 \Big\langle\mathscr{L}_{\lambda,r}[A( u_\ep)(t,\cdot)](y), \rho_m(x-y) \eta^\prime(u_\ep(t,y)-k)\Big\rangle \,dt\bigg]  \rho_l(u_\theta(t,x)-k)\varphi(t,x) \,dk\,dx\bigg]\\
&-\E \bigg[\int_{Q_T\times\R\times \R^d}A^\eta_k (u_\theta(t,x))  \mathscr{L}_{\lambda,r}[\varphi(t,\cdot) \rho_m(\cdot -y)](x) \rho_l(u_\eps(t,y)-k) \,dy\,dk\,dx\,dt\bigg].
\end{align*}
Identical calculations  as the ones concerning $I_r$ in Appendix~\ref{appendix_entropy} yield
\begin{align*}
\lim_{\gamma \to 0} \lim_{n \to \infty}(I_8+J_8) &\le -\E \bigg[\int_{Q_T\times \R^d\times \R} A^\eta_k( u_\ep(t,y)) \mathscr{L}_{\lambda,r}[ \rho_m(x-\cdot)](y)\varphi(t,x) \rho_l(u_\theta(t,x)-k)\,dy\,dk\,dx\,dt\bigg]\\
&-\E \bigg[\int_{Q_T\times \R^d\times \R} A^\eta_k(u_\theta(t,x)) \mathscr{L}_{\lambda,r}[ \varphi(t,\cdot) \rho_m(\cdot-y)](x) \rho_l(u_\ep(t,y)-k) \,dy\,dk\,dx\,dt\bigg].
\end{align*}
Now we shall pass to the limit as $\delta$ goes to $0$. For that, consider,
\begin{align*}
&\mathcal{B}_1=\Bigg|\bigg[ \E \int_{Q_T \times \R^d\times \R}  A^\eta_k( u_\ep(t,y)) \mathscr{L}_{\lambda,r}[ \rho_m(x-\cdot)](y)\varphi(t,x) \rho_l(u_\theta(t,x)-k)\,dy\,dk\,dx\,dt \bigg]\\
&-
 \E \bigg[\int_{Q_T \times \R^d\times \R}  |A( u_\ep(t,y))-A(k)| \mathscr{L}_{\lambda,r}[ \rho_m(x-\cdot)](y)\varphi(t,x)\rho_l(u_\theta(t,x)-k) \,dy\,dk\,dx\,dt\bigg]\Bigg|\\
& \le \E \bigg[ \int_{Q_T \times \R^d\times \R}  \bigg|A^\eta_k( u_\ep(t,y))-|A( u_\ep(t,y))-A(k)|\bigg|\,| \mathscr{L}_{\lambda,r}[ \rho_m(x-\cdot)](y)|\varphi(t,x) \rho_l(u_\theta(t,x)-k)\,dy\,dk\,dx\,dt \bigg]\\
&\le \delta\|A'\|_\infty \E\bigg[ \int_{Q_T \times \R^d\times \R} | \mathscr{L}_{\lambda,r}[ \rho_m(x-\cdot)](y)|\varphi(t,x) \rho_l(u_\theta(t,x)-k)\,dk\,dy\,dx\,dt \bigg] \underset{\delta \to 0} \longrightarrow 0,
\end{align*}
where we have used the fact that $\Big|A^\eta_b(a)-|A(a)-A(b)|\Big| \le \delta \|A'\|_\infty$ and that both $x,y$  vary over a fixed compact support depending on $m,r$ and $\supp(\varphi)$. \\
Similarly for the second term we have
\begin{align*}
&-\E \bigg[\int_{Q_T\times \R^d\times \R} A^\eta_k(u_\theta(t,x)) \mathscr{L}_{\lambda,r}[ \varphi(t,\cdot) \rho_m(\cdot-y)](x) \rho_l(u_\ep(t,y)-k) \,dy\,dk\,dx\,dt\\
&\underset{\delta \to 0} \longrightarrow -\E \bigg[\int_{Q_T\times \R^d\times \R} | A(u_\theta(t,x))-A(k)| \mathscr{L}_{\lambda,r}[ \varphi(t,\cdot) \rho_m(\cdot-y)](x) \rho_l(u_\ep(t,y)-k) \,dy\,dk\,dx\,dt
\end{align*}

\noindent {\bf Step 4} (Passing to the limit as $ l \to 0$):
Consider the term 
\begin{align*}
\mathcal{B}_2=&\Bigg|   \bigg[\E \int_{Q_T \times \R^d\times \R}  |A( u_\ep(t,y))-A(k)| \mathscr{L}_{\lambda,r}[ \rho_m(x-\cdot)](y)\varphi(t,x)\rho_l(u_\theta(t,x)-k) \,dy\,dk\,dx\,dt \bigg]\\
&-  \E \bigg[\int_{Q_T \times \R^d}  |A( u_\ep(t,y))-A(u_\theta(t,x))| \mathscr{L}_{\lambda,r}[ \rho_m(x-\cdot)](y)\varphi(t,x) \,dy\,dx\,dt \bigg]\Bigg|\\
&\le \E \bigg[\int_{Q_T\times \R \times \R^d} \bigg| |A( u_\ep(t,y))-A(k)|-|A( u_\ep(t,y))-A(u_\theta(t,x))|\bigg|\, |\mathscr{L}_{\lambda,r}[ \rho_m(x-\cdot)](y)|\\[-0.3cm]
&\hspace{8cm}\times \varphi(t,x) \rho_l(u_\theta(t,x)-k) \,dy\,dk\,dx\,dt\bigg]\\
&\le \E \int_{Q_T \times \R \times \R^d} | A(k) -A( u_\theta(t,x))|\,  |\mathscr{L}_{\lambda,r}[ \rho_m(x-\cdot)](y)|\varphi(t,x) \rho_l(u_\theta(t,x)-k) \,dy\,dk\,dx\,dt\bigg]\\
&\le l \| A'\|_\infty \E \int_{Q_T \times \R^d}   |\mathscr{L}_{\lambda,r}[ \rho_m(x-\cdot)](y)|\varphi(t,x) \,dy\,dx\,dt\bigg]\underset{ l \to 0} \longrightarrow 0,
\end{align*}
where we have used the fact that $|u_\theta(t,x)-k| <l$ as supp$\rho \subset [-1,1]$. 
Similarly for the second term we get
\begin{align*}
&-\E \bigg[\int_{Q_T\times\R\times \R^d} |A( u_\theta(t,x))-A(k)| \mathscr{L}_{\lambda,r}[ \varphi(t,\cdot) \rho_m(\cdot-y)](x) \rho_l(u_\ep(t,y)-k) \,dy\,dk\,dx\,dt\bigg]\\
&\underset{l \to 0} \longrightarrow -\E \bigg[\int_{Q_T\times \R^d} |A( u_\theta(t,x))-A(k)| \mathscr{L}_{\lambda,r}[ \varphi(t,\cdot) \rho_m(\cdot-y)](x)  \,dy\,dx\,dt\bigg].
\end{align*}

\noindent {\bf Step 5} (Passing to the limit as $ \ep,\theta \to 0$):
Consider
\begin{align*}
\mathcal{B}_1:=\E \bigg[\int_{\R^d \times Q_T}  |A(u_\ep(t,y)) - A(u_\theta(t,x) )| \,\mathscr{L}_{\lambda,r} \big[\varphi(t,\cdot) \rho_m (\cdot-y)\big](x)  \,dx\,dt \,dy \bigg]
\end{align*}
As before we define,
$$G(t,y,\omega,\mu) = \intrd |A(\mu)-A(u_\theta(t,x))| \mathscr{L}_{\lambda,r} \big[ \varphi(t,\cdot)\rho_m (\cdot-y)\big](x)\,dx. $$
$G$ is a Carath\'eodory function and the above integral holds in the compact set $K+\bar{B}(0,r)$ thanks to the compact support of $\varphi$. $G$ is compactly supported in $K+\bar{B}(0,r+\frac1m)$ and $G(u_\ep)$ is bounded in $L^2(\Omega \times Q_T)$, indeed, 
\begin{align*}
&\E \bigg[\int_{Q_T} \bigg| \intrd |A(u_\ep(t,y)) - A(u_\theta(t,x))| \mathscr{L}_{\lambda,r} \big[\varphi(t,\cdot)\rho_m (\cdot-y)\big](x)\,dx\bigg|^2\,dy\,dt \bigg]\\
& \le C(A^\prime, \rho_m,\varphi) \E \bigg[\int_0^T \int_{K +\bar{B}(0, r +1/m)}\int_{K+\Bar{B}(0,r)} | u_\theta(t,x)-u_\ep(t,y)|^2\,dx \,dy \,dt \bigg] \le C.
\end{align*}
Thus, $G(\cdot, u_\ep(\cdot))$ is uniformly integrable and the Young measures theory gives 
\begin{align*}
\mathcal{B}_1 \underset{\ep \to 0} \longrightarrow  \E \bigg[\int_{\R^d} \int_{Q_T} \int_{\beta=0}^1  |A(v(t,y,\beta))-A(u_\theta(t,x))| \,\mathscr{L}_{\lambda,r} \big[\varphi(t,\cdot)\rho_m (\cdot-y)\big](x)  \, d\beta \,dx\,dt \,dy \bigg].
\end{align*}
Similarly, it can be shown that 
\begin{align*}
& \E \bigg[\int_{\R^d} \int_{Q_T} \int_{\beta=0}^1  |A(v(t,y,\beta)-A(u_\theta(t,x))| \,\mathscr{L}_{\lambda,r} \big[ \varphi(t,\cdot)\rho_m (\cdot-y)\big](x) \, d\beta \,dx\,dt \,dy \bigg]\\ &\underset{\theta\to 0}\longrightarrow \E \bigg[\int_{\R^d} \int_{Q_T}\int^1_{\alpha=0} \int_{\eta=0}^1  |A(v(t,y,\beta))-A(u(t,x,\alpha))| \,\mathscr{L}_{\lambda,r} \big[\varphi(t,\cdot)\rho_m (\cdot-y)\big](x) \,d\alpha  \, d\beta \,dx\,dt \,dy \bigg].
\end{align*}
For the other term, we consider
$$\mathcal{B}_2:= \E \bigg[\int_{\Rd} \int_{Q_T} |A( u_\theta(t,x))-A(u_\ep(t,y))| \mathscr{L}_{\lambda, r}[\rho_m(x-\cdot)](y) \varphi(t,x) \, \,dx\,dt\,\,dy \bigg].$$
Following the same analysis as for the term $\mathcal{A}_1$, we conclude
\begin{align*}
\mathcal{B}_2 &\underset{\ep \to 0} \longrightarrow  \E \bigg[\int_{\R^d} \int_{Q_T}  \int_{\beta=0}^1 |A(u_\theta(t,x))-A(v(t,y,\beta))| \mathscr{L}_{\lambda, r}[\rho_m(x-\cdot)](y) \varphi(t,x) \, d\beta \,dx\,dt \,dy \bigg]\\
&\underset{\theta \to 0}\longrightarrow \E \bigg[\int_{\R^d} \int_{Q_T} \int_{\alpha=0}^1 \int_{\beta=0}^1 |A(u(t,x,\alpha))-A(v(t,y,\beta) )| \mathscr{L}_{\lambda, r}[\rho_m(x-\cdot)](y) \varphi(t,x) \, \,d\alpha \, d\beta \,dx\,dt \,dy \bigg].
\end{align*}
This concludes the proof.
\end{proof}
\end{lem}

\begin{lem}
\label{kato_lemma3}
It holds that 
\begin{align*}
&\limsup_{\theta \goto 0}\limsup_{\eps \goto 0}\limsup_{l \goto 0}\limsup_{\delta \goto 0}\lim_{\gamma \to 0}\lim_{n \goto \infty}(I_7+J_7+I_8+J_8) \le \\
 &-\E \bigg[\int_{Q_T \times \R^d\times (0,1)^2} | A(u(t,x,\alpha)) - A( v(t,y,\beta)) |\mathscr{L}^r_\lambda[ \varphi(t,\cdot)](x) \rho_m(x-y) \,d\alpha\,d\beta\,dx\,dy\,dt \bigg]\\
&-\E \bigg[ \int_{Q_T \times \R^d \times (0,1)^2} | A(v(t,y,\beta))-A(u(t,x,\alpha))| \mathscr{L}_{\lambda,r}[ \rho_m(x-\cdot)](y)\varphi(t,x) \,d\beta\,d\alpha\,dy\,dx\,dt \bigg]\\
& -\E \bigg[ \int_{Q_T\times \R^d\times (0,1)^2} | A(u(t,x,\alpha))-A(v(t,y,\beta))| \mathscr{L}_{\lambda,r}[ \varphi(t,\cdot) \rho_m(\cdot-y)](x)  \,d\beta\,d\alpha\,dy\,dx\,dt \bigg]\\
& \underset{ r \to 0} \longrightarrow -\E \bigg[ \int_{Q_T \times \R^d\times (0,1)^2} | A(u(t,x,\alpha)) - A( v(t,y,\beta)) |\mathscr{L}_\lambda[ \varphi(t,\cdot)](x) \rho_m(x-y) \,d\alpha\,d\beta\,dx\,dy\,dt \bigg]\\
& \underset{m \to \infty} \longrightarrow -\E \bigg[ \int_{Q_T \times (0,1)^2} | A(u(t,x,\alpha)) - A( v(t,x,\beta)) |\mathscr{L}_\lambda[ \varphi(t,\cdot)](x)\,d\alpha\,d\beta\,dx\,dt \bigg].
\end{align*} 
\begin{proof}
First, note that we already have from Lemma~\ref{kato_lemma1} and Lemma~\ref{kato_lemma2}
\begin{align*}
\limsup_{\theta \goto 0} &\limsup_{\eps \goto 0}\limsup_{l \goto 0}\limsup_{\delta \goto 0}\lim_{\gamma \to 0}\lim_{n \goto \infty}(I_7+J_7+I_8+J_8) \le \\
 &-\E \bigg[\int_{Q_T \times \R^d\times (0,1)^2} | A(u(t,x,\alpha)) - A( v(t,y,\beta)) |\mathscr{L}^r_\lambda[ \varphi(t,\cdot)](x) \rho_m(x-y) \,d\alpha\,d\beta\,dx\,dy\,dt\bigg]\\
&-\E \bigg[\int_{Q_T \times \R^d \times (0,1)^2} | A(v(t,y,\beta))-A(u(t,x,\alpha))| \mathscr{L}_{\lambda,r}[ \rho_m(x-\cdot)](y)\varphi(t,x) \,d\beta\,d\alpha\,dy\,dx\,dt \bigg]\\
& -\E \bigg[\int_{Q_T\times \R^d\times (0,1)^2} | A(u(t,x,\alpha))-A(v(t,y,\beta))| \mathscr{L}_{\lambda,r}[ \varphi(t,\cdot) \rho_m(\cdot-y)](x)  \,d\beta\,d\alpha\,dy\,dx\,dt \bigg].
\end{align*}
\noindent {\bf Step 1} (Passing to the limit as $ r \to 0$):
First note that (cf. \cite{CifaniJakobsen})
\[
|\mathscr{L}_{\lambda,r}[\varphi](x)|\le 
\begin{cases}
c_{\lambda}\|D\varphi\|_{L^\infty} \int_{|z|\le r} \frac{|z|}{|z|^{d+2\lambda}}\,dz, & \lambda \in (0,1/2) \\[2mm]
\frac{c_{\lambda}}{2}\|D^2\varphi\|_{L^\infty}\int_{|z|\le r} \frac{|z|^2}{|z|^{d+2\lambda}}\,dz, & \lambda \in [1/2,1)
\end{cases}
\]
Thus we see that in both cases $|\mathscr{L}_{\lambda,r}[\varphi](x)| \le c r^s$ for some $s>0$ and  $\lim_{r\to 0}|\mathscr{L}_{\lambda,r}[\varphi](x)|=0$.
On the other hand, since supp$\varphi(t,\cdot) \subset K$ for any $t$, assuming $r+1/m<1$ one gets 
\begin{align*}
\bigg|\mathscr{L}_{\lambda,r} \big[ \varphi(t, \cdot)\rho_m (\cdot -y)\big](x)\bigg|=&c_\lambda \bigg|\int_{|z|\le r} \frac{\varphi(t,x)\rho_m(x-y)-\varphi(t,x+z)\rho_m(x+z-y)}{|z|^{d+2\lambda}}\,dz \bigg| \\[1mm]
\leq& C(\varphi,\rho_m) r^s 1_{K + \bar B(0,1)}(x)1_{K + \bar B(0,1)}(y),
\\[2mm]
\bigg| \mathscr{L}_{\lambda,r} \big[ \rho_m (x -\cdot)\big](y)\,\varphi(t,x) \bigg|
=& \bigg| \varphi(t,x) \int_{|z|\le r} \frac{\rho_m(x-y)-\rho_m(x-y-z)}{|z|^{d+2\lambda}}\,dz\bigg| \\[1mm]
\leq& C(\varphi,\rho_m) r^s 1_{K}(x)1_{K + \bar B(0,1)}(y).
\end{align*}
Therefore
\begin{align*}
& \E \bigg[\int_{\R^d} \int_{Q_T} \int_{0}^1\int_{0}^1  |A(u(t,x, \alpha)) -A(v(t,y, \beta)) | \,\mathscr{L}_{\lambda,r} \big[ \varphi(t, \cdot)\rho_m (\cdot -y)\big](x) \,d\alpha \, d\beta \,dx\,dt \,dy \bigg] \notag\\ 
&\quad +\E \bigg[\int_{\R^d} \int_{Q_T} \int_{0}^1 \int_{0}^1 |A(v(t,y,\beta)) - A(u(t,x,\alpha))| \,\mathscr{L}_{\lambda,r} \big[ \rho_m (x -\cdot)\big](y)\,\varphi(t,x) \,d\alpha \, d\beta \,dx\,dt \,dy \bigg]\underset{ r \to 0} \longrightarrow 0.
\end{align*}
On the other hand, using $\fr[\varphi] := \mathscr{L}_{\lambda, r}[\varphi] + \mathscr{L}_{\lambda}^{r}[\varphi]$, we conclude using similar argument
\begin{align*}
&\E \bigg[\int_{\R^d} \int_{Q_T} \int_{0}^1\int^1_{0}\fr^r[\varphi(t,\cdot)](x) \rho_m(x-y) |A(u(t,x,\alpha))-A(v(t,y,\beta))| \,d\alpha\, d\beta \,dy\,dx\,dt \bigg]\\
&=\E \bigg[\int_{\R^d} \int_{Q_T} \int_{0}^1\int^1_{0}(\fr[\varphi(t,\cdot)](x)- \mathscr{L}_{\lambda, r}[\varphi(t,\cdot)](x)) \rho_m(x-y)|A(u(t,x,\alpha))- A(v(t,y,\beta))| \,d\alpha\, d\beta \,dy\,dx\,dt \bigg]\\
&\qquad \underset{ r\to 0} \longrightarrow \E \bigg[\int_{\R^d} \int_{Q_T} \int_{0}^1\int_{0}^1 
|A(u(t,x, \alpha)) - A(v(t,y, \beta))| \, \fr[\varphi(t,\cdot)](x)\, \rho_m(x-y)  \,d\alpha \,d\beta \,dx\,dt \,dy \bigg].
\end{align*}
\noindent {\bf Step 2} (Passing to the limit as $ m \to \infty $):
Let us consider  
\begin{align*}
\mathcal{A}_2:= &\Bigg| \E \bigg[\int_{\R^d} \int_{Q_T} \int_{0}^1\int_{0}^1 
|A(u(t,x, \alpha)) - A(v(t,y, \beta))| \, \fr[\varphi(t,\cdot)](x)\, \rho_m (x-y)  \,d\alpha \,d\beta \,dx\,dt \,dy \bigg]\\
& - \E \bigg[\int_{Q_T} \int_{0}^1\int_{0}^1 
|A(u(t,x, \alpha)) -A( v(t,x, \beta))| \, \fr[\varphi(t, \cdot)](x)\,d\alpha \,d\beta \,dx\,dt  \bigg]\Bigg|\\
& \le \E \bigg[\int_{\R^d} \int_{Q_T} \int_{0}^1\int_{0}^1\Big| |A(u(t,x,\alpha))- A(v(t,y,\beta))|-|A(u(t,x,\alpha))-A(v(t,x,\beta))|\Big|\\
&\hspace{7cm}\times |\fr[\varphi(t, \cdot)](x)| \rho_m(x-y) \,d\alpha \, d\beta \,dx \,dt \,dy \bigg]\\
& \le \E\bigg[\int_{\R^d} \int_{Q_T}\int_{0}^1|A( v(t,y,\beta))-A(v(t,x,\beta))| |\fr[\varphi(t, \cdot)](x)| \rho_m(x-y) \,d\beta \,dx\,dy\,dt\bigg]\\
&\le  \E\bigg[\int_{\R^d} \int_{Q_T} |\mathscr{L}[ \varphi(t, \cdot)](x)|^2 \rho_m(x-y)\,dx\,dy\,dt\bigg]^{1/2}\\
& \hspace{2cm} \times \,\E\bigg[\int_{\R^d} \int_{Q_T} \int_{0}^1| v(t,y,\beta)-v(t,x,\beta)|^2 \rho_m(x-y) \,d\beta \,dx\,dy\,dt\bigg]^{1/2}\underset{m \to \infty} \longrightarrow 0.
\end{align*}
This finishes the proof.
\end{proof}
\end{lem}
Using lemma \ref{kato_lemma1}, lemma \ref{kato_lemma2} and lemma \ref{kato_lemma3} we have the expected Kato inequality 
\begin{align}
0 \leq& \int_\D |u_0-v_0|\varphi(0)\,dx + \E\bigg[\int_{Q_T} \int^1_{0}\int^1_{0}  |u(t,x,\alpha) -v(t,x,\beta)| \partial_t \varphi(t,x)\,d\alpha\,d\beta\,dx\,dt\bigg]\label{kato}\\
&\hspace{3.5cm} -\E\bigg[\int_{Q_T} \int^1_{0}\int^1_{0}  F(u(t,x,\alpha),v(t,x,\beta))\nabla \varphi(t,x)\,dx\,d\alpha d\beta\,dt 
\bigg]\notag \\
&\hspace{4cm}-\E\bigg[\int_{Q_T} \int^1_{0} \int^1_{0} |A(u(t,x,\alpha))-A(v(t,x,\beta))|  \mathscr{L}_\lambda[\varphi(t,\cdot)](x) \,dx\,d\alpha d\beta\,dt\bigg],\notag 
\end{align}
\textit{a priori} for any non-negative $\varphi \in \mathcal{D}([0,T[\times\D)$, but for any non-negative $\varphi \in L^2(0,T,H^2(\D))\cap H^1(Q_T)$ by a density argument.

\begin{rem}\label{remnonvisc}
An important remark here is to mention that the proofs of this section are using the $L^2$-regularity of $u_\theta$. Thus, the same results hold also if one assumes that $u_\theta$ is just an entropy solution and not a viscous one: \textit{i.e.} if $\theta=0$.
\end{rem}

\subsection{Well-posedness}
\label{Well-posedness}

\subsubsection{Uniqueness of (Measure-valued (mild)) Solution}
\label{Uniqueness of (Measure-Valued) Entropy Solution}

We propose to follow closely the idea developed by Endal \& Jakobsen in \cite{EndalJakobsen}. 
Let us note first that \eqref{kato} yields, for a suitable regular test-function $\varphi$, 
\begin{align*}
0 \leq& \int_\D |u_0-v_0|\varphi(0)\,dx 
\\ &+ \E\bigg(\int_{Q_T\times(0,1)^2}\hspace{-0,7cm}  |u(t,x,\alpha) -v(t,x,\beta)| \bigg(\partial_t \varphi(t,x) + \|f^\prime\|_\infty |\nabla \varphi|(t,x)+\|A^\prime\|_\infty  \big[-\mathscr{L}_\lambda[\varphi(t,\cdot)]\big]^+(x)\bigg) \,dx\,d\alpha d\beta\,dt.
\end{align*}

We recall the following information about the way to choose $\varphi$:
\\[0.1cm]
Let $\Phi$ be the unique viscosity solution of $\partial_t \Phi - (-\mathscr{L}_\lambda \Phi)^+=0$ with initial condition $\Phi_0 \in \mathcal{D}^+(\R^d)$ satisfying $0 \leq \Phi \in C([0,T],L^1(\D))$ (see \cite[Lemma 2.6]{EndalJakobsen}). For  a space-time  mollifier $\rho_\delta$, define $\Phi_\delta:= ( \Phi \star \rho_\delta)(x,t)$, for $0<\tau<T$, and $K_\delta:=\Phi_\delta(x,\|A'\|_\infty(\tau-t))$. One has that $\partial_t K_\delta + \|A^\prime\|_\infty \big[-\mathscr{L}_\lambda(K_\delta)\big]^+ \leq 0$ in $ Q_T$, with $0 \le K_\delta \in C([0,T],L^1(\D))\cap C^\infty(Q_T) \cap L^\infty(Q_T)$ ( \cite[ Corollary 5.8]{EndalJakobsen}).
\\
 Set $\gamma_{\tilde\delta}(t,x)=(1_{(-\infty,R)}\star\rho_\ep)(\sqrt{\tilde\delta+|x-x_0|^2}+\|f^\prime\|_{\infty}t)$ (see \cite[Lemma 5.9]{EndalJakobsen}) where $x_0 \in \D$ is given, $R$ is big and $\rho_\ep$ is a mollifier sequence in $\mathcal{D}^+(Q_T)$ such that $\partial_t \gamma_{\tilde\delta} + \|f^\prime\|_{\infty}|\nabla \gamma_{\tilde\delta}| \leq 0$ and $\gamma_{\tilde \delta} \in C^\infty_c(Q_T)$.
\\
 Then, $\Gamma = K_\delta \star_x \gamma_{\tilde\delta}(t,x)$ satisfies $\partial_t \Gamma + \|f^\prime\|_\infty |\nabla \Gamma|+\|A^\prime\|_\infty  \big[-\mathscr{L}_\lambda(\Gamma)\big]^+ \leq 0$ (see \cite[Lemma 5.2]{EndalJakobsen}).
\medskip

Assuming $u_0=v_0$ and setting $\varphi(t,x)=\theta(t)\Gamma(t,x)$ where $\theta(t) = 1-\frac{t}{T}$ for $ t \in [0,T]$ yield 
\begin{align*}
\E \bigg[\int_{Q_T\times(0,1)^2}\hspace{-0,7cm}  |u(t,x,\alpha) -v(t,x,\beta)| (-1/T) \Gamma(t,x) \,dx\,d\alpha d\beta\,dt \bigg] \ge  0,
\end{align*}
and, sending $R$ to $+\infty$, Fatou's lemma gives,
\begin{align*}
\E \bigg[\int_{Q_T\times(0,1)^2}\hspace{-0,7cm}  |u(t,x,\alpha) -v(t,x,\beta)|  \,dx\,d\alpha d\beta\,  \|K_\delta(t)\|_{L^1(\D)}\,dt \bigg] \leq 0.
\end{align*}
Note that $\|K_\delta(\tau)\|_{L^1(\D)}=\|\Phi_\delta(0)\|_{L^1(\D)}$ and 
\begin{align*}
\|\Phi_\delta(0)\|_{L^1(\D)} =& \int_\D \Phi_\delta(x,0)dx=\int_{\R^{2d+1}}\Phi(y,s)\rho_\delta(y-x,s)dydxds=\int_{\R}\|\Phi(s)\|_{L^1(\D)}\int_\D\rho_\delta(z,s)dzds
\\ \geq & \frac12 \|\Phi_0\|_{L^1(\D)}, \text{ for any $\delta$ less than a given $\delta_0$ since $\Phi \in C([0,T],L^1(\D))$.}
\end{align*}
As $K_\delta \in C([0,T],L^1(\D))$ we conclude that $\|K_\delta(t)\|_{L^1(\D)}>0$ in a neighborhood of $\tau$ and thus \\
$\E\bigg[\dint_{\D\times(0,1)^2}\hspace{-0,7cm}  |u(t,x,\alpha) -v(t,x,\beta)|  \,dx\,d\alpha d\beta\bigg]=0$, a.e. in this neighborhood of $\tau$. Since $\tau$ is arbitrary in $(0,T)$,  we conclude that $$\E\bigg[ \dint_{\D\times(0,1)^2}\hspace{-0,7cm}  |u(t,x,\alpha) -v(t,x,\beta)|  \,dx\,d\alpha d\beta\bigg]=0, \,\text{a.e. in }  (0,T).$$
\\
This ensures the uniqueness of the measure valued (mild) solution coming from a viscous regularization. Moreover, the above equality also implies that this unique measure valued (mild) solution is independent of its additional variable $\alpha$ or $\beta$. On the other hand, we conclude that the whole sequence of viscous approximation converges weakly in $L^2(\Omega \times Q_T)$. Since the limit process is independent of the additional variable, the viscous approximation converges strongly in $L^p(\Omega \times (0,T)\times B(0,M))$, for any $M>0$ and any $1\le p <2 $.

\subsubsection{Existence of an Entropy Solution}
\label{entropy_existence}

As a by-product of the above convergence and the \textit{a priori} estimates, one can prove the result of existence of an entropy solution in the sense of Definition~\ref{Defi_Entropy_formulation}. Consider $u$ the limit of the viscous approximation $(u_\eps)$. Thanks to the \textit{a priori} estimates,  $u$ satisfies the regularity of a solution with the additional information that $A(u) \in L^2(\Omega\times(0,T),H^\lambda(\R^d))$. Then, for any $\psi \in \mathcal{D}^{+}(Q_T)$, any pair of entropy-entropy flux pair $(\eta, \zeta)$, with $\eta$ convex, and for any $\mathbb{P}$-measurable set $B$, one needs to pass to the limit in the following inequality satisfied by  $u_{\eps}$:
\begin{align*}
0 & \le  \E \Big[\mathds{1}_{B}\, \hspace{-0.2cm} \int_{\R^d} \eta(u_{\eps}(0,x)-k)\, \varphi(0,x)\,dx \Big]+ 
 \E \Big[\mathds{1}_{B}\, \hspace{-0.2cm}\int_{Q_T} \hspace{-0.2cm} \Big\{\eta(u_{\eps}(t,x)-k) \,\partial_t\varphi(t,x) -  \grad \varphi(t,x)\cdot \zeta(u_{\eps}(t,x)) \Big\}dx\,dt \Big] \\
 & \quad + \E \Big[\mathds{1}_{B}\, \sum_{k\ge 1}\int_{Q_T} g_k(u_\eps(t,x))\eta^\prime (u_\eps(t,x)-k)\varphi(t,x)\,d\beta_k(t)\,dx \Big] \notag \\
 & \quad \quad + \frac{1}{2} \E \Big[\mathds{1}_{B}\sum_{k\ge 1}\, \int_{Q_T}\mathbb{G}^2(u_\eps(t,x))\eta^{\prime\prime} (u_\eps(t,x)-k)\varphi(t,x)\,dx\,dt \Big] + \mathcal{O}(\eps) \notag \\
 & \quad -  \E \Big[\mathds{1}_{B}\,\int_{Q_T} \Big\{ \mathscr{L}^r_\lambda[A(u_\eps(t,\cdot))](x) \varphi(t,x) \eta'(u_\eps(t,x)-k) + A^\eta_k(u_\eps(t,x)) \mathscr{L}_{\lambda,r}[\varphi(t,\cdot)](x) \Big\}\,dx\,dt \Big]. 
\end{align*} 
For that, we may use the same arguments as in \cite{BaVaWit_2012,BaVaWitParab} and \cite[Subsection 3.4]{BhKoleyVa}. We leave the details to the interested reader.


\subsubsection{Uniqueness of Entropy Solution}
\label{Unique}

As alluded to before in Remark \ref{remnonvisc}, Kato's inequality \eqref{kato} is also available if $v$ is any stochastic entropy solution in the sense of Definition \ref{Defi_Entropy_formulation} and $u$  a weakly converging limit of the sequence of viscous solutions. Then, as depicted in subsection~\ref{Kato} and subsection~\ref{Uniqueness of (Measure-Valued) Entropy Solution}, one gets, $t$ a.e. the following equality
\begin{align*}
\E \bigg[\dint_{\R^d} |u(t,x) -v(t,x)|  \,dx \bigg]=0.
\end{align*}
This confirms the uniqueness of the stochastic entropy solution. Indeed, any stochastic entropy solution is equal to the limit point of the viscous approximation.
\\[0.1cm]
Note that thanks to the \textit{a priori} estimates and Remark~\ref{weakcont}, $u$ and $v$ are also in $C_w([0,T],L^2(\Omega\times\R^d))$. Denote by $\mathcal{Z}\subset (0,T)$ a set of full measure such that $\E \bigg[\dint_{\R^d} |u(s,x) -v(s,x)|  \,dx \bigg]=0$, for any $s\in \mathcal{Z}$. Let $t \in (0,T)$ and $(s_n) \subset\mathcal{Z}$ such that $s_n \to t$. Since  $u(s_n,x) -v(s_n,x) \rightharpoonup u(t,x) -v(t,x)$ in $L^2(\Omega\times\R^d)$, and since the $L^1$ norm is a lower semi-continuous convex function on $L^2$, $\E \bigg[\dint_{\R^d} |u(t,x) -v(t,x)|  \,dx \bigg]=0$ for any $t \in (0,T)$.

\subsubsection{Uniqueness: stability}
\label{stability}
Thanks to the uniqueness result, it is then known that any entropy solution $u$ is stemmed from the sequence of viscosity solutions $u_\ep$ and that $A(u) \in L^2(\Omega\times(0,T),H^\lambda(\D))$. Thus, we can recast Kato's inequality as
\begin{align*}
0 \leq& \int_\D |u_0-v_0|\varphi(0)\,dx + \E\Big[\int_{Q_T} |u(t,x) -v(t,x)| \partial_t \varphi(t,x) - F(u(t,x),v(t,x))\nabla \varphi(t,x)\,dx\,dt\Big] 
\\&
-\E \bigg[\int_{Q_T}   \mathscr{L}_{\lambda/2}\big[|A(u(t,\cdot))-A(v(t,\cdot))|\big](x)  \mathscr{L}_{\lambda/2}[\varphi(t,\cdot)](x) \,dx\,dt \bigg].
\end{align*}
To achieve the stability of entropy solution with respect to its initial data, we again follow the argument depicted in \cite[Subsection 3.6]{BhKoleyVa} to conclude that for any $t$, 
\begin{align*}
\E\bigg[\int_{\D} |u(t,x) -v(t,x)| \,dx\bigg] \leq& \int_\D |u_0-v_0|\,dx.
\end{align*}

\section{Proof of Theorem \ref{continuous-dependence}: Continuous Dependence on Nonlinearity}
\label{ContDep_NonL}

In this section, we establish the continuous dependence estimates on the given data. However, as we mentioned earlier, we are only concerned with the continuous dependence on the nonlinearity $A$ coming from the fractional diffusion, as the continuous dependence on the other parameters is well studied (see \cite[Section 4]{BhKoleyVa}). In what follows, 
for $\eps>0$, let $v_\eps$ be a weak solution to the problem 
\begin{align}\label{eq:viscous}
dv_\eps(s,y) - \eps \Delta v_\eps(s,y)\,ds + \mathscr{L}_{\lambda}[B(v_\eps(s, \cdot))](y)\,ds  & - \mbox{div}_y f(v_\eps(s,y)) \,ds = \h(v_\eps(s,y))\,dW(s),   \\
v_\eps(0,y)&=v_0^{\eps}(y) \notag.
\end{align}
In view of Theorem \ref{thm:existence-bv}, we conclude that $v_\eps$ converges to the unique BV-entropy solution $v$ of \eqref{eq:stoc_frac} with initial data $v_0 \in BV(\D)$. Let $u$ be the unique $L^1$-entropy solution of  \eqref{eq:stoc_frac} with initial data $u_0\in L^1(\D)$. Note that, by the $L^1$-entropy solution we refer to the entropy solution corresponding to initial data $u_0\in L^1 (\R^d)\cap L^2(\D)$. We assume that the \ref{A1}, \ref{A2}, \ref{A3} and \ref{A4}  hold for both sets of given functions $(u_0, f,A, h, \lambda)$ and $(v_0, f,B, h , \lambda)$ with additionally $f^{\prime\prime}$ bounded.

We shall estimate the average $L^1$-difference between two $L^1$-entropy/BV-entropy solutions $u$ and $v$. To achieve this, we shall make use of the ``{\it doubling of variables}" technique. However, as in the proof of Kato's inequality, we can not directly compare two entropy solutions $u$ and $v$, but instead we first compare the entropy solution $u$ with the regularized version of the viscous approximation \eqref{eq:viscous}, \textit{i.e.}, $v^{\gamma}_{\eps}$. This approach is somewhat different from the deterministic approach, where one can directly compare two entropy solutions. 

For technical purpose (as observed in \cite{Alibaud_one}), we need to consider the following partition of $\R$ based on the region $A'\ge B'$ and its complementary. Let $E_{\pm}$ be sets satisfying
\begin{align*}
\left\{\begin{array}{l}{E_{ \pm} \subseteq \mathbb{R} \text { are Borel sets, }} \\ {\cup_{ \pm} E^{ \pm}=\mathbb{R} \text { and } \cap_{ \pm} E_{ \pm}=\emptyset,} \\ {\mathbb{R} \backslash \operatorname{supp}\left(A^{\prime}-B^{\prime}\right)^{\mp} \subseteq E_{ \pm}.}\end{array}\right. 
\end{align*}
For all $u \in \R$, we define 
\begin{align*}
\begin{aligned} A_{ \pm}(u)  :=\int_{0}^{u} A^{\prime}(\tau) \mathbf{1}_{E_{ \pm}}(\tau) \mathrm{d} \tau, 
\quad
B_{ \pm}(u)  :=\int_{0}^{u} B^{\prime}(\tau) \mathbf{1}_{E_{ \pm}}(\tau) \mathrm{d} \tau, 
\quad  
C_{ \pm}(u)  :=\pm\left(A_{ \pm}(u)-B_{ \pm}(u)\right).
\end{aligned}
\end{align*}
It is easy to see (for details consult \cite{Alibaud_one}) that the above functions satisfy
\begin{equation}
\label{two}
\begin{cases}
A=A_++A_- \text{ and } B=B_+ +B_-,\\
A_{\pm}, B_{\pm}, C_{\pm} \text{ satisfy } \eqref{A3},\\
|C_{\pm} (u)|_{BV} \le \| A'- B'\|_\infty \| u\|_{BV},\\
\sum_{\pm} \| C_{\pm}(u(\cdot+z))- C_\pm(u) \|_{L^1(\D)} \le \| A'-B' \|_\infty \| u(\cdot +z, \cdot) -u \|_{L^1(\D)}.\\
\end{cases}
\end{equation}
Next, for a nonnegative test function $\varphi\in C_c^{1,2}([0,\infty)\times \rd)$, we define the same test function as in \eqref{test_function}          
\begin{align}
\psi (t,x, s,y) = \rho_n(t-s) 
\,\rho_m(x-y) \,\varphi(t,x).
\end{align}

As in the proof of Kato's inequality, we first apply the It\^{o}'s formula to $ \eta(v^\gamma_\eps(s,y)-k) \psi(t,x,s,y)$, and then multiply with the test function $\rho_l(u(t,x)-k)$ and integrate with respect to $x,t,k$. Taking the expectation we get 
\begin{align}
 0\le  &\, \E \Big[\int_{Q_T}\int_{\R^d}\int_{\R} 
 \eta(v^\gamma_\eps(0,y)-k)\psi(t,x,0,y) \rho_l(u(t,x)-k)\,dk \,dy\,dx\,dt\Big] \notag \\
   & +   \E \Big[\int_{Q_T} \int_{Q_T} \int_{\R} 
 \eta(v^\gamma_\eps(s,y)-k)\partial_s \psi(t,x,s,y)
 \rho_l(u(t,x)-k)\,dk \,dy\,ds\,dx\,dt\Big] \notag \\ 
 & + \E \Big[\int_{Q_T} \int_{Q_T}
\int_{\R} \eta^\prime (v^\gamma_\eps(s,y)-k)\, \psi(t,x,s,y) \h(v_\eps)^\gamma(s,y)\,dk\,dW(s) \rho_l(u(t,x)-k)\,dy\,ds\,dx\,dt \Big] \notag \\
 &+ \frac{1}{2}\, \E \Big[ \int_{Q_T} \int_{Q_T} 
\int_{\R} \eta^{\prime\prime} (v^\gamma_\eps(s,y) -k)(\h(v_\eps)^\gamma(s,y))^2\,\psi(t,x,s,y) \rho_l(u(t,x)-k)\,dk
\,dy\,ds\,dx\,dt \Big] \notag \\
   & -  \E\Big[\int_{Q_T}\int_{Q_T} \int_{\R}  
 f^\eta(v^\gamma_\eps(s,y),k)\cdot \grad_y \psi(t,x,s,y)\, \rho_l(u(t,x)-k)\,dk\,dy\,ds\,dx\,dt\Big] \notag \\
 &  - \eps  \,\E \Big[\int_{Q_T} \int_{Q_T} \int_{\R} 
 \eta^\prime(v^\gamma_\eps(s,y)-k)\grad_y v^\gamma_\eps(s,y) \cdot \grad_y  \psi(t,x,s,y)
  \rho_l(u(t,x)-k)\,dk\,dy\,ds\,dx\,dt\Big]  \notag \\
  & - \E \bigg[\int_{Q_T} \int_{Q_T} 
\int_{\R} \mathscr{L}^r_{\lambda}[B(v_\eps)^\gamma(s, \cdot)](y) \, \psi(t,x,s,y)\, \eta^\prime (v^\gamma_\eps(s,y)-k)\rho_l(u(t,x)-k)\,dk\,dy\,ds\,dx\,dt \bigg]  \notag \\
& -\E \bigg[\int_{Q_T} \int^T_0 
\int_{\R} \mathscr{L}_{\lambda,r}[B(v_\ep)^\gamma(s,\cdot)](y)\, \psi(t,x,s,y)\, \eta'(v^\gamma_\ep(s,y)-k) \, \rho_l(u(t,x)-k)\,dk\,ds\,dx\,dt \bigg] \notag \\[2mm]
& =: I_1 + I_2 + I_3 + I_4 + I_5 + I_6 + I_7 + I_8. \label{continuous_dep_2}
\end{align}

We now write the entropy inequality for $u(t,x)$, based on the 
entropy pair $(\eta(\cdot-k), f^\eta(\cdot, k))$, and 
then multiply by $\rho_l(v^\gamma_\eps(s,y)-k)$, integrate with 
respect to $ s, y, k$ and take the expectation. The result is
\begin{align}
0\le  & \E \Big[\int_{Q_T}\int_{\R^d}\int_{\R} \eta(u(0,x)-k)
\psi(0,x,s,y) \rho_l(v^\gamma_\eps(s,y)-k)\,dk \,dx\,dy\,ds\Big] \notag \\
 & + \E \Big[\int_{Q_T} \int_{Q_T} \int_{\R} \eta(u(t,x)-k)\partial_t \psi(t,x,s,y)
\rho_l(v^\gamma_\eps(s,y)-k)\,dk \,dx\,dt\,dy\,ds \Big]\notag \\ 
& + \E \Big[\int_{Q_T} \int_{Q_T} 
\int_{\R} \eta^\prime (u(t,x)-k)\h(u(t,x))\, \psi(t,x,s,y)\rho_l(v^\gamma_\eps(s,y)-k)\, dk \,dW(t) \,dx\,dy\,ds \Big] \notag \\
 &+ \frac{1}{2}\, \E \Big[ \int_{Q_T} \int_{Q_T}  
\int_{\R}  \eta^{\prime\prime} (u(t,x) -k)h^2(u(t,x))\, \psi(t,x,s,y)\rho_l(v^\gamma_\eps(s,y)-k) \,dx\,dt\,dk\,dy\,ds \Big] \notag \\
& -  \E \Big[\int_{Q_T}\int_{Q_T} \int_{\R} 
 f^\eta(u(t,x),k) \cdot \grad_x \psi(t,x,s,y) \, \rho_l(v^\gamma_\eps(s,y)-k)\,dk\,dx\,dt\,dy\,ds\Big] \notag \\
 & - \E \Big[\int_{Q_T} \int_{Q_T} 
\int_{\R} \mathscr{L}^r_{\lambda}[A(u(t,\cdot))](x)\, \psi(t,x,s,y)\, \eta'(u(t,x) -k)
\, \rho_l(v^\gamma_\eps(s,y)-k)\,dk\,dx\,dt\,dy\,ds \Big] \notag \\
& -  \E \Big[\int_{Q_T} \int_{Q_T} 
\int_{\R} A^\eta_k(u(t,x)) \,\mathscr{L}_{\lambda,r} [\psi(t,\cdot,s,y)](x) \rho_l( v^\gamma_\eps(s,y)-k)\,dk\,dx\,dt\,dy\,ds \Big]\notag \\[2mm]
& =:  J_1 + J_2 + J_3 +J_4 + J_5 + J_6 + J_7. \label{continuous_dep_1}
\end{align}
Our aim is to add \eqref{continuous_dep_1} and \eqref{continuous_dep_2}, 
and pass to the  limits with respect to the various parameters involved. We do this by claiming a series of lemmas and some of the proofs of these lemmas follow from the proof of Kato's inequality in Section \ref{kato} (see Remark~\ref{remnonvisc}) and \cite{BaVaWitParab,BisMajKarl_2014} modulo cosmetic changes. 

To begin with, note that the particular choice of test function \eqref{test_function} implies that $I_1=0$. 
\begin{lem}
\label{stochastic_lemma_1}
It holds that 
\begin{align*}
 I_1 + J_1 & \underset{n \goto \infty} \longrightarrow  \int_{\R^d}\int_{\R^d}\int_{\R} 
  \eta(u(0,x)-k)\varphi(0,x)\rho_m(x-y) \rho_l(v^\gamma_\eps(0,y)-k)\,dk\,dx\,dy
  \\
  & \underset{\gamma \goto 0} \longrightarrow  \int_{\R^d}\int_{\R^d}\int_{\R} 
  \eta(u(0,x)-k)\varphi(0,x)\rho_m(x-y) \rho_l(v_\eps(0,y)-k)\,dk\,dx\,dy
  \\
  &\underset{\delta \to 0} \longrightarrow   \int_{\R^d}\int_{\R^d} \int_\R
  |u(0,x)-k|\varphi(0,x)\rho_m(x-y)\rho_l(v_\ep(0,y)-k)\,dk\,dx\,dy
  \\
  &\underset{l \to 0} \longrightarrow   \int_{\R^d}\int_{\R^d} 
  |u(0,x)-v_\ep(0,y)|\varphi(0,x)\rho_m(x-y)\,dx\,dy.
 \end{align*}
\end{lem}
 
\begin{lem}\label{stochastic_lemma_2}
It holds that
\begin{align*}
I_2 + J_2  &\underset{n \goto \infty}\longrightarrow   \E\Big[ \int_{Q_T}\int_{\R^d}\int_{\R}
    \eta(u(t,x)-k) \partial_t\varphi(t,x)\rho_m(x-y)\rho_l(v^\gamma_\eps(t,y)-k)\,dk\,dy\,dx\,dt\Big]
    \\
    &\underset{\gamma \to 0} \longrightarrow   \E\Big[ \int_{Q_T}\int_{\R^d}\int_{\R}
   \eta(u(t,x)-k) \partial_t\varphi(t,x)\rho_m(x-y)\rho_l(v_\eps(t,y)-k)\,dk\,dy\,dx\,dt\Big]
   \\
&\underset{\delta \to 0}\longrightarrow  \E \Big[\int_{Q_T}\int_{\R^d} \int_\R |u(t,x)-k| \partial_t\varphi(t,x) \, \rho_m(x-y)\, \rho_l(v_\ep(t,y)-k)\,dk\,dy\,dx\,dt\Big]
   \\
&\underset{l \to 0}\longrightarrow  \E \Big[\int_{Q_T}\int_{\R^d} |u(t,x)-v_\ep(t,y)| \partial_t\varphi(t,x) \, \rho_m(x-y)\,dy\,dx\,dt\Big].
\end{align*}
\end{lem}
\noindent Next we consider the stochastic terms. Regarding that we have the following result
\begin{lem}
\label{lem:stochastic-terms_01}
The following holds:
\begin{align*}
\lim_{\delta \to 0}& \lim_{\gamma \to 0}\lim_{n \goto \infty}\Big( \big(I_3 + J_3\big) + \big(I_4 + J_4 \big)\Big)
\\
&= \E\Big[ \sum_{k \ge 1} \int_{Q_T}\int_{\R^d} \big(g_k(u(t,x))-g_k(v_\eps(t,y)))^2 \rho_l(u(t,x)-v_\ep(t,y)) \varphi(t,x)\rho_m(x-y)\,dx\,dy\,dt \Big] 
\\
& \le C \E\Big[ \int_{Q_T}\int_{\R^d} |u(t,x)-v_\eps(t,y)|^2 
\varphi(t,x)\rho_m(x-y) \rho_l(u(t,x)-v_\ep(t,y))\,dx\,dy\,dt \Big]
\underset{l \to 0}\longrightarrow  0.
\end{align*}
\end{lem}
\noindent For the terms coming from the flux functions, we have the following lemma.

\begin{lem}\label{stochastic_lemma_4}
\begin{align*}
\lim_{\delta \to 0} \lim_{\gamma \to 0} \lim_{n \to \infty}&( I_5+J_5) \le \|v_0\|_{BV}  l \|f''\|_\infty\int^T_0 \| \varphi(t,\cdot)\|_\infty \,dt
\\
&-\E \left[ \int_{Q_T} \int_\D \int_\R F(u(t,x),v_\ep(t,y)+k). \nabla_x\varphi(t,x)  \rho_l(k)\rho_m(x-y)\,dk\,dx\,dy\,dt\right]
\\&
\underset{l \to 0}\longrightarrow  -\E \left[ \int_{Q_T} \int_\D F(u(t,x),v_\ep(t,y)). \nabla_x\varphi(t,x)  \rho_m(x-y)\,dx\,dy\,dt\right].
\end{align*}
\end{lem} 
\noindent We briefly sketch the proof of the above Lemma~\ref{stochastic_lemma_4}. Indeed, the following hold:
\begin{align*}
&\lim_{\delta \to 0} \lim_{\gamma \to 0} \lim_{n \to \infty}( I_5+J_5)
\\=& 
- \E \left[ \int_{Q_T} \int_\D \int_\R F(u(t,x),k). \nabla_x[ \varphi(t,x) \rho_m(x-y)] \rho_l(v_\ep(t,y)-k)\,dk\,dx\,dy\,dt \right]
\\&
-\E \left[ \int_{Q_T} \int_\D \int_\R F(v_\ep(t,y),k). \nabla_y\rho_m(x-y) \varphi(t,x) \rho_l(u(t,x)-k)\,dk\,dx\,dy\,dt\right]
\\=&
-\E \left[ \int_{Q_T} \int_\D \int_\R F(u(t,x),v_\ep(t,y)+k). \Big[\rho_m(x-y)\nabla_x\varphi(t,x) -\varphi(t,x) \nabla_y\rho_m(x-y) ) \Big]\rho_l(k)\,dk\,dx\,dy\,dt \right]
\\&
-\E \left[ \int_{Q_T} \int_\D \int_\R F(v_\ep(t,y),u(t,x)-k). \nabla_y\rho_m(x-y) \varphi(t,x) \rho_l(k)\,dk\,dx\,dy\,dt\right]
\\=&
-\E \bigg[ \int_{Q_T} \int_\D \int_\R \nabla_y.( F(u(t,x), v_\ep(t,y)+k)- F(v_\ep(t,y),u(t,x)-k))
\varphi(t,x)\rho_m(x-y)\rho_l(k) \,dk \,dx\,dy\,dt \bigg]
\\&
-\E \left[ \int_{Q_T} \int_\D \int_\R F(u(t,x),v_\ep(t,y)+k). \nabla_x \varphi(t,x) \rho_l(k)\rho_m(x-y) \,dk\,dx\,dy\,dt\right]
\\=&
\E \bigg[ \int_{Q_T} \int_\D \int_\R \nabla_y.( F(v_\ep(t,y),u(t,x)-k)-F(v_\ep(t,y)+k,u(t,x) ))
\varphi(t,x)\rho_m(x-y)\rho_l(k) \,dk \,dx\,dy\,dt \bigg]
\\&
-\E \left[ \int_{Q_T} \int_\D \int_\R F(u(t,x),v_\ep(t,y)+k). \nabla_x \varphi(t,x) \rho_l(k)\rho_m(x-y) \,dk\,dx\,dy\,dt\right]
\\=&\E \bigg[ \int_{Q_T} \int_\D \int_\R \nabla_y v_\ep(t,y). \partial_a( F(a,b-k)-F(a+k,b))\bigg|_{(a,b)=(v_\ep(t,y),u(t,x))}
\hspace{-2cm}\varphi(t,x) \rho_m(x-y)\rho_l(k) \,dk\,dx\,dy\,dt\bigg]
\\& 
-\E \left[ \int_{Q_T} \int_\D \int_\R F(u(t,x),v_\ep(t,y)+k). \nabla_x\varphi(t,x)  \rho_l(k)\rho_m(x-y)\,dk\,dx\,dy\,dt\right]\\
\end{align*}
To proceed further we notice that
\begin{align*}
\begin{aligned}\left|\partial_{a}(F(a, b-k)-F(a+k, b))\right|=&\left|\frac{\partial}{\partial u}\left(\int_{b-k}^{a} \operatorname{sgn}(\sigma-(b-k)) f^{\prime}(\sigma) d \sigma-\int_{b}^{a+k} \operatorname{sgn}(\sigma-b) f^{\prime}(\sigma) d \sigma\right)\right| \\ &=\left|\operatorname{sgn}(a-b+k)\left(f^{\prime}(a)-f^{\prime}(a+k)\right)\right| \leq \left\|f^{\prime \prime}\right\|_{\infty}|k|.  \end{aligned}
\end{align*}
Thus we get the lemma thanks to the \textit{a priori} estimates.
Moreover, it is easy to see that $|I_6| \le C(m)\ep$ and we are left with fractional terms. 
To deal with these terms, we follow closely the uniqueness proof in Section~\ref{uniqueness}. In particular, following Lemma~\ref{kato_lemma2}, and Lemma~\ref{kato_lemma3}, we conclude
\begin{lem}
\label{fractional_lemma_4}
The following holds:
\begin{align*}
\limsup_{\delta \goto 0} & \lim_{\gamma \to 0}\lim_{n\goto \infty}  \big(J_7+I_8 \big) 
\\
\leq& -  \E \bigg[\int_{\R^d} \int_{Q_T}\int_\R | A(u(t,x))-A(k)| \,\mathscr{L}_{\lambda,r} \big[ \varphi(t, \cdot) \rho_m (\cdot -y)\big](x) \rho_l(v_\ep(t,y)-k) \,dk\,dx\,dt \,dy \bigg]
\\ 
& - \E \bigg[\int_{\R^d} \int_{Q_T} \int_\R |B(v_\eps(t,y))-B(k)| \,\mathscr{L}_{\lambda,r} \big[ \rho_m (x -\cdot)\big](y)\,\varphi(t,x) \rho_l(u(t,x)-k) \,dk\,dx\,dt \,dy \bigg]
\\&
\underset{l \to 0}\longrightarrow  
-  \E \bigg[\int_{\R^d} \int_{Q_T} |A(u(t,x))-A(v_\ep(t,y))| \,\mathscr{L}_{\lambda,r} \big[ \varphi(t, \cdot) \rho_m (\cdot -y)\big](x) \,dx\,dt \,dy \bigg]
\\ 
&- \E \bigg[\int_{\R^d} \int_{Q_T} |B(v_\eps(t,y))-B(u(t,x))| \,\mathscr{L}_{\lambda,r} \big[ \rho_m (x -\cdot)\big](y)\,\varphi(t,x) \,dx\,dt \,dy \bigg].
\end{align*} 
\end{lem} 
Now, we are left with the last two terms. To deal with those terms, we make use of the Lemma~\ref{kato_lemma1} to conclude
\begin{align*}
\mathrm{M}:= & \lim_{\delta \to 0}\,\lim_{\gamma \to 0}\lim_{n \to \infty}  \big(I_7+J_6\big)\\ 
&= - \E \bigg[\int_{\R^d} \int_{Q_T} \int_\R \fr^r[A(u(t,\cdot))](x)\,\varphi(t,x) \rho_m(x-y)\, \sgn(u(t,x)-k) \rho_l(v_\ep(t,y)-k) \,dk\,dy\,dx\,dt \Big] \\
&\qquad - \E \bigg[\int_{\R^d} \int_{Q_T}\int_\R \mathscr{L}^r_{\lambda} [B(v_\eps(t,\cdot))](y)\,\varphi(t,x) \rho_m(x-y)\, \sgn(v_{\eps}(t,y)-k) \rho_l(u(t,x)-k)\,dk \,dy\,dx\,dt \Big] \\
&= - \E \bigg[\int_{\R^d} \int_{Q_T}\int_\R \Big( \fr^r[A(u(t,\cdot))](x)- \fr^r[ B(v_\ep(t,\cdot))](y) \Big)\sgn( u(t,x)-v_\ep(t,y)+k) \\[-0.5cm]
&\hspace{10cm} \times \varphi(t,x) \rho_m(x-y) \rho_l(k) \,dk\,dx\,dt\,dy \bigg]
\end{align*}
In order to proceed and estimate  $M$, we first state the following lemma. In what follows, let us denote by $d\mu_\lambda(z):= \frac{dz}{|z|^{d+2\lambda}}$.
\begin{lem}\label{lemma_01}
The following hold:
\begin{align*}
&\E \bigg[\int_{\R^d} \int_{Q_T}\int_\R  \int_{|z| >r}\bigg[ \Big(A(u(t,x+z))- A(u(t,x))\Big) -\Big( A(v(t,y+z))-A(v(t,y))\Big) \bigg] \d\mu_\lambda(z) \\[-0.3cm]
&\hspace{5cm} \times \varphi(t,x) \rho_m(x-y) \rho_l(k)\, \sgn( u(t,x)-v(t,y)+k)\,dk\,dy\,dx\,dt \bigg]\\
&\le - \E \bigg[ \int_\D \int_{Q_T} |A(u(t,x))-A(v(t,y))| \fr^r[ \varphi(t,\cdot)](x) \rho_m(x-y) \,dy\,dx\,dt \bigg] +\frac{C\| A'\|_\infty l}{\lambda r^{2\lambda}}\\
\end{align*}
\end{lem}
\begin{proof}
First notice that for any $k$ in the support of $\rho_l$, 
\begin{align*}
&\Big[\Big(A(u(t,x+z))-A(u(t,x))\Big)-\Big(A(v(t,y+z))-A(v(t,y))\Big)\Big] \sgn(u(t,x)-v(t,y)+k) \\ 
&=\Big[\Big(A(u(t,x+z))-A(v(t,y+z))\Big)-\Big(A(u(t,x))-A(v(t,y))\Big)\Big] \sgn(u(t,x)-v(t,y)+k) \\
& \leq|A(u(t,x+z))-A(v(t,y+z))|-\Big(A(u(t,x))-A(v(t,y)-k)+A(v(t,y)-k)-A(v(t,y))\Big) \\[-0.2cm] & \hspace{11cm}\times\sgn(u(t,x)-(v(t,y)-k)) \\
&  \leq|A(u(t,x+z))-A(v(t,y+z))|-|A(u(t,x))-A(v(t,y)-k)|+\left\|A^{\prime}\right\|_{\infty} l \\ 
& \leq|A(u(t,x+z))-A(v(t,y+z))|-| A(u(t,x))-A(v(t,y))|+2\left\|A^{\prime}\right\|_{\infty} l.
\end{align*}
Therefore
\begin{align*}
&\E \bigg[\int_{\R^d} \int_{Q_T}\int_\R  \int_{|z| >r}\bigg[ \Big(A(u(t,x+z))- A(u(t,x))\Big) -\Big( A(v(t,y+z))-A(v(t,y))\Big) \bigg] \d\mu_\lambda(z) \\[-0.3cm]
&\hspace{5cm} \times \varphi(t,x) \rho_m(x-y) \rho_l(k)\, \sgn( u(t,x)-v(t,y)+k)\,dk\,dy\,dx\,dt \bigg]\\
& \le \E\bigg[\int_\D \int_{Q_T} \int_\R \bigg[\int_{|z|>r} \big(|A(u(t,x+z))-A(v(t,y+z))| - | A(u(t,x))-A(v(t,y))|+ 2l\|A'\|_\infty \big)d \mu_{\lambda}(z)\bigg]\\[-0.2cm]
&\hspace{10cm} \times \varphi(t,x) \rho_m(x-y) \rho_l(k) \,dk\,dy\,dx\,dt\bigg]\\
& \le -\E \bigg[ \int_\D \int_{Q_T} | A(u(t,x)) - A(v(t,y))| \fr^r[ \varphi(t,\cdot)](x) \rho_m(x-y) \,dy\,dx\,dt \bigg]+ \frac{C\|A'\|_\infty l}{\lambda r^{2\lambda}},
\end{align*}
where we have used similar tricks as in Lemma ~\ref{kato_lemma1} and the symmetry of $\rho_m$ to get the last inequality.
\end{proof}

\begin{lem}\label{lemma_02}
For any $k \in \R$, the following hold:
\begin{align*}
\int_\D \sgn( k- u(t,x)) \mathscr{L}^r_\lambda[ A(u(t,\cdot))](x) \varphi(t,x) \,dx \le - \int_\D |A(k) - A(u(t,x))| \mathscr{L}^r_\lambda[ \varphi(t,\cdot)](x) \,dx.
\end{align*}
\end{lem}
\begin{proof}
For a proof of this, see \cite[Lemma 4.3]{Alibaud_one} or \cite[Lemma 4.8]{BhKoleyVa}.
\end{proof}

Now making use of  Lemma~\ref{lemma_01}, we can rewrite 
\begin{align*}
\mathrm{M}& = \sum_{\pm} \E \Big[\int_{\R^d} \int_{Q_T} \int_\R \Big( \mathscr{L}^r_{\lambda} [B_{\pm}(v_\eps(t,\cdot))](y) - \mathscr{L}^r_{\lambda} [A_{\pm}(u(t,\cdot))](x) \Big)\,\varphi(t,x) \rho_m(x-y)\\ 
&\hspace{5cm} \times \sgn(u(t,x)-v_{\eps}(t,y)+k) \rho_l(k)\,dk \,dy\,dx\,dt \Big] \\
& \quad \le \E \Big[\int_{\R^d} \int_{Q_T} \int_\R\Big( \mathscr{L}^r_{\lambda} [B_+(u(t,\cdot))](x) - \mathscr{L}^r_{\lambda} [A_+(u(t,\cdot))](x) \Big)\,\varphi(t,x) \rho_m(x-y)\, \\
& \hspace{5cm} \times \sgn(u(t,x)-v_{\eps}(t,y)+k) \rho_l(k) \,dk\,dy\,dx\,dt \Big] \\
& \quad - \E \Big[\int_{\R^d} \int_{Q_T} |B_+(u(t,x))-B_+(v_\eps(t,y))| \mathscr{L}^r_{\lambda} [\varphi(t,\cdot)](x) \, \rho_m(x-y)\,dy\,dx\,dt \Big] +\frac{C\| B_+'\|_\infty l}{\lambda r^{2\lambda}}\\
& \quad + \E \Big[\int_{\R^d} \int_{Q_T} \Big( \mathscr{L}^r_{\lambda} [B_-(v_\eps(t,\cdot)](y) - \mathscr{L}^r_{\lambda} [A_-(v_\eps(t,\cdot))](y) \Big)\,\varphi(t,x) \rho_m(x-y)\,\\ 
& \hspace{5cm} \times\sgn(u(t,x)-v_{\eps}(t,y)+k) \rho_l(k) \,dk\,dy\,dx\,dt \Big] \\
& \quad - \E \Big[\int_{\R^d} \int_{Q_T} |A_-(u(t,x))-A_-(v_\eps(t,y))|   \mathscr{L}^r_{\lambda} [\varphi(t,\cdot)](x) \, \rho_m(x-y)\,dy\,dx\,dt \Big] +\frac{C\| A_-'\|_\infty l}{\lambda r^{2\lambda}}\\
& \quad = \E \Big[\int_{\R^d} \int_{Q_T}\int_\R \mathscr{L}^r_{\lambda} [C_+(u(t,\cdot))](x)\,\varphi(t,x) \rho_m(x-y)\, \sgn(v_{\eps}(t,y) -k-u(t,x)) \rho_l(k) \,dk \,dy\,dx\,dt \Big] \\
& \quad + \E \Big[\int_{\R^d} \int_{Q_T}\int_\R \mathscr{L}^r_{\lambda} [C_-(v_\eps(t,\cdot))](y)\,\varphi(t,x) \rho_m(x-y)\, 
\sgn(u(t,x)+k-v_{\eps}(t,y)) \rho_l(k) \,dk\,dy\,dx\,dt \Big] \\
& \quad - \E \Big[\int_{\R^d} \int_{Q_T} |B_+(u(t,x))-B_+(v_\eps(t,y)| \mathscr{L}^r_{\lambda} [\varphi(t,\cdot)](x) \, \rho_m(x-y)\,dy\,dx\,dt \Big] +\frac{C\| B_+'\|_\infty l}{\lambda r^{2\lambda}}\\
& \quad - \E \Big[\int_{\R^d} \int_{Q_T} |A_-(u(t,x))-A_-(v_\eps(t,y))|  \mathscr{L}^r_{\lambda} [\varphi(t,\cdot)](x) \, \rho_m(x-y)\,dy\,dx\,dt \Big] +\frac{C\| A_-'\|_\infty l}{\lambda r^{2\lambda}}\\[2mm]
&\quad := \mathrm{M}_1 + \mathrm{M}_2 + \mathrm{M}_3 + \mathrm{M}_4.
\end{align*}

To deal with the above terms, first for any measure $\Gamma$, we let $\Gamma^1 = \Gamma|_{\{0 <|z|\leq r_1\}}$ and write $\Gamma = \Gamma^1 + \Gamma|_{\{|z|>r_1\}}$ for $r_1 >r$. Therefore, denoting by  $\mathscr{L}^{r,r_1}_{\lambda}$ the non-local operator integrated over $\{r <|z| \leq r_1\}$ and using the proof of Lemma~\ref{lemma_02}, we have
\begin{align*}
\mathrm{M}_1 & \le -\E \Big[\int_{\R^d} \int_{Q_T}\int_\R |C_+(v_{\eps}(t,y)-k) -C_+(u(t,x))|
\mathscr{L}^{r,r_1}_{\lambda} \big[\varphi(\cdot,t) \rho_m(\cdot-y) \big](x)\rho_l(k) \,dk\,dy\,dx\,dt \Big] \\
& + \E \Big[\int_{\R^d} \int_{Q_T} \int_\R\mathscr{L}^{r_1}_{\lambda} [C_+(u(t,\cdot))](x)\,\varphi(t,x) \rho_m(x-y)\, \sgn(v_{\eps}(t,y)-k -u(t,x)) \rho_l(k) \,dk \,dy\,dx\,dt \Big] 
\end{align*}
To estimate the first term of the above inequality, we note that by Taylor's expansion
\begin{align*}
\mathcal{K}:=& \E \Big[\int_{\R^d} \int_{Q_T} \int_\R  |C_+(v_\eps(t,y)-k) - C_+(u(t,x))| \,\mathscr{L}^{r,r_1}_{\lambda} \big[ \varphi(\cdot,t) \rho_m(\cdot-y) \big](x) \rho_l(k) \,dx\,dt \,dy \Big] \\
& = \E \bigg[\int_0^1 \int_{\R^d} \int_{Q_T} \int_\R \int_{r<|z|\le r_1} \hspace{-1cm} (1-\tau)|C_+(v_\eps(t,y)-k) - C_+(u(t,x))| \, D_x^2 [\varphi(x-\tau z,t) \rho_m(x-\tau z-y)] z.z \,d\mu_\lambda(z) 
\\[-0.4cm] & \hspace{13cm}\rho_l(k) \,dk\,dx\,dt \,dy \,\d\tau \bigg]. 
\end{align*} 
\\[-0.5cm]
By using that 
\begin{align*}
&D_x^2 [\varphi(x-\tau z,t) \rho_m(x-\tau z-y)] 
\\ =& \rho_m(x-\tau z-y)D^2 \varphi(x-\tau z,t) + 2 D \rho_m(x-\tau z-y) D\varphi(x-\tau z,t)  + \varphi(x-\tau z,t) D^2 [\rho_m(x-\tau z-y)]
\\ =& \rho_m(x-\tau z-y)D^2 \varphi(x-\tau z,t) - 2 D_y[\rho_m(x-\tau z-y) D\varphi(x-\tau z,t)]  +  D_y^2 [\varphi(x-\tau z,t)\rho_m(x-\tau z-y)],
\end{align*}
one gets that 
\begin{align*}
\mathcal{K} \leq & 
\Big[\|C_+^\prime\|_\infty l \|D^2 \varphi\|_{L^1(Q_T)}  + \int_{Q_T} \|D^2 \varphi(\cdot,t)\|_\infty \E \big[ |C_+(v_\eps(t,x))| +  |C_+(u(t,x))|   \big] \,   
\,dx\,dt \Big]\int_{r<|z|\le r_1} |z|^2 \,d\mu_\lambda(z)
\\&
+2\|C_+^\prime\|_\infty \E \bigg[  \int_0^T \|\nabla\varphi(\cdot,t)\|_\infty   \|\nabla v_\eps(t,\cdot)| _{L^1(\D)}\, dt  \bigg] \int_{r<|z|\le r_1} |z|^2 \,d\mu_\lambda(z)
\\ &+
Cm\|C_+^\prime\|_\infty \E \bigg[\int_0^T \|\varphi(\cdot,t)\|_\infty \|\nabla v_\eps(t,\cdot)\|_{L^1(\D)}\,dt \bigg] \int_{r<|z|\le r_1} |z|^2 \,d\mu_\lambda(z)
\\ \leq &
C\|C_+^\prime\|_\infty\Big[
l \|D^2 \varphi\|_{L^1(Q_T)} + (\|v_0\|_{BV(\D)}+\|u_0\|_{L^1(\D)})\int_0^T \|D^2 \varphi(\cdot,t)\|_\infty\, dt
\\ & \hspace{3cm}
+  \| v_0\| _{BV(\D)}\int_0^T (\|\nabla\varphi(\cdot,t)\|_\infty + m \|\varphi(\cdot,t)\|_\infty)  dt\Big]\int_{r<|z|\le r_1} |z|^2 \,d\mu_\lambda(z)
\end{align*}

On the other hand, to handle the other term of $M_1$ we proceed as follows:
\begin{align*}
& \E \Big[\int_{\R^d} \int_{Q_T}\int_\R \mathscr{L}^{r_1}_{\lambda} [C_+(u(t,\cdot))](x)\,\varphi(t,x) \rho_m(x-y)\, \sgn(v_{\eps}(t,y) -k-u(t,x)) \rho_l(k)\,dk\,dy\,dx\,dt \Big] \\
  \le &\int_0^T \norm{\varphi(t,\cdot)}_{\infty}\, \int_{|z|> r_1} \norm{C_+(u(t,\cdot+z) -C_+(u(t,\cdot))}_{L^1(\R^d)} \,d\mu_\lambda(z) \,dt,
\end{align*}
and, since $u(\cdot,\cdot+z)$ is the entropy solution associated with the initial condition $u_0(\cdot+z)$, thanks to the $L^1$-contraction  principle of Theorem \ref{uniqueness_new}, 
\begin{align*}
&\limsup_lM_1 \leq
\\[-0.3cm] & 
\|C_+^\prime\|_{\infty} \Bigg[C(\|v_0\|_{BV}+\|u_0\|_{L^1})
\int_0^T \Big(\|D^2  \varphi(\cdot,t)\|_\infty\,  + \|\nabla\varphi(\cdot,t)\|_\infty + m \|\varphi(\cdot,t)\|_\infty \Big)dt\Big]\int_{r<|z|\le r_1} |z|^2 \,d\mu_\lambda(z)
\\&+
 \int_0^T \norm{\varphi(t,\cdot)}_{\infty} \,dt\, \int_{|z|> r_1} \norm{u_0(\cdot+z) -u_0}_{L^1(\R^d)} \,d\mu_\lambda(z)\Bigg].
\end{align*}
Observe that, similar calculations will help us to estimate $\mathrm{M}_2$. Indeed, using Lemma \ref{lemma_02}, we have
\begin{align*}
\mathrm{M}_2 & \le -\E \Big[\int_{\R^d} \int_{Q_T}\int_\R |C_-(u (t,x)+k) -C_-(v_\ep(t,y))|
\mathscr{L}^{r,r_1}_{\lambda} \big[ \rho_m(x-\cdot) \big](y) \varphi(t,x) \rho_l(k)\,dk\,dy\,dx\,dt \Big] \\
& + \E \Big[\int_{\R^d} \int_{Q_T} \int_\R \mathscr{L}^{r_1}_{\lambda} [C_-(v_\ep(t,\cdot))](y)\,\varphi(t,x) \rho_m(x-y)\, \sgn(u(t,x)+k-v_{\eps}(t,y)) \rho_l(k) \,dk \,dy\,dx\,dt \Big] 
.
\end{align*}
We estimate the first term of the above inequality as follows:
\begin{align*}
&\E \Big[\int_{\R^d} \int_{Q_T}\int_\R |C_-(u (t,x)+k) -C_-(v_\ep(t,y))|
\mathscr{L}^{r,r_1}_{\lambda} \big[ \rho_m(\cdot-y) \big](y) \varphi(t,x) \rho_l(k)\,dk\,dy\,dx\,dt \Big]
\\
&=\E \Big[\int^1_0 \int_{\R^d} \int_{Q_T}\int_\R \int_{r < |z| \le r_1} (1-\tau) | C_-(u(t,x)+k) -C_-(v_\ep(t,y))| \\
&\hspace{6cm} \times D_y^2 \rho_m(x-y+\tau z)z.z \varphi(t,x) \rho_l(k) \,d\mu_\lambda(z)\,dk\,dy\,dx\,dt\,d\tau\Big]
\\&\le 
\|C_-^\prime\|_{\infty} \E\bigg[\int_{0}^{1} \int_{\mathbb{R}^{d}} \int_{Q_{T}} \int_\R\int_{r<|z| \leq r_{1}}(1-\tau)  \varphi(t, x)  |\nabla v_{\varepsilon}(t, y)|
\\[-0.3cm]
&\hspace{6cm} \times  |\nabla \rho_m(x-y-\tau z)| |z|^2 \rho_l(k)\,d\mu_\lambda(z)\,dk\,dx\, dt\, d y \mathrm{d} \tau \bigg] 
\\& \le 
\|C_-^\prime\|_{\infty}m \int_0^T \|\varphi(t, \cdot)\|_\infty   \E \|\nabla v_{\varepsilon}(t, \cdot)\|_{L^1(\D)}\,dt   \int_{r<|z| \leq r_{1}}  |z|^2 \,d\mu_\lambda(z).
\end{align*}
And the second one:
\begin{align*}
&\E \Big[\int_{\R^d} \int_{Q_T} \int_\R \mathscr{L}^{r_1}_{\lambda} [C_-(v_\ep(t,\cdot))](y)\,\varphi(t,x) \rho_m(x-y)\, \sgn(u(t,x)+k-v_{\eps}(t,y)) \rho_l(k) \,dk \,dy\,dx\,dt \Big] 
\\ \leq &
\|C_-^\prime\|_{\infty} \E \Big[ \int_0^T\|\varphi(t,\cdot)\|_\infty \int_{r_1<|z| }\hspace{-0.5cm} \|v_\ep(t,\cdot+z)-v_\eps(t,\cdot)\|_{L^1(\D)} \,dt  \,d\mu_\lambda(z)\Big] 
\\ \leq &
\|C_-^\prime\|_{\infty}\int_0^T\|\varphi(t,\cdot)\|_\infty  \,dt\int_{r_1<|z| }\hspace{-0.5cm} \|v_0(\cdot+z)-v_0(\cdot)\|_{L^1(\D)}  \,d\mu_\lambda(z).
\end{align*}
Therefore, 
\begin{align*}
\limsup_{l}M_2 \leq 
\|C_-^\prime\|_{\infty} \int_0^T \|\varphi(t, \cdot)\|_\infty dt  
\Big[ m\|v_0\|_{BV(\D)} \int_{r<|z| \leq r_{1}}\hspace{-1cm} |z|^2 \,d\mu_\lambda(z)
+
\int_{r_1<|z| }\hspace{-0.5cm} \|v_0(\cdot+z)-v_0(\cdot)\|_{L^1(\D)}d\mu_\lambda(z) \Big].
\end{align*}
We are now in a position to add \eqref{continuous_dep_1} and \eqref{continuous_dep_2}
and pass to the limits in $n,\gamma$ and $l$. In what follows, invoking the above estimates and keeping in mind that $\{v_\eps\}_{\eps>0}$ converges in $L^p_{\mathrm{loc}} (\R^d; L^p((0,T) \times \Omega))$, for any $p \in [1,2)$, to the unique BV entropy solution $v$ of \eqref{eq:stoc_frac} with data  $(v_0, f,B, h , \lambda)$, we have 
\begin{align}\label{stoc_entropy_4}  
0 \leq &
\int_{\R^d}\int_{\R^d} |u_0(x)-v_0(y)|\varphi(0,x)\rho_m(x-y)\,dx\,dy \nonumber
  \\ &+
  \E \Big[\int_{Q_T}\int_{\R^d} |u(t,x)-v(t,y)| \partial_t\varphi(t,x) \, \rho_m(x-y)\,dy\,dx\,dt\Big]  \nonumber
  \\ & -
  \E \left[ \int_{Q_T} \int_\D F(u(t,x),v(t,y)).\nabla_x\varphi(t,x)  \rho_m(x-y)\,dx\,dy\,dt\right] \nonumber
  \\ & -  
  \E \bigg[\int_{\R^d} \int_{Q_T} |A(u(t,x))-A(v(t,y))| \,\mathscr{L}_{\lambda,r} \big[ \varphi(t, \cdot) \rho_m (\cdot -y)\big](x) \,dx\,dt \,dy \bigg]  \nonumber
  \\ & 
  - \E \bigg[\int_{\R^d} \int_{Q_T} |B(v(t,y))-B(u(t,x))| \,\mathscr{L}_{\lambda,r} \big[ \rho_m (x -\cdot)\big](y)\,\varphi(t,x) \,dx\,dt \,dy \bigg]  \nonumber
  \\ &
+ C\, \|C^\prime\|_\infty 
\int_0^T \Big(\|D^2  \varphi(\cdot,t)\|_\infty\,  + \|\nabla\varphi(\cdot,t)\|_\infty + m \|\varphi(\cdot,t)\|_\infty \Big)dt\Big]
\int_{r<|z| \leq r_{1}}\hspace{-1cm} |z|^2 \,d\mu_\lambda(z)  \nonumber
\\ & +
  \|C^\prime\|_{\infty} 
 \int_0^T \norm{\varphi(t,\cdot)}_{\infty} \,dt\, \int_{|z|> r_1} \Big[\norm{u_0(\cdot+z) -u_0}_{L^1(\R^d)} +\|v_0(\cdot+z)-v_0(\cdot)\|_{L^1(\D)}\Big]\,d\mu_\lambda(z)  \nonumber
\\ & - 
\E \Big[\int_{\R^d} \int_{Q_T} |B_+(u(t,x))-B_+(v(t,y)| \mathscr{L}^r_{\lambda} [\varphi(t,\cdot)](x) \, \rho_m(x-y)\,dy\,dx\,dt \Big] \nonumber
\\ & - 
\E \Big[\int_{\R^d} \int_{Q_T} |A_-(u(t,x))-A_-(v(t,y))|  \mathscr{L}^r_{\lambda} [\varphi(t,\cdot)](x) \, \rho_m(x-y)\,dy\,dx\,dt \Big],
\end{align}
where $C$ is a constant depending on $\|u_0\|_{L^1(\R^d)}$ and $\|v_0\|_{BV(\R^d)}$, and using the fact that $\| C^\prime_{\pm}\|_{\infty}\le \| A'-B'\|_\infty=\| C'\|_\infty$.


To proceed further, we make a special choice for the function $\psi(t,x)$. To this end, for each $h>0$ and fixed $t\ge 0$, we define
 \begin{align}
 \psi_h^t(s):=\begin{cases} 1, &\quad \text{if}~ s\le t, \notag \\
 1-\frac{s-t}{h}, &\quad \text{if}~~t\le s\le t+h,\quad \psi_R(x):=\text{min}\Big(1,\frac{R^a}{|x|^a}\Big)\notag \\
 0, & \quad \text{if} ~ s \ge t+h.
 \end{cases}
 \end{align}
Furthermore, let $\rho$ be any non-negative mollifier.  Clearly, \eqref{stoc_entropy_4} holds with
 $\varphi(s,x)=\psi_h^t(s) \, (\psi_R \star \rho)(x)$. 
 
With the above choice of test function in \eqref{stoc_entropy_4}, we first wish to pass to the limit as $R \to \infty$ and subsequently as $r \to 0$ in \eqref{stoc_entropy_4}. Thanks to the \textit{a priori} estimates in Appendix~\ref{sec:apriori+existence}, we recall that $u, v \in L^1(\Omega \times Q_T)$. 
Also note that by properties of $\psi_R$ and $\rho $, it follows that $\psi_R \star \rho \to 1$ pointwise as $R \to \infty$. 
Therefore, for any $x$ and $z$, $\psi_R \star \rho(x) -\psi_R \star \rho(x+z) \to 0$ and since it is bounded by $2$ which is integrable on the set $\{|z|>r\}$ with respect to $\mu_{\lambda}$, one concludes that $\mathscr{L}^{r}_{\lambda} [\varphi(t,\cdot)](x) \to 0$. 
\\
As moreover $|\mathscr{L}^{r}_{\lambda} [\varphi(t,\cdot)](x)| \leq 2 \int_{|z|>r} \frac{dz}{r^{d+2\lambda}}$, Lebesgue's theorem once again, yields 
\begin{align*}
& \E \Big[\int_{\R^d} \int_{Q_T} |B_+(u(t,x))-B_+(v(t,y))| \mathscr{L}^{r}_{\lambda} [\varphi(t,\cdot)](x) \, \rho_{m}(x-y)\,dy\,dx\,dt \Big] \to 0\quad (R \to + \infty).
\end{align*}
Similarly,
$$  \E \Big[\int_{\R^d} \int_{Q_T} |A_-(u(t,x))-A_-(v(t,y))| \mathscr{L}^{r}_{\lambda} [\varphi(t,\cdot)](x) \, \rho_{m}(x-y)\,dy\,dx\,dt \Big] \to 0\quad (R \to + \infty).
$$
Note that
\begin{align*}
&\mathscr{L}^r_{\lambda} [\varphi(t,\cdot)\rho_m (\cdot -y)](x) = \int_{|z|<r}\int_0^1 (1-\tau)D_x^2[\varphi(t,\cdot)\rho_m (\cdot -y)](x+\tau z)(z.z)\, d\tau d\mu_\lambda(z)
\\ =& 
\int_{|z|<r} \int_0^1 (1-\tau)\Big[\rho_m (x+\tau z -y)D^2\varphi(t,x+\tau z) + 2 D\varphi(t,x+\tau z)D\rho_m (x+\tau z -y) 
\\[-0.3cm]&\hspace{7cm}+ \varphi(t,x+\tau z)D^2\rho_m (x+\tau z -y)\Big](z.z)\, d\tau d\mu_\lambda(z),
\end{align*}
where $D\varphi$ and $D^2\varphi$ converge uniformly to $0$ when $R\to +\infty$. Thus, 
\begin{align*}
&\Big|\E \bigg[\int_{\R^d} \int_{Q_T} |A(u(t,x))-A(v(t,y))| \,\int_{|z|<r} \int_0^1 (1-\tau)\rho_m (x+\tau z -y)D^2\varphi(t,x+\tau z)(z.z)\, d\tau d\mu_\lambda(z) \,dx\,dt \,dy\Big|
\\ & \qquad  \leq 
\|D^2\varphi\|_\infty  C(A) \Big[\|u\|_{L^1(\Omega\times Q_T)}+\|v\|_{L^1(\Omega\times Q_T)}\Big] \int_{|z|<r} |z|^2\, d\mu_\lambda(z) \underset{R \to +\infty}{\to} 0,
\end{align*}
 and similarly with $2 D\varphi(t,x+\tau z)D\rho_m (x+\tau z -y)$ for a constant depending also on $m$.
\\
Concerning the term with $\varphi$, since it is bounded by $1$, 
\begin{align*}
\kappa=&\Big|\E \bigg[\int\limits_{\R^d} \int_{Q_T} |A(u(t,x))-A(v(t,y))| \,\int\limits_{|z|<r} \int_0^1 (1-\tau)\varphi(t,x+\tau z)D^2\rho_m (x+\tau z -y)(z.z)\, d\tau d\mu_\lambda(z) \,dx\,dt \,dy\Big|
\\ \leq &
C(A,m^2) \Big[\|u\|_{L^1(\Omega\times Q_T)}+\|v\|_{L^1(\Omega\times Q_T)}\Big] \,\int_{|z|<r} |z|^2\, \, d\mu_\lambda(z)
\end{align*}
and $\lim_{r \to 0}\limsup_{R \to +\infty} \kappa =0$, as well as the similar term 
\begin{align*}
  - \E \bigg[\int_{\R^d} \int_{Q_T} |B(v(t,y))-B(u(t,x))| \,\mathscr{L}_{\lambda,r} \big[ \rho_m (x -\cdot)\big](y)\,\varphi(t,x) \,dx\,dt \,dy \bigg].
\end{align*}
Hence using the converge of $\nabla \varphi$, passing limit over $R$ and $r$, and simple applications of dominated convergence theorem yield
\begin{align}
 0 & \le   \E \Big[\int_{\R^d} \int_{\R^d}\big|u_0(x)-v_0(y)\big|\rho_m(x-y)\,dx\,dy\Big] + 
 \E \Big[\int_{Q_T}\int_{\R^d}  |u(t,x)-v(t,y)| \partial_s\psi^t_h(s)\rho_m(x-y))\,dy\,dx\,dt\Big]\notag 
 \\
&+  \|A'-B'\|_\infty \int_0^T \psi^t_h(s)\,ds\, \int_{|z|> r_1} \Big[\norm{u_0(\cdot+z) -u_0}_{L^1(\R^d)}+\norm{v_0(\cdot+z) -v_0}_{L^1(\R^d)}\Big] \,d\mu_\lambda(z)\notag\\
&+ Cm\|A'-B'\|_\infty \int_0^T \psi^t_h(s)\,ds\, \int_{0<|z|\le r_1} |z|^{2} \,d\mu_\lambda(z)  \label{stoc_entropy_nonl},
\end{align}
where $C$ is a constant depending on $\|u_0\|_{L^1(\R^d)}$ and $\|v_0\|_{BV(\R^d)}$.
Let $\mathbb{T}$ be the set all points $t$ in $[0, \infty)$ such that $t$ is a right
 Lebesgue point of 
 \begin{align*}
\mathcal{B}(t)= \E \Big[\int_{\R^d}\int_{\D} |u(t,x)-v(t,x)| \rho_m (x-y)\,dx\,dy \Big].
\end{align*}

\noindent Clearly, $\mathbb{T}^{\complement}$ has zero Lebesgue measure. Fix  $t\in \mathbb{T}$. Thus, passing to the limit as $h\goto 0$ in \eqref{stoc_entropy_nonl}, we obtain
 \begin{align}
&\E \Big[\int_{\R^{d}}\int_{\D} | u(t,x)-v(t,y)| \rho_m (x-y)\,dx\,dy \Big] 
\le  \Big[\int_{\R^{d}} \int_{\R^{d}} |u_0(x)-v_0(y)| \rho_m (x-y) \,dx\,dy\Big] \label{pre:final} \\
&+  \|A'-B'\|_\infty t\, \int_{|z|> r_1} \Big[\norm{u_0(\cdot+z) -u_0}_{L^1(\R^d)}+\norm{v_0(\cdot+z) -v_0}_{L^1(\R^d)}\Big] \,d \lambda(z) \nonumber
\\ &+ C\,t\,m\|A'-B'\|_\infty  \, \int_{0<|z|\le r_1} |z|^{2} \,d\mu_\lambda(z). \notag 
\end{align}
Next, observe that
 \begin{align}
\E \Big[\int_{\R^d} & \big|u(t,x)-v(t,x)\big|\,dx\Big]  \notag \\
  \le &   \E \Big[\int_{\R^d}\int_{\R^d}  \big| u(t,x)-v(t,y)\big| 
\rho_m(x-y)\,dx\,dy \Big]
 + \E \Big[\int_{\R^d}\int_{\R^d} \big| v(t,x) -v(t,y)\big|\rho_m(x-y)\,dx\,dy \Big]\notag \\
\le &\E  \Big[\int_{\R^{d}}\int_{\D} | u(t,x)-v(t,y)| \rho_m (x-y)\,dx\,dy \Big] + \frac1m\,\|v_0\|_{\mathrm{BV}}, \label{estimate:solu_levy}
 \end{align} 
and 
 \begin{align}
    \Big[\int_{\R^d}\int_{ \R^d} \big| u_0(x) -v_0(y)\big|\rho_m(x-y) \,dx\, dy\Big]
   \le   \Big[\int_{\R^d} \big|u_0(x)-v_0(x)\big| \,dx \Big]+ \frac1m \, \|v_0\|_{\mathrm{BV}}. \label{estimate:ini_levy}
 \end{align}
Making use of \eqref{estimate:solu_levy} and \eqref{estimate:ini_levy} in \eqref{pre:final}  , we have 
\begin{align}
  \E& \Big[\int_{\R^d}  \big| u(t,x)-v(t,x)\big|\,dx \Big] 
 \le \Big[\int_{ \R^d}  \big| u_0(x) -v_0(x)\big|\,dx \Big]  
 + 2/m \,\|v_0\|_{\mathrm{BV}}  \label{inq:pre-final_001} \\
 & + t \|A'-B'\|_\infty \, \int_{|z|> r_1} \Big[\norm{u_0(\cdot+z) -u_0}_{L^1(\R^d)}+\norm{v_0(\cdot+z) -v_0}_{L^1(\R^d)}\Big] \,d\mu_\lambda(z) \notag 
 \\ &+ C t m\|A'-B'\|_\infty  \, \int_{0<|z|\le r_1} |z|^{2} \,d\mu_\lambda(z). \notag 
  \end{align}
Now we optimize the terms involving $m$ in \eqref{inq:pre-final_001}, by using  $\min_{m>0} \big(m a + \frac{b}{m}  \big) = 2 \sqrt{ab}$, for $a,b > 0$ , we obtain
\begin{align}
 \label{inq:pre-final_002}
  \E & \Big[ \int_{\D} \big| u(t,x)-v(t,x)\big|\,dx \Big] 
 \le   \Big[ \int_{\D} \big| u_0(x) -v_0(x)\big|\,dx \Big] + C\sqrt{T\|A'-B'\|_\infty} r_1^{1-\lambda} \notag \\
 & + C \|A'-B'\|_\infty T\, \int_{|z|> r_1} \Big[\norm{u_0(\cdot+z) -u_0}_{L^1(\R^d)}+\norm{v_0(\cdot+z) -v_0}_{L^1(\R^d)}\Big] \,d\mu_\lambda(z) \notag 
\\  \le&   \| u_0 -v_0\|_{L^1} + C\sqrt{T\|A'-B'\|_\infty} r_1^{1-\lambda}   + C \|A'-B'\|_\infty T\, \Big[\norm{u_0}_{L^1(\R^d)}+\norm{v_0}_{L^1(\R^d)}\Big] r_1^{-2\lambda},\notag 
\end{align}
 where C depends on $\|v_0\|_{BV}$ and $\|u_0\|_{L^1}$. Then, choosing $r_1=[T\|A'-B'\|_\infty]^{\frac1{2(1+\lambda)}}$, one gets that
\begin{align*}
  \E & \Big[ \int_{\D} \big| u(t,x)-v(t,x)\big|\,dx \Big] 
 \le   \norm{u_0-v_0}_{L^1(\R^d)} + C\Big(\norm{u_0}_{L^1(\R^d)},\norm{v_0}_{BV(\R^d)}\Big)\Big[T\|A'-B'\|_\infty\Big]^{\frac1{1+\lambda}}. 
\end{align*}
This finishes the proof of the theorem.

\section{Proof of Corollary \ref{rate} : Rate of Convergence}
\label{rateofconv}
Having achieved the convergence of vanishing viscosity solutions $u_\ep(t,x)$ of the problem \eqref{eq:viscous-Brown} to the unique entropy solution of the stochastic conservation law \eqref{eq:stoc_frac}, we now look forward to derive the explicit rate of convergence. We shall use the continuous dependence estimates to explicitly obtain the rate of convergence of the sequence $\{ u_\ep(t,x)\}_{\ep>0}$ to the BV entropy solution $u(t,x)$ of \eqref{eq:stoc_frac}.\\
For $\ep>0$, let $u_\ep(t,x)$ be the weak solution to \eqref{eq:viscous-Brown} with data $(u_0,f,A,h,\lambda)$ and $u(t,x)$ be the entropy solution to \eqref{eq:stoc_frac}. A similar argument (with $A= B$) leading to \eqref{stoc_entropy_4} yields
\begin{align}
0 \leq &
\int_{\R^d}\int_{\R^d} |u_0(x)-u^\eps_0(y)|\varphi(0,x)\rho_m(x-y)\,dx\,dy + C\eps m \nonumber
  \\ &+
  \E \Big[\int_{Q_T}\int_{\R^d} |u(t,x)-u_\eps(t,y)| \partial_t\varphi(t,x) \, \rho_m(x-y)\,dy\,dx\,dt\Big]  \nonumber
  \\ & -
  \E \left[ \int_{Q_T} \int_\D F(u(t,x),u_\eps(t,y)).\nabla_x\varphi(t,x)  \rho_m(x-y)\,dx\,dy\,dt\right] \nonumber
  \\ & -  
  \E \bigg[\int_{\R^d} \int_{Q_T} |A(u(t,x))-A(u_\eps(t,y))| \,\mathscr{L}_{\lambda,r} \big[ \varphi(t, \cdot) \rho_m (\cdot -y)\big](x) \,dx\,dt \,dy \bigg]  \nonumber
  \\ & 
  - \E \bigg[\int_{\R^d} \int_{Q_T} |A(u_\eps(t,y))-A(u(t,x))| \,\mathscr{L}_{\lambda,r} \big[ \rho_m (x -\cdot)\big](y)\,\varphi(t,x) \,dx\,dt \,dy \bigg]  \nonumber
\\ & - 
\E \Big[\int_{\R^d} \int_{Q_T} |A_+(u(t,x))-A_+(u_\eps(t,y)| \mathscr{L}^r_{\lambda} [\varphi(t,\cdot)](x) \, \rho_m(x-y)\,dy\,dx\,dt \Big] \nonumber
\\ & - 
\E \Big[\int_{\R^d} \int_{Q_T} |A_-(u(t,x))-A_-(u_\eps(t,y))|  \mathscr{L}^r_{\lambda} [\varphi(t,\cdot)](x) \, \rho_m(x-y)\,dy\,dx\,dt \Big],
\end{align}
where $C$ is a constant depending on $\|u_0\|_{BV(\R^d)}$.

As before we choose the test function $\varphi(s,x)= \psi^t_h(s)( \psi_R \ast \rho)(x)$, where $\psi^t_h(s)$,$\psi_R$ are defined previously. Passing to the limit as $R \to \infty$ and then $r \to 0$, we get 
\begin{align*}
&-\E \left[ \int_{Q_T} \int_\D |u(s,x)-u_\ep(s,y)| \partial_s \psi^t_h(s) \rho_m(x-y) \,dy\,dx\,ds \right]\\
& \hspace{3cm}\le \E \left[ \int_\D \int_\D | u_0(x) - u^\ep_0(y)| \rho_m(x-y)\,dx\,dy \right]+ C \ep m.
\end{align*}
Next, we let $h \to 0$ to get 
\begin{align*}
&\E \left[ \int_\D \int_\D |u(s,x)-u_\ep(s,y)|  \rho_m(x-y) \,dy\,dx \right] \le  \left[ \int_\D \int_\D | u_0(x) - u^\ep_0(y)| \rho_m(x-y)\,dx\,dy \right]+ C \ep m.
\end{align*}
As $u_\ep(t,y)$ and $u(t,x)$ satisfy the spatial BV bounds, bounded by the BV norm of $u_0$, we obtain 
\begin{equation}
\begin{aligned}
\E \bigg[ \int_\D | u_\ep(t,x) -u(t,x)| \,dx \bigg] \le C\bigg( \ep^{1/2}+ \frac{1}{m} + \ep m \bigg). \label{ROC}
\end{aligned}
\end{equation}
Choosing the optimal value of $m =\ep^{-1/2}$ in \eqref{ROC} yields
$$\E \big[ \| u_\ep(t,\cdot) -u(t,\cdot) \|_{L^1(\D)} \big] \le C \ep^{1/2},$$
where $C>0$ is a constant depending only on $|u_0|_{BV}$.


\section{Proof of Theorem~\ref{prop:vanishing viscosity-solution}: Existence of Viscous Solution}
\label{viscous}

In this section, we demonstrate a proof of existence and uniqueness of the solution $u_\eps$ in $N_w^2(0,T,H^1(\R^d))$ which has pathwise continuous trajectories with values in $L^2(\R^d)$, under assumptions~\ref{A1}--\ref{A4}, to the regularized problem \eqref{eq:viscous-Brown}:
\begin{align*}
du_\eps(t,x) -[\eps \Delta u_\eps(t,x)+ \Div f(u_\ep(t,x))]\,dt & + \mathscr{L}_{\lambda}[A(u_\eps(t, \cdot))](x)\,dt = \h(u_\eps(t,x))\,dW(t),
\end{align*}
with a regular initial data $u_{\eps}(0,\cdot)=u_0^{\eps}\in H^1(\R^d)\cap L^1(\R^d)$ such that 
$u_0^{\eps}$ converges to $u_0$ in $L^2(\R^d)$, $\sqrt{\eps}u_0^{\eps}$ is bounded in $H^1(\R^d)$ by $C\|u_0^{\eps}\|_{L^2(\R^d)}$; if moreover $u_0\in L^1(\R^d)$ then $\|u_0^{\eps}\|_{L^1(\R^d)} \leq \|u_0\|_{L^1(\R^d)}$ and if $u_0\in BV(\R^d)$ then $TV(u_0^{\eps}) \leq TV(u_0)$. 
\\
Note that such an initial condition can be chosen in the following way: $u_0^{\eps}=v$ where $v$ is the solution to $v - \eps \Delta v = u_0$ if $u_0$ is also known to be in $L^1(\R^d)$, else $u_0^{\eps}=v\chi_\eps$ where $\chi_\eps$ is the classical cut-off function.
Note that one may expect to establish an existence result for \eqref{eq:viscous-Brown} by applying classical results using monotone arguments. However, there is a complicating factor at play here, mainly for the case $\lambda >1/2$. In fact, because of the nonlinear function $A$, present in the fractional Laplace operator, it is not feasible to apply the classical results of existence based on monotone arguments. As a remedy, we follow a general strategy of proving existence of solutions by a compactness argument. To that context, we first propose the following singular perturbation of \eqref{eq:viscous-Brown}
\begin{equation}
\label{eq:viscous-Brown_one_01} 
\begin{aligned}
du_{\eps, \gamma}(t,x) + \mathscr{L}_{\lambda}[A(u_{\eps, \gamma}(t,\cdot))](x)\,dt  & - \Div f(u_{\eps, \gamma}(t,x)) \,dt + \gamma \Delta^2 u_{\eps,\gamma}(t,x) \,dt  \\
&=  \eps \Delta u_{\eps,\gamma}(t,x) \,dt  + \h(u_{\eps,\gamma}(t,x))\,dW(t),  \quad (x,t) \in {Q}_{T}, \\
u_{\eps, \gamma}(0,x) &= u^{\eps}_0(x), \quad x \in \D.
\end{aligned}
\end{equation}

\subsection{Existence $\&$ uniqueness of solution for \eqref{eq:viscous-Brown_one_01}}
To establish the existence and uniqueness of strong solutions to \eqref{eq:viscous-Brown_one_01}, we closely follow the work of Pr\'evot and R\"ockner \cite[Sec. 4.1 p.55]{PrevotRockner}.

In what follows, let us denote by $H=L^2(\D)$, $V=H^2(\D)$ with dual space $H^{-2}(\R^d)$ and by $\bar A$ the operator $\bar A(t,u):= \eps \Delta u - \mathscr{L}_{\lambda}[A(u)] + \Div f(u) -\gamma \Delta^2 u$ understood in the weak sense:
\begin{align*}
\forall u,v \in H^2(\D), \quad \langle \bar A(t,u),v \rangle = 
- \int_{\D} \big[ \eps \nabla u \nabla v + \gamma \Delta u \Delta v + \mathscr{L}_{\lambda/2}[A(u)]\mathscr{L}_{\lambda/2}[v] + f(u)\nabla v \big] dx.
\end{align*}

\begin{lem}
	For any $\gamma \in (0,1)$, there exists a unique strong $L^2(\D)$-valued continuous process solution $u_{\eps, \gamma} \in N^2_w(0,T;H^2(\D))$, such that $\partial_t \big[ u_{\eps, \gamma} - \int_0^t \h(u_{\eps, \gamma}) dW \big] \in L^2(\Omega\times(0,T),H^{-2}(\D))$, to the regularized problem \eqref{eq:viscous-Brown_one_01}.
\end{lem}
\begin{proof}
Note that it is enough to verify conditions $(H1)$-$(H4)$ of \cite[p.556]{PrevotRockner} : the hemicontinuity, the weak monotonicity, the coercivity and the  boundedness. 
	\\
	To verify the condition of the hemicontinuity $(H1)$, we remark that this continuity is obvious for the linear part of the operator, and one only needs to remark that 
	\begin{align*}
	s \mapsto \int_{\D} \big[  \mathscr{L}_{\lambda/2}[A(u+s v)]\mathscr{L}_{\lambda/2}[p] + f(u+s v)\nabla p \big] dx
	\end{align*}
	is continuous since $f$ and $A$ are Lipschitz-continuous functions and $u$, $v$ and $p$ are in $H^2(\D)$. 
	
	For the verification of the weak monotonicity $(H2)$, we note that  for any $u$ and $v$ in  $H^2(\D)$, 
	\begin{align*}
&- 2\int_{\D} \big[ \eps |\nabla (u- v)|^2 + \gamma |\Delta (u- v)|^2 + \mathscr{L}_{\lambda/2}[A(u)-A(v)]\mathscr{L}_{\lambda/2}[u-v] \\
&\hspace{7cm} + [f(u)-f(v)]\nabla (u-v) \big] dx+  \| \h(u)-\h(v)\|^2_2
\\ \leq &
- 2\Big[\eps \|\nabla (u- v)\|_2^2 + \gamma \|\Delta (u- v)\|_2^2 + \int_{\D} [A(u)-A(v)]\mathscr{L}_{\lambda}[u-v]dx\Big] \\
&\hspace{7cm} + \|f'\|_\infty\|u-v\|_2\|\nabla (u-v)\|_2 +  K\| u-v\|^2_2
\\ \leq &
- 2\Big[\eps \|\nabla (u- v)\|_2^2 + \gamma \|\Delta (u- v)\|_2^2\Big]+\|A'\|_\infty\|u-v\|_2\|\mathscr{L}_{\lambda}[u-v]\|_2 \\
& \hspace{7cm}  + \|f'\|_\infty\|u-v\|_2\|\nabla (u-v)\|_2 +  K\| u-v\|^2_2
\\ \leq &
-\Big[\eps \|\nabla (u- v)\|_2^2 + \gamma \|\Delta (u- v)\|_2^2\Big]+ (K+C(\eps,\gamma)\| u-v\|^2_2
\end{align*}
where we have used the fact that $\|\mathscr{L}_{\lambda}[u-v]\|_2\leq C\|u-v\|_{H^2}$ and that $u \mapsto \|u\|_2+\|\Delta u\|_2$ is a norm on $H^2(\D)$ equivalent to the usual one.
	
To verify the coercivity $(H3)$, we proceed as follows: 
\begin{align*}
&- 2\int_{\D} \big[ \eps |\nabla u)|^2 + \gamma |\Delta u|^2 + \mathscr{L}_{\lambda/2}[A(u)]\mathscr{L}_{\lambda/2}[u] + f(u)\nabla (u-v) \big] dx+  \| \h(u)\|^2_2
\\ & \qquad \leq 
- 2\Big[\eps \|\nabla u\|_2^2 + \gamma \|\Delta u\|_2^2\Big] +  K\| u\|^2_2 \leq - 2\min(\eps,\gamma) \| u\|_{H^2}^2 +  (K+2\min(\eps,\gamma))\| u\|^2_2,
\end{align*}
since $A$ is an increasing function and thanks to Green's formula. 
	
Finally, to verify the boundedness condition $(H4)$, we proceed as follows: observe that
\begin{align*}
&\norm{\ep \Delta v - \gamma\,\Delta^2 v -\mathscr{L}_{\lambda}[A(v)] +\mathrm{div} f(v)}_{H^{-2}}
= \sup_{\| u\| \neq 0} \frac{\big| \big<\ep \Delta v - \gamma\,\Delta^2 v -\mathscr{L}_{\lambda}[A(v)] +\mathrm{div} f(v) ,u \big> \big|}{\|u\|_{H^2}}\\
&\qquad = \sup_{\|u\|\neq 0} \frac{\big| -\ep \langle \nabla v, \nabla u\rangle -\gamma \langle \Delta v, \Delta u\rangle -\langle \mathscr{L}_{\lambda}[A(v)], u \rangle -\langle f(v),\nabla u\rangle \big|}{\|u\|_{H^2}}\\
&\le  \sup_{\|u\| \neq 0} \ep \frac{\| \nabla v\|_2 \| \nabla u\|_2}{\|u\|_{H^2}} +\sup_{\|u\| \neq 0} \gamma \frac{\| \Delta v\|_2 \| \Delta u\|_2}{\|u\|_{H^2}} + C_A \sup_{\|u \| \neq 0} \frac{\| v\|_{H^1}\| u\|_{H^1}}{\|u \|_{H^2}}+ C_f\sup_{\|u \| \neq 0} \frac{\|v\|_2 \|\nabla u\|_2}{\|u\|_{H^2}},
\end{align*}
and thus we get 
\begin{align*}
\norm{\ep \Delta v - \gamma\,\Delta^2 v -\mathscr{L}_{\lambda}[A(v)] +\mathrm{div} f(v)}_{H^{-2}}  \le \big( c(\ep)+ c(\gamma) +c_{A}+C_f \big) \|v\|_{H^2}.
\end{align*}
Thus there exists a unique solution to \eqref{eq:viscous-Brown_one_01}, thanks to \cite[Theorem 4.2.4, p.75]{PrevotRockner}.
\end{proof}


As a first step towards proving the existence and uniqueness of strong solutions to the regularized problem \eqref{eq:viscous-Brown}, 
we will assume that $u_0^\eps$ enjoys additional regularity: $u_0^\eps \in L^2\big(\R^d,\|x\|^2 dx \big)$. In the subsequent section, we will show how to remove this additional regularity assumption and deal with the case when $u_0^\eps \in L^2(\R^d)$ by an approximation argument and making use of the ``$L^1$-stability" estimate of the solution with respect to the initial data. 

\subsection{Energy estimates}
\label{regularity}

In this subsection, our main aim is to derive \textit{a priori} energy estimates required to furnish a compactness argument. To that context, we have the following lemma:
\begin{lem}[Energy estimates]
\label{EnergyEstimates}
There exists a constant $C>0$ such that, for all $\gamma \in (0,1)$, the unique weak solution $u_{\eps, \gamma} \in N_w^2(0,T,H^2(\R^d))$ satisfies the following $\gamma$-independent energy bounds
\begin{itemize}
\item [$(a)$] The solution satisfies
\begin{align*}
\displaystyle{\sup_{0\le t\le T} \E\Big[\big\|u_{\eps, \gamma}(t)\big\|_{L^2(\R^d)}^2\Big]}  &+ \eps \int_0^T \E\Big[\big\|\grad u_{\eps, \gamma}(s)\big\|_{L^2(\R^d)}^2\Big]\,ds \\
&+ \gamma \int_0^T \E\Big[\big\|\Delta u_{\eps, \gamma}(s)\big\|_{L^2(\R^d)}^2\Big]\,ds  + \int_0^T \E\Big[\big\| A(u_{\eps, \gamma})\|_{H^{\lambda}(\R^d)}^2\Big]\,ds \le C, 
\end{align*} \vspace{.2cm}
\item [$(b)$] $\norm{\partial_t \big(u_{\eps, \gamma} - \int_0^t \h(u_{\eps, \gamma}(s,\cdot))\, dW(s) \big)}_{L^2(\Omega\times(0,T); H^{-2}(\R^d))} \le C,$\vspace{.2cm}
\item [$(c)$] $\Big\|u_{\eps, \gamma} -\int_0^t \h(u_{\eps,\gamma})\,dW \Big\| _{L^2 \big(\Omega, W^{1,2}(0,T,H^{1}(\D),H^{-2}(\D))\big)} \leq C/\sqrt{\eps}$,
\item [$(d)$] For any $p\geq 2$ and any $0<\alpha<\frac12$
\begin{align*}
\norm{\int_0^t \h(u_{\eps, \gamma}) \,dW}_{L^p\big(\Omega,W^{\alpha,p}(0,T,L^2(\D))\big)} \le C,\text{ and }\sup_\gamma  \E \big[\|u_{\eps, \gamma} \|^p_{C([0,T],L^2(\D))} \big] \le C
\end{align*}
\end{itemize}
\begin{itemize}
\item[$(e)$] Assuming moreover that $u_0 \in L^2(\R^d,\|x\|^2 dx)$, one has $\sup_t\E \big[\int_{\R^d} u^2_{\eps, \gamma}(t) \|x\|^2 dx\big] \leq C(\eps)$.
\end{itemize}
\end{lem}

\begin{proof} 

	By It\^o's energy, one gets that
	\begin{align*}
	&\norm{u_{\eps, \gamma}(t)}_{L^2(\R^d)}^2  - \norm{u^{\eps}_0}_{L^2(\R^d)}^2 
	+ 2\eps \int_0^t\int_\D |\nabla u_{\eps, \gamma} |^2\,dx\,ds + 2\gamma  \int_0^t\int_\D |\Delta u_{\eps, \gamma} |^2\,dx\,ds \nonumber \\
	&+2\int_0^t\int_\D f(u_{\eps, \gamma})\nabla u_{\eps, \gamma}\,dx\,ds+ 2  \int_0^t \int_\D\mathscr{L}_{\lambda}[A(u_{\eps, \gamma})]\mathscr{L}_{\lambda}[u_{\eps, \gamma}] \,dx\,ds \nonumber
	\\ = &
	2\sum_{k\ge 1}\int_0^t\int_{\D} g_k(u_{\eps,\gamma})u_{\eps,\gamma}\,dx\,d\beta_k(s)
	+ \int_0^t\int_{\D} \mathbb{G}^2(u_{\eps,\gamma}) \,dx\,ds. 
	\end{align*}
	Since $\int_\D f(u_{\eps,\gamma})\nabla u_{\eps,\gamma} \,dx=0$ and 
	\begin{align*}
	\frac{1}{\|A^\prime\|_\infty} |A(u_{\eps, \gamma}(t,x))-A(u_{\eps, \gamma}(t,y))|^2 \leq 
	\big[A(u_{\eps, \gamma}(t,x))-A(u_{\eps, \gamma}(t,y))\big]\big[u_{\eps, \gamma}(t,x)-u_{\eps, \gamma}(t,y)\big],
	\end{align*}
	this yields
	\begin{align}\label{viscousito}
	&\norm{u_{\eps, \gamma}(t)}_{L^2(\R^d)}^2
	+ 2\eps \int_0^t\|\nabla u_{\eps, \gamma}\|_2^2\,ds + 2\gamma  \int_0^t\|\Delta u_{\eps, \gamma}\|^2\,ds  + \frac{2}{\|A^\prime\|_\infty}  \int_0^t \|A(u_{\eps, \gamma})\|^2_{H^{\lambda}}\,ds \nonumber
	\\ & \qquad \leq  \norm{u^{\eps}_0}_{L^2(\R^d)}^2  + 
	2\sum_{k\ge 1}\int_0^t\int_{\D} g_k(u_{\eps,\gamma})u_{\eps,\gamma}\,dx\,d\beta_k(s)
	+ K\int_0^t\|u_{\eps,\gamma}\|^2_2\,ds. 
	\end{align}
	Taking expectation, we get 
	\begin{align*}
	&\E \Big[\norm{u_{\eps, \gamma}(t)}_{L^2(\R^d)}^2\Big]
	+ \eps \E \Big[\int_0^t\int_\D |\nabla u_{\eps, \gamma} |^2\,dx\,ds \Big]+ \gamma \E \Big[\int_0^t\int_\D |\Delta u_{\eps, \gamma} |^2\,dx\,ds \Big] 
	+  \E \Big[ \int_0^t \norm{A(u_{\eps, \gamma})}^2_{H^{\lambda}(\R^d)} \Big]
	\\ & \qquad \qquad \le \E \Big[\norm{u^{\eps}_0}_{L^2(\R^d)}^2 \Big] + C\, \E \Big[ \int_0^t  \norm{u_{\eps, \gamma}(s)}_{L^2(\R^d)}^2 \,ds \Big] , 
	\end{align*}
	and a simple application of Gronwall inequality proves (a).

	To prove (b), first observe that we can recast the equation as
	\begin{align*}
	\partial_t \Big[u_{\eps, \gamma} -\int_0^t \h(u_{\eps,\gamma})\,dW \Big] + \gamma \Delta^2 u_{\eps,\gamma}  - \eps\Delta u_{\eps,\gamma} - \Div f(u_{\eps, \gamma}) +  \mathscr{L}_{\lambda}[A(u_{\eps, \gamma})] =0.
	\end{align*}
	Therefore, thanks to the regularity obtained in part (a), we conclude that
	\begin{align*}
	\sup_{\gamma}\bigg\| \partial_t \Big[u_{\eps, \gamma} -\int_0^t \h(u_{\eps,\gamma})\,dW \Big]\bigg\|^2_{L^2(\Omega\times(0,T);H^{-2}(\D))} \leq C.
	\end{align*}
	Note that \ref{A4} and It\^o's isometry yields the boundedness of $\int_0^t \h(u_{\eps,\gamma})\,dW$ in $L^2(\Omega\times(0,T),H^{1}(\D))$ by the norm of $u_{\eps,\gamma}$ in the same space. Hence, (c) is proved by adding (a) and the above inequality. 
	Then, following a classical result from \cite[Lemma 2.1]{FlandoliGatarek}, we conclude
	\begin{align*}
	\forall p \ge 2, \forall \alpha < \frac12, &\quad \bigg\|\int_0^t \h(u_{\eps,\gamma})\,dW \bigg\|^p_{L^p\big(\Omega,W^{\alpha,p}(0,T,L^2(\D))\big)} \leq c \| \h(u_{\eps,\gamma}) \|^p_{L^p\big(\Omega\times(0,T); L^2(\D))\big)}.
	\end{align*}
Next, to get an estimate on $\sup_\gamma  \E \big[\|u_{\eps, \gamma} \|^p_{L^\infty(0,T,L^2(\D))} \big]$, observe that from the equation \eqref{viscousito} we get
	\begin{align*}
	\sup_{0 \le t \le T} \norm{u_{\eps, \gamma}(t)}_{L^2(\R^d)}^2 \le \norm{u^{\eps}_0}_{L^2(\R^d)}^2 & + K \int_0^T  \norm{u_{\eps, \gamma}(s)}_{L^2(\R^d)}^2 \,ds \\
	& + 2 \sup_{0 \le t \le T} \Bigg|  \sum_{k\ge 1}\int_0^t\int_{\D} g_k(u_{\eps,\gamma})\,u_{\eps,\gamma} \,d\beta_k(s)\,dx   \Bigg|
	\end{align*}
	We raise both sides of the above inequality to the power $p/2$, take the expectation, and apply several elementary inequalities, to arrive at 
	\begin{align*}
	\E \Bigg[\sup_{0 \le t \le T} \norm{u_{\eps, \gamma}(t)}_{L^2(\R^d)}^p \Bigg] \le C(p)\Bigg[ \E \Big[\norm{u^{\eps}_0}_{L^2(\R^d)}^p \Big]& +  \int_0^T  \E \Big[\norm{u_{\eps, \gamma}(s)}_{L^2(\R^d)}^p \Big] \,ds \Big] \\
	& +  \E \Big[\sup_{0 \le t \le T} \Big|  \sum_{k\ge 1}\int_0^t\int_{\D} g_k(u_{\eps,\gamma})\,u_{\eps,\gamma} \,d\beta_k(s)\,dx   \Big|^{p/2} \Big]\Bigg]
	\end{align*}
	A classical application of Burkholder-Davis-Gundy's, Cauchy-Schwarz's and Young's inequalities (cf. \cite{Hofmanova}) reveal that
	\begin{align*}
	&\E \Bigg[\sup_{0 \le t \le T} \Bigg|  \sum_{k\ge 1}\int_0^t\int_{\D} g_k(u_{\eps,\gamma})\,u_{\eps,\gamma} \,d\beta_k(s)\,dx   \Bigg|^{p/2} \Bigg] 
	\leq 
	C \E \Bigg[\Big[ \int_0^T\sum_{k\ge 1} \Big|  \int_{\D} g_k(u_{\eps,\gamma})\,u_{\eps,\gamma}\,dx\Big|^{2} \,ds   \Big]^{p/4} \Bigg] 
	\\
	\leq& 
	C \E \Bigg[\Big[ \int_0^T\sum_{k\ge 1}  \|g_k(u_{\eps,\gamma})\|^2_{L^2(\D)}\|u_{\eps,\gamma}\|^2_{L^2(\D)} \,ds   \Big]^{p/4} \Bigg] \\
		\leq &
	C \E \Bigg[\Big[\sup_{0 \le t \le T} \|u_{\eps,\gamma}\|^2_{L^2(\D)}\int_0^T\sum_{k\ge 1}  \|g_k(u_{\eps,\gamma})\|^2_{L^2(\D)} \,ds   \Big]^{p/4} \Bigg] 
	\\
	\leq& 
	CK^{p/4} \E \Bigg[\sup_{0 \le t \le T} \|u_{\eps,\gamma}\|^{p/2}_{L^2(\D)}\Big[\int_0^T\|u_{\eps,\gamma}\|^2_{L^2(\D)} \,ds   \Big]^{p/4} \Bigg] 
	\leq
	C \E \Bigg[\sup_{0 \le t \le T} \|u_{\eps,\gamma}\|^{p/2}_{L^2(\D)}\int_0^T\|u_{\eps,\gamma}\|^{p/2}_{L^2(\D)} \,ds    \Bigg] 
	\\%
	 \le & \frac12 \E \Bigg[ \sup_{0 \le t \le T} \norm{u_{\eps, \gamma}(t)}_{L^2(\R^d)}^p  \Bigg]
	+ C \int_0^T  \E \Big[\norm{u_{\eps, \gamma}(s)}_{L^2(\R^d)}^p \Big] \,ds.
	\end{align*}
Therefore we have 
$$\E \Bigg[\sup_{0 \le t \le T} \norm{u_{\eps, \gamma}(t)}_{L^2(\R^d)}^p \Bigg] \le  C\Bigg(\E\Big[\norm{u^{\eps}_0}_{L^2(\R^d)}^p \Big]+1+\int_0^T  \E \Big[\norm{u_{\eps, \gamma}(s)}_{L^2(\R^d)}^p \Big] \,ds \Bigg).$$
Since $T$ is arbitrary, applying Gronwall's lemma we get the result.
	\\[0.2cm]
Finally, to prove (e), we momentarily assume the existence of a function $\varphi \in W^{2,\infty}(\R^d)$ satisfying $0 \leq \varphi(x) \leq C_\varphi\|x\|^2$ and $\|\nabla\varphi(x)\| \leq C\sqrt{\varphi(x)}$ for a given constant $C$. Note that $L^2(\R^d)$ embeds continuously in $L^2(\R^d,\varphi dx)$, and thanks to It\^o's formula, 
\begin{align*}
\E \bigg[\int_{\R^d} u^2_{\eps, \gamma}(t) \varphi dx \bigg]- \E \bigg[\int_0^t\langle \bar A(s,u_{\eps, \gamma}(s)),u_{\eps, \gamma}(s)\varphi \rangle ds \bigg]= \frac12\E \bigg[\int_0^t \int_{\R^2}\mathbb{G}^2(u_{\eps, \gamma}(s)) \varphi dxds \bigg]+ \int_{\R^d} (u^\eps_{0})^2 \varphi dx.
\end{align*}
Note that, thanks to assumption \eqref{A4}
\begin{align*}
\mathbb{G}^2(u_{\eps, \gamma}(s)) \varphi \leq K u^2_{\eps, \gamma}(s) \varphi,
\end{align*}
and for any $0<\theta<\frac{1}{C_\varphi^2+C(f')}$
	\begin{align*}
	\langle \bar A(s,u_{\eps, \gamma}(s)),& u_{\eps, \gamma}(s)\varphi \rangle = 
	- \eps \int_{\D} \big[ |\nabla u_{\eps, \gamma}(s)|^2 \varphi + u_{\eps, \gamma}(s) \nabla u_{\eps, \gamma}(s) \nabla \varphi \big] dx
	\\& 
	- \gamma \int_{\D} \big[ |\Delta u_{\eps, \gamma}(s)|^2 \varphi + 2  \Delta u_{\eps, \gamma}(s) \nabla u_{\eps, \gamma}(s)\nabla \varphi + u_{\eps, \gamma}(s) \Delta u_{\eps, \gamma}(s) \Delta \varphi  \big]dx
	\\ &
	- \int_{\D} \big[ \mathscr{L}_{\lambda/2}[A(u_{\eps, \gamma}(s))]\mathscr{L}_{\lambda/2}[u_{\eps, \gamma}(s) \varphi] 
	- \text{div}[f(u_{\eps, \gamma}(s))] u_{\eps, \gamma}(s) \varphi \big] dx
	\\ &\leq  
	\eps \int_{\D} \big[ (\theta C^2_\varphi-1) |\nabla u_{\eps, \gamma}(s)|^2 \varphi +C_\theta u^2_{\eps, \gamma}(s) \big] dx
	\\& 
	+ \gamma \int_{\D} \big[ (\theta C^2_\varphi-1) |\Delta u_{\eps, \gamma}(s)|^2 \varphi + C_\theta \|\nabla u_{\eps, \gamma}(s)\|^2 + \frac12\|\Delta \varphi\|_{\infty}(u^2_{\eps, \gamma}(s) + |\Delta u_{\eps, \gamma}(s)|^2) \big]dx
	\\ &
	- \int_{\D}  \mathscr{L}_{\lambda/2}[A(u_{\eps, \gamma}(s))]\mathscr{L}_{\lambda/2}[u_{\eps, \gamma}(s) \varphi] dx 
	+ C(f')\int_{\D} \eps \theta |\nabla u_{\eps, \gamma}(s)|^2 \varphi + \frac{C_\theta}{\eps} |u_{\eps, \gamma}(s)|^2 \varphi \ dx 
	\\ &\leq 
	C_\theta \Big[ \eps\|u_{\eps, \gamma}(s)\|^2_{L^2} + \gamma \|\nabla u_{\eps, \gamma}(s)\|^2_{L^2}\Big] 
	+ \frac\gamma2 \|\Delta \varphi\|_{\infty}\Big[\|u_{\eps, \gamma}(s)\|^2_{L^2} + \|\Delta u_{\eps, \gamma}(s)\|_{L^2}^2 \Big]
	\\ &
	- \int_{\D}  \mathscr{L}_{\lambda/2}[A(u_{\eps, \gamma}(s))]\mathscr{L}_{\lambda/2}[u_{\eps, \gamma}(s) \varphi] dx + C(\theta,\eps)\int_{\D} |u_{\eps, \gamma}(s)|^2 \varphi \ dx.
	\end{align*}
	Let us be more specific on our choice of the weight-function $\varphi$: set $\widetilde\varphi_k$ the even fonction defined for positive real $x$ by $\widetilde\varphi_k(x)=x^2$ if $x \in[0,k]$, $\widetilde\varphi_k(x)=2k^2-(x-2k)^2$ if $x \in[k,2k]$, and $\widetilde\varphi_k(x)=2k^2$ if $x \geq 2k$ and denote by $\varphi_k(x)=\sum_{i=1}^d \widetilde\varphi_k(x_i)$.  
	\\
	Thus,   one has that $0 \leq \widetilde\varphi_k(x_i) \leq x_i^2$, $|\widetilde\varphi^\prime_k(x_i)| \leq 2 \sqrt{\widetilde\varphi_k(x_i)}$ and $|\widetilde\varphi^{\prime\prime}_k(x_i)| \leq 2$. Thus, $0 \leq \varphi_k(x) \leq \|x\|^2$, $\|\nabla \varphi_k(x)\| = \sqrt{\sum_i [\widetilde\varphi^\prime_k(x_i)]^2}\leq 2\sqrt{\varphi_k(x)}$ and $|\Delta \varphi_k(x)| \leq 2d$.
	\\[0.1cm]
	Let us check the last term of the above estimate, 
	\begin{align*}
	&- \int_{\D}  \mathscr{L}_{\lambda/2}[A(u_{\eps, \gamma}(s))]\mathscr{L}_{\lambda/2}[u_{\eps, \gamma}(s) \varphi_k] dx
	\\ = &
	-\int_{\D} \int_{\D} \frac{[A(u_{\eps, \gamma}(s,x+z))-A(u_{\eps, \gamma}(s,x))][u_{\eps, \gamma}(s,x+z) \varphi_k(x+z)-u_{\eps, \gamma}(s,x) \varphi_k(x)]}{|z|^{d+2\lambda}}dz dx
	\\ = &
	-\sum_i \int_{\D} \int_{\D} \frac{[A(u_{\eps, \gamma}(s,x+z))-A(u_{\eps, \gamma}(s,x))][u_{\eps, \gamma}(s,x+z) \widetilde\varphi_k(x_i+z_i)-u_{\eps, \gamma}(s,x) \widetilde\varphi_k(x_i)]}{|z|^{d+2\lambda}}dz dx
	\end{align*}
	Note that 
	\begin{align*}
	|\widetilde\varphi_k(x_i+z_i)-\widetilde\varphi_k(x_i)| = |\int_0^1 \widetilde\varphi^\prime_k(x_i+tz_i) z_i dt| \leq |z_i| \int_0^1 |\widetilde\varphi^\prime_k(x_i+tz_i)|  dt \leq |z_i| \sup_{t\in[0,1]} \sqrt{\widetilde\varphi_k(x_i+tz_i)}.
	\end{align*}
	Assume on one hand that $x$ and $z$ are such that  $\max(\widetilde\varphi_k(x_i+z_i), \widetilde\varphi_k(x_i))=\widetilde\varphi_k(x_i)$. Thanks to the specific definition of $\widetilde\varphi_k$, 
	\begin{align*}
	|\widetilde\varphi_k(x_i+z_i)-\widetilde\varphi_k(x_i)|^2 \leq |z_i|^2 \widetilde\varphi_k(x_i)
	\end{align*}
	and 
	\begin{align*}
	&-[A(u_{\eps, \gamma}(s,x+z))-A(u_{\eps, \gamma}(s,x))][u_{\eps, \gamma}(s,x+z) \widetilde\varphi_k(x_i+z_i)-u_{\eps, \gamma}(s,x) \widetilde\varphi_k(x_i)]
	\\ =&
	-\widetilde\varphi_k(x_i)[A(u_{\eps, \gamma}(s,x+z))-A(u_{\eps, \gamma}(s,x))][u_{\eps, \gamma}(s,x+z) -u_{\eps, \gamma}(s,x)]
	\\ &
	-u_{\eps, \gamma}(s,x+z)[A(u_{\eps, \gamma}(s,x+z))-A(u_{\eps, \gamma}(s,x))][ \widetilde\varphi_k(x_i+z_i)-\widetilde\varphi_k(x_i)]
	\\ \leq&
	-\widetilde\varphi_k(x_i)[A(u_{\eps, \gamma}(s,x+z))-A(u_{\eps, \gamma}(s,x))][u_{\eps, \gamma}(s,x+z) -u_{\eps, \gamma}(s,x)]
	\\ &
	+\frac{\|A^\prime\|_\infty}{2}|u_{\eps, \gamma}(s,x+z)|^2 |z_i|^2 + \frac1{2\|A^\prime\|_\infty}[A(u_{\eps, \gamma}(s,x+z))-A(u_{\eps, \gamma}(s,x))]^2\widetilde\varphi_k(x_i)
	\\ \leq&
	\frac{\|A^\prime\|_\infty}{2}|u_{\eps, \gamma}(s,x+z)|^2 |z_i|^2 
	\end{align*}
	using the fact that \\
	$[A(u_{\eps, \gamma}(s,x+z))-A(u_{\eps, \gamma}(s,x))]^2\leq \|A^\prime\|_\infty[A(u_{\eps, \gamma}(s,x+z))-A(u_{\eps, \gamma}(s,x))][u_{\eps, \gamma}(s,x+z)-u_{\eps, \gamma}(s,x)]$.
	\\[0.2cm]
	Assume on the other hand that $x$ and $z$ are such that  $\max(\widetilde\varphi_k(x_i+z_i), \widetilde\varphi_k(x_i))=\widetilde\varphi_k(x_i+z_i)$. Similarly, 
	\begin{align*}
	|\widetilde\varphi_k(x_i+z_i)-\widetilde\varphi_k(x_i)|^2 \leq |z_i|^2 \widetilde\varphi_k(x_i+z_i)
	\end{align*}
	and 
	\begin{align*}
	&-[A(u_{\eps, \gamma}(s,x+z))-A(u_{\eps, \gamma}(s,x))][u_{\eps, \gamma}(s,x+z) \widetilde\varphi_k(x_i+z_i)-u_{\eps, \gamma}(s,x) \widetilde\varphi_k(x_i)]
	\\ =&
	-\widetilde\varphi_k(x_i+z_i)[A(u_{\eps, \gamma}(s,x+z))-A(u_{\eps, \gamma}(s,x))][u_{\eps, \gamma}(s,x+z) -u_{\eps, \gamma}(s,x)]
	\\ &
	-u_{\eps, \gamma}(s,x)[A(u_{\eps, \gamma}(s,x+z))-A(u_{\eps, \gamma}(s,x))][ \widetilde\varphi_k(x_i+z_i)-\widetilde\varphi_k(x_i)]
	\\ \leq&
	-\widetilde\varphi_k(x_i+z_i)[A(u_{\eps, \gamma}(s,x+z))-A(u_{\eps, \gamma}(s,x))][u_{\eps, \gamma}(s,x+z) -u_{\eps, \gamma}(s,x)]
	\\ &
	+\frac{\|A^\prime\|_\infty}{2}|u_{\eps, \gamma}(s,x)|^2 |z_i|^2 + \frac1{2\|A^\prime\|_\infty}[A(u_{\eps, \gamma}(s,x+z))-A(u_{\eps, \gamma}(s,x))]^2\widetilde\varphi_k(x_i+z_i)
	\\ \leq&
	\frac{\|A^\prime\|_\infty}{2}|u_{\eps, \gamma}(s,x)|^2 |z_i|^2 .
	\end{align*}
	Thus, one concludes 
	\begin{align*}
&	- \int_{\D}  \mathscr{L}_{\lambda/2}[A(u_{\eps, \gamma}(s))]\mathscr{L}_{\lambda/2}[u_{\eps, \gamma}(s) \varphi_k] dx \\
& \qquad	 \qquad  \qquad \leq 
	C \int_{\R^d}\int_{\R^d} \Big[|u_{\eps, \gamma}(s,x)|^2+|u_{\eps, \gamma}(s,x+z)|^2\Big] \frac{\|z\|^2}{\|z\|^{d+2\lambda}}dxdz
 \leq C\|u_{\eps, \gamma}(s)\|^2_{L^2}
	\end{align*}
	and, thanks to estimate (a), 
	\begin{align*}
	\int_0^t \E \Big[\langle \bar A(s,u_{\eps, \gamma}(s)),u_{\eps, \gamma}(s)\varphi_k \rangle \Big]\,ds  
	\leq 
	C(\eps)\Big(1 + \int_0^t\E \Big[\int_{\D} |u_{\eps, \gamma}(s)|^2 \varphi_k \,dx \Big]\,ds\Big).
	\end{align*}
	Therefore, 
	\begin{align*}
	\E \Big[\int_{\R^d} u^2_{\eps, \gamma}(t) \varphi_k dx\Big] \leq 
	C(\eps)+ (C(\eps)+K)\int_0^t\E \Big[\int_{\D} |u_{\eps, \gamma}(s)|^2 \varphi_k \,dx \Big]\,ds  + \int_{\R^d} (u^\eps_{0})^2 \|x\|^2 dx,
	\end{align*}
	and, by Gronwall's lemma,  $\sup_t\E \Big[\int_{\R^d} u^2_{\eps, \gamma}(t) \varphi_k dx\Big] \leq C(\eps)$ and $\sup_t\E \Big[\int_{\R^d} u^2_{\eps, \gamma}(t) \|x\|^2 dx\Big] \leq C(\eps)$ by passing to the limit over $k$.
\end{proof}


\subsection{Compactness argument}

It is well known that, without assuming any topological structure on $\Omega$, establishing a result of compactness in the probability variable ($\omega$-variable) is a non-trivial task. In what follows, to obtain strong (a.s.) convergence in the $\omega$-variable, we make use of Skorokhod representation theorem, linked to tightness of probability measures and a.s. representations of random variables with Gy\"ongy-Krylov's characterization of convergence in probability adapted to this situation.
\medskip

To proceed, we first denote by $\mu_{\eps,\gamma}$ the joint law of $(u_{\eps, \gamma},W)$ in the space $U=L^2(Q_T) \times C([0,T],\mathbb{H})$.  
Thanks to Lemma \ref{EnergyEstimates}, $(u_{\eps, \gamma})_\gamma$ is a bounded sequence in $L^2(\Omega,L^{2}(0,T,H^{1}(\R^d))$ and \\ $L^2(\Omega,W^{\alpha,2}(0,T,H^{-2}(\R^d)))$; assuming moreover that $u_0 \in L^2(\R^d,\|x\|^2 dx)$ yields the boundedness of $(u_{\eps, \gamma})_\gamma$ in $L^2(\Omega,L^2(\R^d,\|x\|^2dx))$. Since $H^{1}(\R^d)\cap L^2(\R^d,\|x\|^2dx)$ is compactly embedded in $L^2(\R^d)$, so \cite[Corollary 5]{Simon} implies that  $W^{\alpha,2}(0,T,H^{-2}(\R^d)) \cap L^{2}(0,T,H^{1}(\R^d)\cap L^2(\R^d,\|x\|^2dx))$ is compactly embedded in $L^2(Q_T)$. Thus, in light of the above estimates, the following lemma holds.
\begin{lem}
If additionally $u_0 \in L^2(\R^d,\|x\|^2 dx)$, the family of laws $\lbrace \mu_{\eps,\gamma}: \gamma \in (0,1)\rbrace$ is tight on  $(U,\mathcal{B}(U))$.
\end{lem}
Applying Prokhorov compactness theorem to the family of laws $\lbrace \mu_{\eps,\gamma}: \gamma \in (0,1)\rbrace$ and the modified version of Skorokhod representation theorem presented in \cite[Thm. C.1]{BrzezniakHausenblasRazafimandimby}, passing to a weakly convergent subsequence, still denoted $\mu_{\eps,\gamma}$ (\textit{i.e.} $\gamma$ is a decreasing sequence converging to $0$), and denoting by $\mu^{\eps}$ the limit law, we infer the following result:
\begin{lem}
Passing to a subsequence (not relabeled), there exists a probability space $(\tilde \Omega, \tilde{ \mathcal{F}}, \tilde{\mathbb{P}})$, and a family of $U$-valued  random variables $\big\{( \tilde u_{\eps, \gamma}, \tilde W_{\eps, \gamma}): \gamma \in (0,1) \big\}\cup\big\{( \tilde u_{\eps}, \tilde  W_{\eps})\big\}$ such that
\begin{itemize}
\item [$(a)$] for any $\gamma$, $( u_{\eps, \gamma},  W)$ and $( \tilde u_{\eps, \gamma}, \tilde W_{\eps, \gamma})$ have the same law $\mu_{\eps,\gamma}$ on $U$,
\item [$(b)$] the law of $\big( \tilde u_{\eps}, \tilde W_{\eps}\big)$ is given by $\mu^{\eps}$,
\item [$(c)$] $\big( \tilde u_{\eps, \gamma}, \tilde W_{\eps, \gamma}\big)$ converge to $\big( \tilde u_{\eps}, \tilde W_{\eps}\big)$ $\tilde{\mathbb{P}}$-a.s. in $U$,
\item[$(d)$] for any $\gamma$, $\tilde W_{\eps, \gamma}=\tilde W_{\eps}$.
\end{itemize}
\end{lem}
Note that without lost of generality, this new probability space can be considered complete. Moreover, as $u_{\eps, \gamma}$ is also a random variable with values in $C([0,T],L^2(\R^d))$, by \cite[Lem. A.3]{ValletZimmermann}, $\tilde u_{\eps, \gamma}$ is also a continuous processes with values in $L^2(\R^d)$. 
Following \cite[Section 8]{BrzezniakHausenblasRazafimandimby}, setting $(\tilde{\mathcal{F}}_t)$ the filtration generated by the family $(\tilde u_{\eps, \gamma})_\gamma$ and $\tilde W_{\eps}$, one gets, by standard arguments, that $\tilde W_{\eps}$ is a $(\tilde{\mathcal{F}}_t)$, $\mathbb{H}$ valued, cylindrical Wiener process and that $u_{\eps, \gamma}$ is a square integrable $(\tilde{\mathcal{F}}_t)$-predictable process.
\\
Let us remark that by assuming that the initial filtration $(\mathcal{F}_t)$ is the one generated by $W$, then, by \cite[Lem. A.6]{ValletZimmermann}, one can consider that $(\tilde{\mathcal{F}}_t)$ is the filtration generated by $\tilde W_{\eps}$.

Therefore, since $( u_{\eps, \gamma},  W)$ and $( \tilde u_{\eps, \gamma}, \tilde W_{\eps})$ have the same law, by an adapted stepwise approximation of It\^o's integral, one has that $\tilde u_{\eps, \gamma}$ is the solution to equation \eqref{eq:viscous-Brown_one_01} for given $(\tilde\Omega, \tilde{\mathcal{F}},(\tilde{\mathcal{F}}_t),\tilde{\mathbb{P}},\tilde{W_\eps})$.
\\
As a consequence, the estimates of Lemma \ref{EnergyEstimates} hold: there exists a positive constant $C$ independent of $\eps$ and $\gamma$ such that 
\begin{itemize}
	\item [$(a)$] 
	\begin{align*}
	\displaystyle{\sup_{0\le t\le T} \tilde\E\Big[\big\|\tilde u_{\eps, \gamma}(t)\big\|_{L^2(\R^d)}^2\Big]}  &+ \eps \int_0^T \tilde\E\Big[\big\|\grad \tilde u_{\eps, \gamma}(s)\big\|_{L^2(\R^d)}^2\Big]\,ds \\
	&+ \gamma \int_0^T \tilde\E\Big[\big\|\Delta \tilde u_{\eps, \gamma}(s)\big\|_{L^2(\R^d)}^2\Big]\,ds  + \int_0^T \tilde\E\Big[\big\| A(\tilde u_{\eps, \gamma})\|_{H^{\lambda}(\R^d)}^2\Big]\,ds \le C, 
	\end{align*} \vspace{.2cm}
	\item [$(b)$] $\norm{\partial_t \big(\tilde u_{\eps, \gamma} - \int_0^t \h(\tilde u_{\eps, \gamma}(s,\cdot))\, d\tilde W_\eps(s) \big)}_{L^2(\tilde\Omega\times(0,T); H^{-2}(\R^d))} \le C,$\vspace{.2cm}
	\item [$(c)$] $\Big\|\tilde u_{\eps, \gamma} -\int_0^t \h(\tilde u_{\eps,\gamma})\,d\tilde W_\eps \Big\| _{L^2 \big(\tilde\Omega, W^{1,2}(0,T,H^{1}(\D),H^{-2}(\D))\big)} \leq C/\sqrt{\eps}$,
	
	\item [$(d)$] For any $p\geq 2$ and any $0<\alpha<\frac12$,  
	\begin{align*}
	\norm{\int_0^t \h(\tilde u_{\eps, \gamma}) \,d\tilde W_\eps}_{L^p\big(\tilde\Omega,W^{\alpha,p}(0,T,L^2(\D))\big)} \le C\text{ and }\sup_\gamma  \tilde\E \big[\|\tilde u_{\eps, \gamma} \|^p_{C([0,T],L^2(\D))} \big] \le C.
	\end{align*}
\end{itemize}

\subsection{Identification of the limit} 
\label{identify}
It is known that $\tilde u_{\eps, \gamma}$ converges a.s. to $\tilde u_{\eps}$ and in $L^2(Q_T)$ when $\gamma$ goes to $0$, thus, by Vitali theorem and the above estimate (d), one concludes that $\tilde u_{\eps, \gamma}$ converges $\tilde u_{\eps}$ in $L^2(\tilde\Omega\times Q_T)$, thus in $\tilde N_w^2(0,T,L^2(\R^d))$ where a "tilde" is used to distinguish it from the same space on $\Omega$, and a.e. in $\tilde\Omega\times Q_T$ for a subsequence denoted similarly. 

This first result will help us to identify the following weak limits: the above \textit{a priori} estimates yield,
\\[0.1cm]
$\tilde u_{\eps, \gamma}$ converges weakly to $\tilde u_{\eps}$ in $\tilde N_w^2(0,T,H^1(\R^d))$ and *-weakly in $L^\infty(0,T,L^2(\tilde\Omega\times\R^d))$, and $\gamma \tilde u_{\eps, \gamma}$ converges to $0$ in $L^2(\tilde\Omega\times(0,T),H^2(\R^d))$;
\\[0.1cm]
$f(\tilde u_{\eps, \gamma})$ converges to $f(\tilde u_{\eps})$ in $L^2(\tilde\Omega\times Q_T)$;
\\[0.1cm]
$A(\tilde{u}_{\eps,\gamma})$ converges to $A(\tilde{u}_{\eps})$ in $L^2(\tilde\Omega\times Q_T)$ and weakly in $L^2(\tilde\Omega,L^2(0,T,H^\lambda(\R^d)))$;
\\[0.1cm]
$\h(\tilde u_{\eps, \gamma})$ converges to $\h(\tilde u_{\gamma})$ in $L^2(\tilde\Omega,L^2(0,T,HS(\mathbb{H},\R^d)))$ and $\int_0^\cdot \h(\tilde u_{\eps, \gamma}) d\tilde W_\eps$ converges to $\int_0^\cdot \h(\tilde u_{\eps}) d\tilde W_\eps$ in $L^2(\tilde\Omega,C([0,T],L^2(\R^d)))$;
\\[0.1cm]
$\tilde u_{\eps, \gamma}-\int_0^\cdot \h(\tilde u_{\eps, \gamma}) d\tilde W_\eps$ converges to $\tilde u_{\eps}-\int_0^\cdot \h(\tilde u_{\eps}) d\tilde W_\eps$ in $L^2(\tilde\Omega,W^{1,2}(0,T,H^{1}(\D),H^{-2}(\D)))$.

Note that from the two last estimates and  the second part of Estimate (d), $\tilde u_{\eps, \gamma}$ converges weakly to $\tilde u_{\gamma}$ in $C_w([0,T],L^2(\R^d))$ and  $\tilde u_{\eps, \gamma}(t)$ converges weakly to $\tilde u_{\gamma}(t)$ in $L^{2}(\D)$ for any $t$. 
\\[0.2cm]
We are now in position to conclude that $\tilde u_{\eps}$ is a martingale solution to the viscous problem \eqref{eq:viscous-Brown}, under the constraint that $u_0^\eps \in L^2(\R^d,\|x\|^2dx)$. 

\subsection{Pathwise solutions} 
\label{pathwise}
To continue this section, one is able, thanks to Section \ref{AnnexeParab} where the pathwise uniqueness of the solution is proved, to give a result of  existence of a strong (in the sense of probability) solution to the viscous problem \eqref{eq:viscous-Brown} by using  Yamada-Watanabe approach (see \cite{YamadaWatanabe}), relying on  Gy\"ongy-Krylov characterization of convergence in probability (see \cite{GyongyKrylov}). 
\medskip

Finally, let us briefly mention why one can remove the assumption $u_0^\eps \in L^2\big(\R^d,\|x\|^2 dx\big)$, which is only used to obtain compactness information, and prove the result of existence for $u_0^\eps \in H^1(\R^d)\cap L^1(\R^d)$.
\\
Indeed, by a classical density argument: convolution and cut-off, a sequence $(v_\delta) \subset \mathcal{D}(\R^d)$ exits converging to $u_0^\eps$ in $H^1(\R^d)$ and $L^1(\R^d)$. Therefore, by the first step, there exists a corresponding sequence $(u^\eps_\delta)$ of viscous solutions whose initial conditions are the $v_\delta$. 
Note that the estimates (a) to (d) of Lemma \ref{EnergyEstimates} still hold, independently of $\delta$, and that \eqref{Uniq-vicous} adds an information of Cauchy sequence of $(u_\delta^\eps)$ in $L^\infty(0,T,L^1(\Omega\times\R^d))$.  
This last information replaces the compactness argument used to obtain an a.e. convergence, needed to identify the weak limits. Then, one is able to reproduce the above Subsection~\ref{identify}, by using weak limits and in particular the strong/weak continuity property of linear operators (like It\^o's integral) between two Banach spaces. This yields the result of existence of Theorem~\ref{prop:vanishing viscosity-solution}.


\subsection{Uniqueness of viscous solutions}
\label{AnnexeParab} \hfill 

Let $u_\ep(s,y)$ and $v_\ep(t,x)$ be two solutions of (\ref{eq:viscous-Brown})  . Applying It\^{o}'s formula to $\eta( u^\gamma_\ep(s,y)-k) \psi(t,x,s,y)$, multiplying by $\rho_l( u_\theta(t,x)-k)$, taking the expectation, and integrating with respect to $t,x,k$, an application of Fubini's theorem yield 

\begin{align}
 0\leq & 
\E\bigg[\dint_{Q_T\times\R\times \D}\eta(u^\gamma_\eps(0)-k)\psi(t,x,0,y)  \rho_l(v_\eps(t,x)-k) \,dy\,dk\,dx\,dt\bigg] \notag
\\ & 
+ \E \bigg[\dint_{\R\times Q^2_T} \eta(u^\gamma_\eps(s,y)-k) \partial_s \psi(t,x,s,y) \rho_l(v_\eps(t,x)-k)\,dy\,ds\,dk\,dx\,dt\bigg]\notag\\
&+ \E\bigg[\sum_{k \ge 1} \dint_{\R\times Q^2_T}  \eta'(u^\gamma_\eps(s,y)-k) \psi(t,x,s,y) g_k(u_\eps)^\gamma(s,y) \rho_l(v_\eps(t,x)-k)\, d\beta_k(s)\,dy \,dk\,dx\,dt\bigg] \notag
\\&
+ \frac{1}{2}\E\bigg[\dint_{\R\times Q^2_T} \eta^{\prime\prime}(u^\gamma_\eps(s,y)-k)  (\mathbb{G}(u_\eps)^\gamma)^2 \psi(t,x,s,y)\rho_l(v_\eps(t,x)-k) \,ds\,dy\,dk\,dx\,dt\bigg]\notag\\
  & -\E\bigg[\dint_{\R\times Q^2_T} F^\eta( u^\gamma_\ep,k)\cdot \nabla_y \psi(t,x,s,y)\rho_l( v_\eps(t,x)-k)\,  dy\,ds\,dk\,dx\,dt\bigg]\notag \\
&- \eps \E\bigg[\dint_{\R\times Q^2_T}\eta'(u^\gamma_\eps(s,y)-k) \nabla_y u^\gamma_\eps(s,y) \nabla_y\psi(t,x,s,y)\rho_l(v_\eps(t,x)-k) \,dy\,ds\,dk\,dx\,dt\bigg]\notag
\\
& -\eps\E\bigg[\dint_{\R\times Q^2_T}\eta''(u^\gamma_\eps(s,y)-k) |\nabla_y u^\gamma_\eps(s,y)|^2 \psi(t,x,s,y) \rho_l(v_\eps(t,x)-k)\,dy\,ds\,dk\,dx\,dt\bigg]\notag\\
& -\E \bigg[\int_{\R \times Q^2_T} \mathscr{L}^r_\lambda[A(u_\eps)^\gamma(s,\cdot)](y) \psi(t,x,s,y)\, \eta^\prime( u^\gamma_\eps(s,y)-k) \rho_l(v_\eps(t,x)-k) \,dy \,ds\,dk\,dx\,dt\bigg]\notag\\
& -\E \bigg[\int_{\R \times Q^2_T} \mathscr{L}_{\lambda,r}[A(u_\eps)^\gamma(s,\cdot)](y) \psi(t,x,s,y) \eta^\prime( u^\gamma_\eps(s,y)-k)\rho_l(v_\eps(t,x)-k) \,ds\,dy\,dk\,dx\,dt\bigg]\notag\\
& =:  I_1 + I_2 + I_3 +I_4 + I_5 + I_6 + I_7 + I_8 +I_9.
\label{kato_one}
\end{align}

We now write  It\^o's formula  for $v_\eps(t,x)$, and then multiply by $\rho_l[u^\gamma_\eps(s,y)-k]$ and  integrate with respect to $k,y,s$ and then take expectation to get
\begin{align}
 0\leq & 
\E\bigg[\dint_{Q_T\times\R\times \D}\eta(u_\eps(0)-k)\psi(0,x,s,y)  \rho_l(u^\gamma_\eps(s,y)-k) \,dy\,dk\,dx\,ds\bigg] \notag
\\ & 
+ \E \bigg[\dint_{\R\times Q^2_T} \eta(v_\eps(t,x)-k) \partial_t \psi(t,x,s,y) \rho_l(u^\gamma_\eps(s,y)-k)\,dy\,ds\,dk\,dx\,dt\bigg]\notag\\
& +\E\bigg[\sum_{k \ge1}\dint_{\R\times Q^2_T} \eta'(v_\eps(t,x)-k) \psi(t,x,s,y) g_k(v_\eps(t,x)) \, d\beta_k(t)\rho_l(u^\gamma_\eps(s,y)-k)\,dy \,dk\,dx\,ds\bigg] \notag
\\&
+ \frac{1}{2}\E\bigg[\dint_{\R\times Q^2_T} \eta^{\prime\prime}(v_\eps(t,x)-k)  \mathbb{G}^2(v_\eps(t,x)) \psi(t,x,s,y)\rho_l(u^\gamma_\eps(s,y)-k) \,ds\,dy\,dk\,dx\,dt\bigg]\notag\\
&-\E\bigg[\dint_{\R\times Q^2_T} F^\eta(v_\eps(t,x),k). \nabla_x \psi(t,x,s,y) \rho_l(u^\gamma_\eps(s,y)-k) \,dy\,ds\,dk\,dx\,dt\bigg] \notag
\\
&- \eps \E\bigg[\dint_{\R\times Q^2_T}\eta'(v_\eps(t,x)-k) \nabla_x v_\eps(t,x) .\nabla_x\psi(t,x,s,y) \rho_l(u^\gamma_\eps(s,y)-k) \,dy\,ds\,dk\,dx\,dt\bigg]\notag
\\
&-\eps \E\bigg[\dint_{\R\times Q^2_T}\eta''(v_\eps(t,x)-k) |\nabla_x v_\eps(t,x)|^2 \psi(t,x,s,y) \rho_l(u^\gamma_\eps(s,y)-k) \,dy\,ds\,dk\,dx\,dt\bigg]\notag\\
& -\E \bigg[\int_{\R \times Q^2_T} \mathscr{L}^r_\lambda[A(v_\eps)(t,\cdot)](x) \psi(t,x,s,y)\, \eta^\prime( v_\eps(t,x)-k) \rho_l(u^\gamma_\eps(s,y)-k) \,dy \,ds\,dk\,dx\,dt\bigg]\notag\\
& -\E \bigg[\int_{\R \times Q^2_T} \mathscr{L}_{\lambda,r}[A(v_\eps)(t,\cdot)](x) \psi(t,x,s,y) \eta^\prime( v_\eps(t,x)-k)\rho_l(u^\gamma_\eps(s,y)-k) \,ds\,dy\,dk\,dx\,dt\bigg]\notag\\
& =:  J_1 + J_2 + J_3 +J_4 + J_5 + J_6 + J_7 + J_8+J_9.
\label{Kato_second}
\end{align}

We shall now add \eqref{kato_one} and \eqref{Kato_second}, and pass to the limit with respect the parameters in a fixed order: first $n$, then $\gamma, \delta, l, r$ and finally $m \to \infty$.\hfill

We shall consider the terms $I_6, I_7$ and $J_6,J_7$, rest of the terms can be treated exactly in the same way as done in the uniqueness method o Section \ref{uniqueness}.
We have the following lemmas
\begin{lem}
\label{viscous1}
\begin{align*}
A:=&\lim_{\delta \to 0} \lim_{\gamma \to 0} \lim_{n \to \infty}( I_6+J_6)
\\=& - \eps \E\bigg[\dint_{\R\times \D \times Q_T}\sgn(u_\eps(t,y)-k) \nabla_y u_\eps(t,y) .\nabla_y \rho_m(x-y)
\varphi(t,x) \rho_l(v_\eps(t,x)-k) \,dy\,dk\,dx\,dt\bigg]\\
& - \eps \E\bigg[\dint_{\R\times \D \times Q_T}\sgn(v_\eps(t,x)-k) \nabla_x v_\eps(t,x) .\nabla_x( \varphi(t,x) \rho_m(x-y))
 \rho_l(u_\eps(t,y)-k) \,dy\,dk\,dx\,dt\bigg].
\end{align*}
\end{lem}

\begin{lem}
\label{viscous2}
\begin{align*}
B:=&\lim_{\delta \to 0} \lim_{\gamma \to 0} \lim_{n \to \infty}( I_7 + J_7)
\\ =& -2\ep\Bigg( \E \int_{ \D \times Q_T} ( | \nabla_x v_\ep(t,y)|^2+ | \nabla_y u_\ep(t,x)|^2) \rho_l(u_\ep(t,x)- v_\ep(t,y))
 \varphi(t,x) \rho_m(x-y) \,dy\,dx\,dt \Bigg).
\end{align*}

\begin{proof}
After an integration by parts on k, we get 
\begin{align*}
 &\lim_{\gamma \to 0} \lim_{ n \to \infty}(I_7 + J_7) \\
 &\qquad =    \eps \E\bigg[\dint_{\R\times \D \times Q_T}\eta'(u_\eps(t,y)-k)\rho'_l(v_\eps(t,x)-k) |\nabla_y u_\eps(t,y)|^2 \varphi(t,x) \rho_m(x-y) \,dy\,ds\,dk\,dx\,dt\bigg]\\
 &\qquad +\eps \E\bigg[\dint_{\R\times \D \times Q_T}\eta'(v_\eps(t,x)-k)\rho'_l(u_\eps(t,y)-k) |\nabla_x v_\eps(t,x)|^2  \varphi(t,x) \rho_m(x-y) \,dy\,dk\,dx\,dt\bigg]\\
 & \underset{ \delta \to 0} \longrightarrow \eps \E\bigg[\dint_{\R\times \D \times Q_T}\sgn(u_\eps(t,y)-k)\rho'_l(v_\eps(t,x)-k) |\nabla_y u_\eps(t,y)|^2 \varphi(t,x) \rho_m(x-y) \,dy\,ds\,dk\,dx\,dt\bigg]\\
 &+\eps \E\bigg[\dint_{\R\times \D \times Q_T}\sgn(v_\eps(t,x)-k)\rho'_l(u_\eps(t,y)-k) |\nabla_x v_\eps(t,x)|^2  \varphi(t,x) \rho_m(x-y) \,dy\,dk\,dx\,dt\bigg]\\
 &=-2\ep \bigg[ \E \int_{ \D \times Q_T} ( | \nabla_x v_\ep(t,y)|^2+ | \nabla_y u_\ep(t,x)|^2) \rho_l(u_\ep(t,x)- v_\ep(t,y)) \varphi(t,x) \rho_m(x-y) \,dy\,dx\,dt \bigg]\\
\end{align*}
where we have used the fact that $\int_\R \sgn ( u_\ep(t,y)-k) \rho'_l(v_\ep(t,x)-k)\,dk = \int_\R \sgn( v_\ep(t,x)-k) \rho'_l(u_\ep(t,y)-k) \,dk= -2\rho_l(u_\ep(t,x)-v_\ep(t,y))$.
\end{proof}
\end{lem}
\begin{lem}
\label{viscous3}
 $$\lim_{m \to \infty}\limsup_{l \to 0}(A+B) \le \ep \E \bigg[ \int_{Q_T} |u_\ep(t,x)-v_\ep(t,x)| \Delta \varphi (t,x) \,dx\,dt \bigg]$$ 
 \begin{proof}

\begin{align*}
A+B=& - \eps \E\bigg[\dint_{\R\times \D \times Q_T}\sgn(u_\eps(t,y)-k) \nabla_y u_\eps(t,y) .\nabla_y \rho_m(x-y)\varphi(t,x) \rho_l(v_\eps(t,x)-k) \,dy\,dk\,dx\,dt\bigg]\\
& - \eps \E\bigg[\dint_{\R\times \D \times Q_T}\sgn(v_\eps(t,x)-k) \nabla_x v_\eps(t,x) .\nabla_x( \varphi(t,x) \rho_m(x-y)) \rho_l(u_\eps(t,y)-k) \,dy\,dk\,dx\,dt\bigg]\\
&+2\ep \bigg[ \E \int_{ \D \times Q_T} ( | \nabla_x v_\ep(t,y)|^2+ | \nabla_y u_\ep(t,x)|^2) \rho_l(u_\ep(t,x)- v_\ep(t,y)) \varphi(t,x) \rho_m(x-y) \,dy\,dx\,dt \bigg]\\
=&- \eps \E\bigg[\dint_{\R\times \D \times Q_T}\sgn(u_\eps(t,y)-k) \nabla_y u_\eps(t,y) .(\nabla_y (\rho_m(x-y) \varphi(t,x))+\nabla_x(\rho_m(x-y) \varphi(t,x)))\\[-0.3cm]
&\hspace{10cm} \times \rho_l(v_\eps(t,x)-k) \,dy\,dk\,dx\,dt\bigg]\\
& - \eps \E\bigg[\dint_{\R\times \D \times Q_T}\sgn(v_\eps(t,x)-k) \nabla_x v_\eps(t,x) .(\nabla_x( \varphi(t,x) \rho_m(x-y)+\nabla_y( \varphi(t,x) \rho_m(x-y)))\\[-0.3cm]
&\hspace{10cm} \times \rho_l(u_\eps(t,y)-k) \,dy\,dk\,dx\,dt\bigg]\\
&-2\ep  \E\bigg[ \int_{ \D \times Q_T} ( | \nabla_x v_\ep(t,y)|^2+ | \nabla_y u_\ep(t,x)|^2) \rho_l(u_\ep(t,x)- v_\ep(t,y)) \varphi(t,x) \rho_m(x-y) \,dy\,dx\,dt \bigg]\\
&+\ep \E \bigg[ {\int_{\R\times \D \times Q_T}} \sgn(u_\ep(t,y)-k) \nabla_y u_\ep(t,y). \nabla_x( \rho_m(x-y) \varphi(t,x)) \rho_l(v_\ep(t,x)-k)\,dy\,dk\,dx\,dt \bigg]\\
&+ \ep \E \bigg[ {\int_{\R\times \D \times Q_T}} \sgn(v_\ep(t,x)-k) \nabla_x v_\ep(t,x). \nabla_y( \rho_m(x-y) \varphi(t,x)) \rho_l(u_\ep(t,y)-k)\,dy\,dk\,dx\,dt \bigg]\\
=& -\ep \E\bigg[\dint_{\R\times \D \times Q_T} \sgn(u_\eps(t,y)-v_\ep(t,x)+k) \nabla_y u_\eps(t,y) .\nabla_{x+y}( \rho_m(x-y) \varphi(t,x)) \rho_l(k) \,dy\,dk\,dx\,dt\bigg]\\
& - \eps \E\bigg[\dint_{\R\times \D \times Q_T}\sgn(v_\eps(t,x)-u_\ep(t,y)-k) \nabla_x v_\eps(t,x) . \nabla_{x+y}( \varphi(t,x) \rho_m(x-y))\rho_l(k) \,dy\,dk\,dx\,dt\bigg]\\
&-2\ep  \E\bigg[ \int_{ \D \times Q_T} ( | \nabla_x v_\ep(t,y)|^2+ | \nabla_y u_\ep(t,x)|^2) \rho_l(u_\ep(t,x)- v_\ep(t,y)) \varphi(t,x) \rho_m(x-y) \,dy\,dx\,dt \bigg]\\
&-\ep \E \bigg[ {\int_{\R\times \D \times Q_T}} \sgn(u_\ep(t,y)-k) \nabla_y u_\ep(t,y). \nabla_x v_\ep(t,x) \rho_m(x-y) \varphi(t,x) \rho'_l(v_\ep(t,x)-k)\,dy\,dk\,dx\,dt \bigg]\\
&- \ep \E \bigg[ {\int_{\R\times \D \times Q_T}} \sgn(v_\ep(t,x)-k) \nabla_x v_\ep(t,x). \nabla_y u_\ep(t,y) \rho_m(x-y) \varphi(t,x) \rho'_l(u_\ep(t,y)-k)\,dy\,dk\,dx\,dt \bigg]\\
=&-\ep \E\bigg[\dint_{\R\times \D \times Q_T} \sgn( u_\eps(t,y)-v_\ep(t,x)+k)( \nabla_y u_\ep(t,y) - \nabla_x v_\ep (t,x)). \nabla_{x+y}( \varphi(t,x) \rho_m(x-y))\\[-0.3cm]
& \hspace{12cm} \times \rho_l(k) \,dy\,dk\,dx\,dt \bigg]\\
&-2\ep  \E\bigg[ \int_{ \D \times Q_T} ( | \nabla_x v_\ep(t,y)|^2+ | \nabla_y u_\ep(t,x)|^2) \rho_l(u_\ep(t,x)- v_\ep(t,y)) \varphi(t,x) \rho_m(x-y) \,dy\,dx\,dt \bigg]\\
&-4 \ep \E \bigg[ \int_{ \D \times Q_T} \nabla_x v_\ep(t,x) . \nabla_y u_\ep (t,y) \rho_l( v_\ep(t,x) - u_\ep (t,y)) \varphi(t,x) \rho_m(x-y) \,dy \,dx \,dt \bigg]\\
=&-\ep \E\bigg[\dint_{\R\times \D \times Q_T} \nabla_{x+y} | u_\ep(t,y)-v_\ep(t,x)+k| . \nabla_{x+y}(\varphi(t,x) \rho_m(x-y))\rho_l(k) \,dy\,dk\,dx\,dt\\
& - 2 \ep  \E\bigg[ \int_{ \D \times Q_T}  | \nabla_x v_\ep(t,y)+ \nabla_y u_\ep(t,x)|^2 \rho_l(u_\ep(t,x)- v_\ep(t,y)) \varphi(t,x) \rho_m(x-y) \,dy\,dx\,dt \bigg]\\
\le& \ep \E \bigg[\dint_{\R\times \D \times Q_T}  | u_\ep(t,y)-v_\ep(t,x)+k| \nabla_{x+y}.\nabla_{x+y}(\varphi(t,x) \rho_m(x-y))\rho_l(k) \,dy\,dk\,dx\,dt\bigg]\\
=& \ep \E \bigg[\dint_{\R\times \D \times Q_T}  | u_\ep(t,y)-v_\ep(t,x)+k| \nabla_{x+y}.(\nabla_x\varphi(t,x) \rho_m(x-y)) \rho_l(k) \,dy\,dk\,dx\,dt\bigg]\\
 =& \ep \E \bigg[\dint_{\R\times \D \times Q_T}  | u_\ep(t,y)-v_\ep(t,x)+k| \Delta(\varphi(t,x)) \rho_m(x-y) \rho_l(k) \,dy\,dk\,dx\,dt\bigg]\\
& \overset { \l \to 0} \longrightarrow \ep\E \bigg[\dint_{\D \times Q_T}  | u_\ep(t,y)-v_\ep(t,x)| \Delta(\varphi(t,x)) \rho_m(x-y) \,dy\,dx\,dt\bigg]\\
& \overset { m \to \infty} \longrightarrow \ep \E \bigg[\dint_{\times Q_T}  | u_\ep(t,x)-v_\ep(t,x)| \Delta(\varphi(t,x)) \,dx\,dt\bigg]
\end{align*}
 \end{proof}
 \end{lem}
Calculating the other terms exactly in the same fashion, as in Section \ref{uniqueness}, and  making use of Lemma~\ref{viscous1}, Lemma~\ref{viscous2} and Lemma~\ref{viscous3}, one obtains the following Kato's inequality:
 \begin{align*}
0 \leq& \int_\D |u_0-v_0|\varphi(0)\,dx + \E\bigg[\int_{Q_T}   |u_\ep(t,x) -v_\ep(t,x)| \partial_t \varphi(t,x)\,dx\,dt\bigg]\\
& +\ep\E \bigg[\dint_{Q_T}  | u_\ep(t,x)-v_\ep(t,x)| \Delta \varphi(t,x) \,dx\,dt\bigg] -\E\bigg[\int_{Q_T}   F(u_\ep(t,x),v_\ep(t,x))\nabla \varphi(t,x)\,dx\,dt 
\bigg]\\
&\quad -\E\bigg[\int_{Q_T}  |A(u_\ep(t,x))-A(v_\ep(t,x))|  \mathscr{L}_\lambda[\varphi(t,\cdot)](x) \,dx \,dt\bigg],
\end{align*} 
\noindent \textit{a priori} for any non-negative $\varphi \in \mathcal{D}([0,T[\times\D)$, but for any non-negative $\varphi \in L^2(0,T,H^2(\D))\cap H^1(Q_T)$ by a density argument.

Finally, to show the result of uniqueness we follow the arguments of Subsection~\ref{stability}, for the two viscous solutions $u_\ep$ and $v_\ep$ with the same initial data. Indeed, we have from Kato's inequality
\begin{align*}
0 \leq& \E\Big[\int_{Q_T} |u_\ep(t,x) -v_\ep(t,x)| \partial_t \varphi(t,x) + \ep |u_\ep(t,x)-v_\ep(t,x)| \Delta \varphi(t,x)-F(u_\ep(t,x),v_\ep(t,x)).\nabla \varphi(t,x)\,dx\,dt\Big]\\
&\hspace{2cm}-\E\Big[\int_{Q_T}   \mathscr{L}_{\lambda/2}\big[|A(u_\ep(t,\cdot))-A(v_\ep(t,\cdot))|\big](x)  \mathscr{L}_{\lambda/2}[\varphi(t,\cdot)](x) \,dx\,dt\Big].
\end{align*}
Let $\rho$ be a space mollifier and $\psi_R$ be as defined before, then $\psi_R \star \rho \to 0$ in $H^\lambda(\Rd)$ and  $\mathscr{L}_{\lambda/2}(\psi_R\star \rho) \to 0$ in $L^2(\Omega\times(0,T),L^2(\D))$.
Then,  for any $\theta \in \mathcal{D}([0,T))$ non-negative, choosing $\varphi= \theta(t) \psi_R \ast \rho$ we get
\begin{align*}
&0 \le  \E\bigg[\int_{Q_T} \theta^\prime(t) |u_\ep(t,x) -v_\ep(t,x)| \psi_R\star\rho (x) +\ep |u_\ep(t,x) -v_\ep(t,x)| \Delta (\psi_R \star \rho)\\
& + \E\bigg[\int_{Q_T}-F(u_\ep(t,x),v_\ep(t,x))\nabla (\psi_R\star\rho) (x) - \theta(t)  \mathscr{L}_{\lambda/2}(|A(u_\ep(t,\cdot))-A_\ep(v(t,\cdot))|)(x)  \mathscr{L}_{\lambda/2}(\psi_R\star\rho)(x) \,dx\,dt\bigg].
\end{align*}
Since $|\nabla (\psi_R\star\rho) (x)| \leq |\nabla \psi_R|\star\rho(x) \leq \frac{c}{R} \psi_R\star\rho (x)$ and $| \Delta( \psi_R \ast \rho)(x)| \le \frac{c}{R^2}\psi_R \star\rho(x)$,
\begin{align*}
0 \leq&  \E\bigg[\int_{Q_T}  |u_\ep(t,x) -v_\ep(t,x)|  \psi_R\star\rho (x) \bigg[ \theta^\prime(t)  +c\bigg(\frac{1}{R}+\frac{\ep}{R^2}\bigg) \theta(t)\bigg] \,dx\,dt\bigg]\\
&-\E\bigg[\int_{Q_T} \theta(t)  \mathscr{L}_{\lambda/2}(|A(u_\ep(t,\cdot))-A(v_\ep(t,\cdot))|)(x)  \mathscr{L}_{\lambda/2}(\psi_R\star\rho)(x) \,dx\,dt\bigg].
\end{align*}
Replacing $\theta(t)$ by $\theta(t)e^{-ct\big(\frac{1}{R}+\frac{\ep}{R^2}\big)}$, one has that 
\begin{align*}
0 \leq& \int_\D |u_0-v_0|\theta(0)\psi_R\star\rho\,dx + \E\int_{Q_T}  |u_\ep(t,x) -v_\ep(t,x)|  \psi_R\star\rho (x) \theta^\prime(t) e^{-ct\big(\frac{1}{R}+\frac{\ep}{R^2}\big)} \,dx\,dt 
\\&
+\|\theta\|_\infty \| \mathscr{L}_{\lambda/2}(|A(u_\ep(t,x))-A(v_\ep(t,x))|)\|_{L^2(\Omega\times(0,T),H^{\lambda/2}(\D))}  \|\mathscr{L}_{\lambda/2}(\psi_R\star\rho)\|_{L^2(\Omega\times(0,T),H^{\lambda/2}(\D))}.
\end{align*}
Choosing $\theta$ in such a way that it is a non-increasing function with $\theta(0)=1$, then one gets, passing to the limit when $R \to \infty$, for $t$ a.e. in $(0,T)$ first, then all $t$ since $u_\ep$ and $v_\ep$ are continuous  processes, 
\begin{align}\label{Uniq-vicous}
\E\bigg[\int_{\D} |u_\ep(t,x) -v_\ep(t,x)| \,dx\bigg] \leq \int_\D |u_0-v_0| \,dx.
\end{align}
As the initial profiles are same, we get
\begin{align*}
\E\bigg[\int_{\D} |u_\ep(t,x) -v_\ep(t,x)| \,dx\bigg] \leq 0.
\end{align*}
This concludes the pathwise uniqueness of viscous solution.

 
\subsection{Bounded variation estimates}
\label{sec:apriori+existence}
In this section, we collect some \textit{a priori} estimates for the viscous solutions \textit{i.e.}, solutions of \eqref{eq:viscous-Brown}. To that context, we have the following results.
\begin{lem}
The viscous solution to \eqref{eq:viscous-Brown} has a.s. continuous trajectories with values in $L^2(\R^d)$ and satisfies, independently of $\eps$, the following estimate: 
\begin{align*}
\sup_{0\le t\le T} \E\Big[\big\|u_\eps(t)\big\|_{L^2(\R^d)}^2\Big]  + \eps \int_0^T \E\Big[\big\|\grad u_\eps(s)\big\|_{L^2(\R^d)}^2\Big]\,ds + \int_0^T \E\Big[\big\| A(u_\eps(s))\|_{H^{\lambda}(\R^d)}^2\Big]\,ds \le C.
\end{align*}
\end{lem}

\begin{proof}
For a proof, follow \textit{e.g.} \cite[Thm. 4.2.5]{PrevotRockner},
\end{proof}

In view of the well-posedness results (cf. Section~\ref{uniqueness}) and the properties of the convergence in the sense of Young measures, we conclude that under the assumptions \ref{A1}-\ref{A4}, the family $\{u_\eps\}_{\eps>0}$ converges to the unique entropy solution $u$ of the underlying problem \eqref{eq:stoc_frac}, weakly in $L^2(\Omega\times(0,T)\times\R^d)$ and strongly $L^p(\Omega\times(0,T)\times B_{\R^d}(0,M))$ for any positive $M$ and any $p\in [1,2)$.

Next, we state some results concerning the uniform spatial BV bound for the viscous solutions and. as corollary, also for the solution of \eqref{eq:stoc_frac}. Indeed, we have the following theorem. 
\begin{thm}
\label{thm:bv-viscous}
Let the assumptions $\ref{A1}$-$\ref{A4}$ hold. For $\eps>0$, let $u_\eps$ be the solution to the Cauchy problem \eqref{eq:viscous-Brown} and assume that $u_0 \in L^1(\R^d)$, resp. $u_0 \in BV(\R^d)$.
Then there exists a constant $C>0$, independent of $\eps$ and $t>0$, such that 
\begin{align*}
 \E\Big[ \|u_\eps(t)\|_{L^1(\R^d)}\Big] \le  C \, \|u_0\|_{L^1(\R^d)}, \quad
&\text{resp.}\quad  \E\Big[ TV_x(u_\eps(t))\Big] \le TV_x(u_0).
 \end{align*}
\end{thm}

\begin{proof}
For the details of the proof of the $L^1$-property, we refer to \cite[Appendix A]{BhKoleyVa}. The total variation estimate is a direct consequence of the stability result \eqref{Uniq-vicous} by noticing that $v_\ep(t,x)=u_\ep(t,x+h)$ if $v_0(x)=u_0(x+h)$.
\end{proof}
	
\begin{rem}
By using Fatou's lemma for the $L^1(\R^d)$-norm on $L^2(\R^d)$ and the fact that the total variation $TV(u)=\sup\Big\{\int_{\R^d}u \mbox{div} \vec \varphi dx,\ \vec \varphi \in C^1_c(\R^d)^d,\ \|\vec \varphi\|_\infty\leq 1\Big\}$ is given by a supremum of a family of continuous linear functions on $L^2(\R^d)$, one gets that the application $u \in L^2(\R^d) \mapsto \|u\|_{BV(\R^d)}\in [0,+\infty]$ is a lower semi-continuous convex mapping. In particular, it is a Borel function.
\\
Again, Fatou's lemma yields $u \in L^2(\Omega\times\R^d) \mapsto \E[\|u\|_{BV(\R^d)}]\in [0,+\infty]$ is a lower semi-continuous convex mapping and, thanks to the properties of supremum of such functions, $u \in L^2(\Omega\times Q_T) \mapsto \supess_{t}\E[\|u\|_{BV(\R^d)}]\in [0,+\infty]$ is a lower semi-continuous convex mapping too.
\end{rem}

 Now, assuming that $u_0 \in BV(\R^d)$, our aim is to show that $\omega$ a.s., $u(\omega)$ is actually a spatial BV solution of \eqref{eq:stoc_frac} provided the initial function $u_0$ lies in $L^2 \cap BV(\R^d)$. 
Since $u_\eps$ converges to $u$ weakly in $L^2(\Omega\times Q_T)$, we have, since the lower semi-continuity property holds also for the weak convergence for convex functions,
\begin{align*}
\E\Big[\int_{Q_T}|u| \,dx\,dt\Big] \leq \lim\inf_\eps \E\Big[\int_{Q_T}|u_\eps| \,dx\,dt\Big] \leq M,
\end{align*}
thanks to Theorem~\ref{thm:bv-viscous}, and $u \in L^1(\Omega\times Q_T)$.
\\
Again, in view of the lower semi-continuity properties of the above remark,  there is a set $Z \subset (0,T)$ of full measure such that  for all $t \in Z$,
\begin{align*}
 \E\Big[ \|u(t)\|_{BV(\R^d)}\Big] \le  \supess_s \E\Big[ \|u(s)\|_{BV(\R^d)}\Big] \le \liminf_{\eps \goto 0} \supess_s \E\Big[ \|u_\eps(s)\|_{BV(\R^d)}\Big] \le \E \Big[ TV_x(u_0)\Big],
\end{align*}
where the last inequality follows from Theorem~\ref{thm:bv-viscous}. 
\\
Using now the facts that $u \in C_w([0,T],L^2(\Omega\times\R^d))$ and that $Z$ is dense in $[0,T]$, any $t$ is a limit of a sequence $(t_n)\subset Z$ with the information that $u(t_n) \rightharpoonup u(t)$ in $L^2(\Omega\times\R^d)$. The above remark and the weak lower semi-continuity property for convex functions yield
\begin{align*}
 \E\Big[ \|u(t)\|_{BV(\R^d)}\Big] \le  \liminf_{n \goto +\infty}  \E\Big[ \|u(t_n)\|_{BV(\R^d)}\Big] \le \E \Big[ TV_x(u_0)\Big].
\end{align*}
In other words, we have the existence of the ``BV entropy solution" for problem \eqref{eq:stoc_frac} given by the following theorem.
 \begin{thm} [BV-entropy solution]
 \label{thm:existence-bv}
Suppose that the assumptions \ref{A1}-\ref{A4} hold with $u_0 \in BV(\R^d)$.
Then the unique entropy solution $u$ of \eqref{eq:stoc_frac} is a BV-entropy solution in the sense that for all $t>0$
\begin{align*}
   \E \Big[\|u(t)\|_{BV(\R^d)} \Big] \le \E \Big[|u_0|_{BV(\R^d)} \Big].
\end{align*}	
\end{thm} 
\begin{rem}
Let us notice that $\E \Big[\|u(t,\cdot)\|_{BV(\R^d)} \Big] < +\infty$ implies that $u(t,\omega) \in BV(\R^d)$ for any $t$ and almost all $\omega$, but we do not claim that it is measurable with values in $BV(\R^d)$. 
\end{rem}


\section{Extension to an Explicit Space-Dependent Noise Coefficient}
\label{extension}
In this section, we consider a larger class of stochastic nonlocal degenerate equations driven by Brownian noise of the type
\begin{equation}
\label{eq:stoc_frac_001}
\begin{cases} 
du(t,x) + \Big[\mathscr{L}_{\lambda}[A(u(t,\cdot))](x)- \Div f(u(t,x))\Big]\,dt
=\h(x,u(t,x)) \,dW(t), & \quad \text{in } Q_T, \\
u(0,x) = u_0(x), & \quad \text{in } \D,
\end{cases}
\end{equation}
Here we assume that
$\Phi(x,u(t,x))$ satisfies the following assumptions: for each $z\in L^2(\mathbb{R}^d)$ we consider a mapping $\Phi: \mathbb{H} \to L^2(\mathbb{R}^d)$ defined by $\Phi(z)e_k = g_k(\cdot, z(\cdot))$. Thus we may also define
$$\Phi(x,u)=\sum_{k\ge1} g_k(x,u)e_k.$$
We assume $g_k\in C(\mathbb{R}^d \times\mathbb{R})$, $g_k(x,0)=0$ for all $x \in \R^d$, with the bounds
\begin{align}\label{2.3} 
\sum_{k\ge1}|g_k(x,u)-g_k(y,v)|^2\le D_1\big(|x-y|^{2}+|u-v|^2\big),
\end{align}
where $x,y\in\mathbb{R}^d$; $ u,v\in\mathbb{R}$. Observe that, the noise coefficient $\Phi(x,u(t,x))$ depends explicitly on the spatial position $x$. Due to some technical difficulties, here we restrict ourselves to the case corresponding to $\lambda <1/2$. 
\subsection{Extension to $\h(x,u)$}
We remark that the explicit dependency on the spatial variable forces us to rearrange the order in computing limits with respect to various parameters to obtain Kato's inequality \eqref{kato}. In what follows, we first pass to the limits in $n,\gamma$ and $\delta$. This is the same as in the uniqueness proof in Section~\ref{uniqueness}, but then we pass to the limits in $\ep$ and $\theta$, and finally pass to the limits in the rest of the parameters: $l,r,m$ simultaneously, to obtain the desired Kato's inequality \eqref{kato}. Let us briefly mention the corresponding changes in the lemmas dealing with various terms in Section~\ref{uniqueness}.
\medskip

\noindent To that context, note that a straightforward adaptation of Lemma \ref{lem:initial+time-terms} leads to the following result:

\begin{lem}\label{lem:initial+time-terms_extd}
	It holds that $I_1=0$ and since $u_0^{\eps} \underset{\eps\to0} \longrightarrow u_0$ in $L^2(\D)$, and $v_0^{\theta} \underset{\theta\to0} \longrightarrow v_0$ in $L^2(\D)$, we have
	\begin{align*}
	&\limsup_{\theta \goto 0}\,\limsup_{\eps \goto 0}\,\,\lim_{\delta\goto 0} \,\lim_{\gamma\to 0}\,\lim_{n\goto \infty} \big(I_1 + J_1\big) \le  \E \left[\int_{\R^d}\int_{\R^d} |u_0(x)-v_0(y)|\,\varphi(0,x)\rho_m(x-y)\,dy \,dx \right]+Cl.\\
	&\limsup_{\theta \goto 0}\,\limsup_{\eps \goto 0}\,\,\lim_{\delta\goto 0} \,\lim_{\gamma\to 0}\,\lim_{n\goto \infty}  \big(I_2 + J_2\big) \\
	& \hspace{3cm}\le \E \Big[ \int_{Q_T}\int_0^1 \int_0^1 
	|u(t,x,\alpha)-v(t,y,\beta)| \partial_t\varphi(t,x)\rho_m(x-y)
	d\alpha \,d\beta\,dx\,dt\Big]+Cl.
	\end{align*}
\end{lem}

\noindent Similarly, a straightforward adaptation of Lemma~\ref{lem:flux-terms} reveals that
\begin{lem}\label{lem:flux-terms_extd}
	The following holds:
	\begin{align*}
	&\limsup_{\theta \to 0} \limsup_{\ep \to 0}\lim_{\delta \to 0} \lim_{\gamma \to 0}\lim_{n \to \infty}( I_5+J_5)\\
	&\qquad \le Clm +Cl  -\E \left[ \int_{Q_T} \int_\D \int_{(0,1)^2} F(u(t,x,\alpha),u(t,y,\beta))\cdot \nabla_x \varphi(t,x) \rho_m(x-y)\,d\alpha\,d\beta\,dy\,dx\,dt \right]
	\end{align*}
\end{lem}

\noindent Regarding the noise term, we have the following variant of Lemma~\ref{lem:stochastic-terms}.
\begin{lem}\label{lem:noise_extd}
	The following holds:
	\begin{align*}
	&\limsup_{\theta \to 0} \limsup_{\ep \to 0}\lim_{\delta \to 0} \lim_{\gamma \to 0}\lim_{n \to \infty}( I_3 +J_3 + I_4 +J_4) \le C \Big(l +\frac{1}{lm^{2}}\Big).
	\end{align*}
\end{lem}
\begin{proof}
To prove this, we can follow Lemma~\ref{lem:stochastic-terms} and pass to the initial limits i.e., $n,\gamma$, and $\delta$ to obtain the following expression
\begin{align*}
&\E \bigg[\sum_{j \ge1}\dint_{Q_T\times \R^d} ( g_j(x,u_\theta(t,x)-g_j(y,u_\ep(t,y))^2 \varphi(t,x) \rho_m(x-y)\rho_l(u_\ep(t,y)- u_\theta(t,x))\,dy\,dx\,dt\bigg]\\
& \le 2\E\bigg[\dint_{Q_T\times \R^d} \big(|x-y|^{2}+|u_\theta(t,x)-u_\ep(t,y)|^2\big) \varphi(t,x) \rho_m(x-y)\rho_l(u_\ep(t,y)- u_\theta(t,x))\,dy\,dx\,dt\bigg]\\
&\le 2\Big(\frac{1}{l m^2}+l\Big)\E \bigg[\dint_{Q_T\times \R^d} \varphi(t,x) \rho_m(x-y)\,dy\,dx\,dt \bigg] \le C \Big(l +\frac{1}{lm^{2}}\Big).
\end{align*}
This finishes the proof of the lemma.
\end{proof}

Regarding the terms $I_6$ and $J_6$, we have the following result:
\begin{lem}\label{lem:eps_extd}
It follows that
\begin{align*}
&\limsup_{\theta \to 0} \limsup_{\ep \to 0}\lim_{\delta \to 0} \lim_{\gamma \to 0}\lim_{n \to \infty}( I_6 + J_6) =0.
\end{align*}
\end{lem}
\begin{proof}
We first notice that after passing to the limits in initial parameters $n,\gamma$, and $\delta$ in $I_6$, we obtain the following term
\begin{align*}
\lim_{n, \gamma, \delta} I_6&= 
- \eps \E\bigg[\dint_{\R\times Q_T \times \R^d}\sgn(u_\eps(t,y)-k) \nabla u_\eps(t,y) \nabla_y\rho_m(x-y) \varphi(x,t)\rho_l(u_\theta(t,x)-k) \,dy\,dk\,dx\,dt\bigg] \\
&\le \eps \E\bigg[\dint_{Q_T \times \R^d} |\nabla u_\eps(t,y)| |\nabla_y\rho_m(x-y)| \varphi(x,t) \,dy\,dx\,dt\bigg] \approx \mathcal{O}(\ep m^2).
\end{align*}
Similar estimate also holds for the term $J_6$. Hence the result follows.
\end{proof}

Remaining terms are expected from the fractional operator. First we look at the regular part $\mathscr{L}^r_\lambda$ of the nonlocal term. In what follows, we start with the following lemma.

\begin{lem}\label{lem:frac_01_extd}
	The following hold:
	\begin{align*}
	&\limsup_{\theta \to 0}\,\limsup_{\ep \to 0}\,\,\lim_{\delta\to 0}\,\lim_{\gamma \to 0} \, \lim_{n \to \infty} (I_7+J_7) \\
	& \quad \le -\E \bigg[\int_{Q_T \times \R^d\times (0,1)^2} | A(u(t,x,\alpha)) - A( v(t,y,\beta)) |\mathscr{L}^r_\lambda[ \varphi(t,\cdot)](x) \rho_m(x-y) \,d\alpha\,d\beta\,dx\,dy\,dt\bigg]+\frac{C\|A'\|_\infty l}{\lambda r^{2\lambda}}
	\end{align*}
\end{lem}

\begin{proof}
We refer to Lemma~\ref{kato_lemma1} for the passage to the limits in the initial parameters \textit{i.e.} $n,\gamma,\delta$; the resulted expression is
\begin{align}
&- \E\bigg[ \int_{Q_T \times \R^d \times \R} \mathscr{L}^r_\lambda[A(u_\ep(t,\cdot))](y) \varphi(t,x)\rho_m(x-y)\, \sgn ( u_\ep(t,y) - k) \rho_l(u_\theta(t,x)-k) \,dy \,dk \,dx\,dt\bigg] \notag\\
&\hspace{0.6cm}-\E\bigg[ \int_{Q_T \times \R^d \times \R} \mathscr{L}^r_\lambda[A(u_\theta(t,\cdot))](x) \varphi(t,x)\rho_m(x-y)\, \sgn( u_\theta(t,x)- k)\rho_l(u_\ep(t,y)-k)  \,dy\,dk \,dx\,dt\bigg] \notag\\
&= -\E\bigg[ \int_{Q_T \times \R^d \times \R} \mathscr{L}^r_\lambda[A(u_\ep(t,\cdot))](y) \varphi(t,x)\rho_m(x-y)\, \sgn ( u_\ep(t,y) - u_\theta(t,x) - k) \rho_l(k) \,dy \,dk \,dx\,dt\bigg]\notag \\
&\hspace{0.6cm}-\E\bigg[ \int_{Q_T \times \R^d \times \R} \mathscr{L}^r_\lambda[A(u_\theta(t,\cdot))](x) \varphi(t,x)\rho_m(x-y)\, \sgn( u_\theta(t,x)-u _\ep(t,y)+ k)\rho_l(k)  \,dy\,dk \,dx\,dt\bigg]\notag \\
&=-\E\bigg[ \int_{Q_T \times \R^d \times \R} (\mathscr{L}^r_\lambda[A(u_\theta(t,\cdot))](x)-\mathscr{L}^r_\lambda[A(u_\ep(t,\cdot))](y))\sgn( u_\theta(t,x)-u _\ep(t,y)+ k) \notag \\[-0.3cm]
&\hspace{8cm}\times \varphi(t,x)\rho_m(x-y)\rho_l(k) \,dy\,dk\,dx\,dt \bigg] \notag\\
& = \E \bigg[\int_{\R^d} \int_{Q_T}\int_\R  \bigg[\int_{|z| >r} (A(u_\theta(t,x+z))- A(u_\theta(t,x))) -( A(u_\ep(t,y+z))-A(u_\ep(t,y))) \d\mu_\lambda(z) \bigg]\notag \\[-0.3cm]
&\hspace{6cm} \times \varphi(t,x) \rho_m(x-y) \rho_l(k)\, \sgn( u_\theta(t,x)-u_\ep(t,y)+k)\,dk\,dy\,dx\,dt \bigg] \notag\\
&\bigg( \text{making use of Lemma \ref{lemma_01}}\bigg)\notag \\
&\le - \E \bigg[ \int_\D \int_{Q_T} |A(u_\theta(t,x))-A(u_\ep(t,y))| \fr^r[ \varphi(t,\cdot)](x) \rho_m(x-y) \,dy\,dx\,dt \bigg] +\frac{C\| A'\|_\infty l}{\lambda r^{2\lambda}}. \label{imp_1}
\end{align} 
At this point we can pass to the limits in $\eps$ and $\theta$ in \eqref{imp_1}, with the help of Young measures theory (cf. Lemma~\ref{kato_lemma1}), to conclude the proof of the lemma.
\end{proof}

Finally, we are left with the irregular part $\mathscr{L}_{\lambda,r}$ of the nonlocal term. We deal with these terms in the following lemma.
\begin{lem}\label{lem:frac_02_extd}
	The following holds:
	\begin{align*}
	&\limsup_{\theta \to 0} \limsup_{\ep \to 0}\lim_{\delta \to 0}\lim_{\gamma \to 0} \lim_{n \to \infty}(I_8+J_8) \\
	& \le \E \left[\int_{Q_T \times \R^d \times (0,1)^2} | A(v(t,y,\beta))-A( u(t,x,\alpha))|\,| \mathscr{L}_{\lambda,r}[ \rho_m(x-\cdot)](y)|\varphi(t,x) \,d\beta\,d\alpha\,dy\,dx\,dt\right]\\
	&\qquad \quad + l\|A'\|_\infty\E \left[ \int_{Q_T \times \R^d }|\mathscr{L}_{\lambda,r}[ \rho_m(x-\cdot)](y)|\varphi(t,x) \,dy\,dx\,dt\right]\\
	& \hspace{1cm}+\E \left[\int_{Q_T\times \R^d\times (0,1)^2} | A(u(t,x,\alpha))-A(v(t,y,\beta))|\, | \mathscr{L}_{\lambda,r}[ \varphi(t,\cdot) \rho_m(\cdot-y)](x)|  \,d\beta\,d\alpha\,dy\,dx\,dt\right]\\
	&\qquad\quad +l \|A'\|_\infty \E \left[ \int_{Q_T \times \R^d }| \mathscr{L}_{\lambda,r}[ \varphi(t,\cdot) \rho_m(\cdot-y)](x)| \,dy\,dx\,dt\right]
	\end{align*}
\end{lem}

\begin{proof}
Note that we can follow  Lemma~\ref{kato_lemma2} to pass limits in the initial parameters, \textit{i.e.} $n, \gamma$, and $\delta$. The outcome is 
\begin{align*}
&-\E \bigg[\int_{Q_T \times \R^d\times \R} | A(u_\ep(t,y))-A(k)| \mathscr{L}_{\lambda,r}[ \rho_m(x-\cdot)](y)\varphi(t,x)\rho_l(u_\theta(t,x)-k) \,dy\,dk\,dx\,dt\bigg]\\
&\hspace{1cm}-\E \bigg[\int_{Q_T\times \R^d\times \R} | A(u_\theta(t,x))-A(k)| \mathscr{L}_{\lambda,r}[ \varphi(t,\cdot) \rho_m(\cdot-y)](x) \rho_l(u_\ep(t,y)-k) \,dy\,dk\,dx\,dt\bigg].
\end{align*}
We shall only estimate the first term (the second term yields the same result): we see that
\begin{align*}
&-\E \bigg[\int_{Q_T \times \R^d\times \R} | A(u_\ep(t,y))-A(k)| \mathscr{L}_{\lambda,r}[ \rho_m(x-\cdot)](y)\varphi(t,x)\rho_l(u_\theta(t,x)-k) \,dy\,dk\,dx\,dt\bigg]\\
&=-\E \bigg[\int_{Q_T \times \R^d\times \R} | A(u_\ep(t,y))-A(u_\theta(t,x)-k)| \mathscr{L}_{\lambda,r}[ \rho_m(x-\cdot)](y)\varphi(t,x)\rho_l(k) \,dy\,dk\,dx\,dt\bigg]\\
& \le\E \bigg[\int_{Q_T \times \R^d} | A(u_\ep(t,y))- A(u_\theta(t,x))|\, |\mathscr{L}_{\lambda,r}[ \rho_m(x-\cdot)](y)|\varphi(t,x) \,dy\,dx\,dt\bigg]\\
&\qquad+ l\,\|A'\|_\infty\,\E \bigg[\int_{Q_T \times \R^d} |\mathscr{L}_{\lambda,r}[ \rho_m(x-\cdot)](y)|\varphi(t,x)\,dy\,dx\,dt\bigg].
\end{align*}
Making use of similar arguments, as depicted in Lemma~\ref{kato_lemma2}, we can pass to the limits in $\eps$ and $\theta$.
\end{proof}

Now we are in a position to combine Lemma~\ref{lem:frac_01_extd} and Lemma~\ref{lem:frac_02_extd} to conclude

\begin{lem}
\label{kato_lemma3_extd}
It holds that,
\begin{align*}
&\limsup_{\theta \goto 0}\limsup_{\eps \goto 0}\lim_{\delta \goto 0}\lim_{\gamma \to 0}\lim_{n \goto \infty}(I_7+J_7+I_8+J_8) \le C\frac{l}{r^{2\lambda}}+C(1+l)r^{2-2\lambda}(1+m+m^2) \\
&\qquad -\E \bigg[\int_{Q_T \times \R^d\times (0,1)^2} | A(u(t,x,\alpha)) - A( v(t,y,\beta)) |\mathscr{L}_\lambda[ \varphi(t,\cdot)](x) \rho_m(x-y) \,d\alpha\,d\beta\,dx\,dy\,dt \bigg].
\end{align*} 
\end{lem}

\begin{proof}
The proof of the above lemma follows from the following observations. First note that we can follow Lemma~\ref{kato_lemma3} to conclude that
\begin{align*}
& - \E \bigg[ \int_\D \int_{Q_T}\int_0^1 \int_0^1 | A(u(t,x,\alpha)) - A( v(t,y,\beta)) | \fr^r[ \varphi(t,\cdot)](x) \rho_m(x-y) \,dy\,dx\,dt\,d\alpha\,d\beta\bigg] \\ 
&\quad \le - \E \bigg[ \int_\D \int_{Q_T} \int_0^1 \int_0^1 | A(u(t,x,\alpha)) - A( v(t,y,\beta)) | \fr[ \varphi(t,\cdot)](x) \rho_m(x-y) \,dy\,dx\,dt\,d\alpha\,d\beta \bigg] + C r^{2-2\lambda}.
\end{align*}
Next, we shall make use of the following observation
\begin{align*}
|\mathscr{L}_{\lambda,r}[ \rho_m(x-\cdot)](y)| \le \int_{|z|<r} \int_0^1 \frac{|z|^2 |D^2 \rho_m(x -y -\tau z)|}{|z|^{d+2\lambda}} \,d\tau\,dz,
\end{align*}
to conclude
\begin{align*}
&\E \left[\int_{Q_T \times \R^d \times (0,1)^2} \hspace{-0.5cm}| A(v(t,y,\beta))-A( u(t,x,\alpha))|\,| \mathscr{L}_{\lambda,r}[ \rho_m(x-\cdot)](y)|\varphi(t,x) \,d\beta\,d\alpha\,dy\,dx\,dt\right] \le C m^2r^{2-2\lambda},  \\
&\E \left[\int_{Q_T\times \R^d\times (0,1)^2} \hspace{-1.2cm}| A(u(t,x,\alpha))-A(v(t,y,\beta))|\, | \mathscr{L}_{\lambda,r}[ \varphi(t,\cdot) \rho_m(\cdot-y)](x)|  \,d\beta\,d\alpha\,dy\,dx\,dt\right] \le C  r^{2-2\lambda}(1 + m + m^2), \\
& \qquad l\|A'\|_\infty\E \left[ \int_{Q_T \times \R^d }|\mathscr{L}_{\lambda,r}[ \rho_m(x-\cdot)](y)|\varphi(t,x) \,dy\,dx\,dt\right] \le C l m^2r^{2-2\lambda},\\
& \qquad l \|A'\|_\infty \E \left[ \int_{Q_T \times \R^d }| \mathscr{L}_{\lambda,r}[ \varphi(t,\cdot) \rho_m(\cdot-y)](x)| \,dy\,dx\,dt\right]\le C  l r^{2-2\lambda}(1 + m + m^2).
\end{align*}
Combining the above results, we conclude the proof of the lemma.
\end{proof}

Using Lemma~\ref{lem:initial+time-terms_extd},  Lemma~\ref{lem:flux-terms_extd}, Lemma~\ref{lem:noise_extd}, Lemma~\ref{lem:eps_extd} and Lemma~\ref{kato_lemma3_extd}, we have the following version of  Kato's inequality: 
\begin{align}
0 \leq& \int_\D |u_0-v_0|\varphi(0)\,dx + \E\bigg[\int_{Q_T} \int^1_{0}\int^1_{0}  |u(\alpha,t,x) -v(\beta,t,x)| \partial_t \varphi(t,x)\,d\alpha\,d\beta\,dx\,dt\bigg]\notag\\
&\qquad -\E\bigg[\int_{Q_T} \int^1_{0}\int^1_{0}  F(u(\alpha,t,x),v(\beta,t,x))\nabla \varphi(t,x)\,dx\,d\alpha d\beta\,dt 
\bigg]+C\left(\frac{1}{l m^2}+lm+l\right) \notag\\
&\qquad \qquad -\E\bigg[\int_{Q_T} \int^1_{0} \int^1_{0} |A(u(\alpha,t,x))-A(v(\beta,t,x))|  \mathscr{L}_\lambda[\varphi(t,\cdot)](x) \,dx\,d\alpha d\beta\,dt\bigg]\notag \\
&\qquad \qquad \qquad +C\left(\frac{l}{r^{2\lambda}}+(1+l)r^{2-2\lambda}(1+m+m^2)\right). \label{kato_extd}
\end{align}

Next, our aim is to pass to the limits in the remaining parameters \textit{i.e.} $m, l$ and $r$ in \eqref{kato_extd}. This forces us to optimize the following terms
\begin{equation}\label{conditions}
\frac{\xi^{2}}{l}, \frac{l}{r^{2\lambda}}, \frac{r^{2-2\lambda}}{\xi^2},
\end{equation}
where for convenience we set $\xi = \frac{1}{m}$ and use the fact that $m^2r^{2-2\lambda}> lm^2r^{2-2\lambda}$, for small $l$. In order to simultaneously pass to the limit when $\xi,l,r \to 0$, we suppose that there is some $\theta_1, \theta_2 >0 $ such that $\xi^{2}= l^\theta_1\,, r=l^\theta_2$, then we must have $\theta_1-1>0,\, 1-2\lambda \theta_2>0$ from the first and the second term of \eqref{conditions}. From the third term of \eqref{conditions}, we must have $\theta_2(1-\lambda)-\frac{\theta_1}{2} >0$. Therefore the conditions we need are as follows $$ \theta_1 > 1, \,  \frac{1}{2\lambda}> \theta_2,\, 1> \frac{\lambda}{1-\lambda}.$$
Since by our assumption $\lambda \in [0,1/2)$, we can pass to the limit and reach at the desired Kato's inequality \eqref{kato}. Then we can proceed as in Subsection~\ref{Well-posedness} to conclude the uniqueness of the solution.


\appendix
\section{Entropy Inequality}
\label{appendix_entropy}
In this section, we present a formal derivation of the entropy inequality for the regularized equation \eqref{eq:viscous-Brown}. In what follows, let $\varphi$ in $\mathcal{D}^+(Q_T)$, $k$ a real number and $\eta$ in $\mathcal{E}$. Let $(\eta,\zeta)$ be an entropy flux pair. Given a non-negative test function $\varphi\in C_{c}^{1,2}([0,\infty )\times\R^d)$, as $u_\ep \in L^2((0,T)\times \Omega; H^1(\D))$ we apply a weak generalized version of It\^{o} formula (cf. \cite[Appendix A]{BisMajVal}) to yield, for all $T >0$
\begin{align*}
&\int_\D\eta(u_\eps(T,x)-k)\varphi(T,x) \,dx=
\int_\D\eta(u_\ep(0,x)-k)\varphi(0,x) \,dx\\
&  + \int_{Q_T} \eta(u_\ep(t,x) - k)\partial_t\varphi(t,x) \,dx\,dt  -\int_{Q_T} \nabla \varphi(t,x) \cdot \zeta( u_\ep(t,x))\,dx\,dt\\
& - \eps \int_{Q_T} [ \eta'(u_\eps(t,x)-k)\nabla u_\eps(t,x) \nabla \varphi(t,x)  +  \eta^{\prime\prime}(u_\eps(t,x)-k)\varphi |\nabla u_\eps(t,x)|^2] \,dx\,dt
\\
& + \int_0^T\int_\D \eta'(u_\ep(t,x) -k) \h(u_\ep(t,x))\varphi(t,x) \,dx \,dW(t) + \frac{1}{2}\int_{Q_T} \mathbb{G}^2(u_\ep(t,x))\eta^{\prime\prime}(u_\ep(t,x) - k)\varphi(t,x) \,dx\,dt\\
& - \frac12\int_{(0,T)\times(\D)^2}\frac{[A(u_\eps(t,x))-A(u_\eps(t,y))][(\varphi \eta'(u_\eps-k)) (t,x)-(\varphi \eta'(u_\eps -k)) (t,y)]}{|x-y|^{d+2\lambda}}\,dx\,dy\,dt.
\end{align*}
For technical reasons, it seems essential to split the non-local term. Indeed, following  \cite{CifaniJakobsen}, for any fixed positive $r$, we write
\begin{align}
& \int_{(0,T)\times(\D)^2}\frac{[A(u_\eps(t,x))-A(u_\eps(t,y))][(\varphi \eta'(u_\eps-k)) (t,x)-(\varphi \eta'(u_\eps-k)) (t,y)]}{|x-y|^{d+2\lambda}}\,dx\,dy\,dt 
 \notag \\ =& \int_{(0,T)\times \{|x-y|\leq r\}}\frac{[A(u_\eps(t,x))-A(u_\eps(t,y))][(\varphi \eta'(u_\eps-k)) (t,x)-(\varphi \eta'(u_\eps-k)) (t,y)]}{|x-y|^{d+2\lambda}}\,dx\,dy\,dt 
\notag \\& +  \int_{(0,T)\times \{|x-y| > r\}} \hspace{-0.2cm} \frac{[A(u_\eps(t,x))-A(u_\eps(t,y))][(\varphi \eta'(u_\eps-k)) (t,x)-(\varphi \eta'(u_\eps-k)) (t,y)]}{|x-y|^{d+2\lambda}}\,dx\,dy\,dt := I_r + I^r. \notag
\end{align}
Note that $A(u_\eps(t)) \in H^1(\D)$ and that $\varphi(t)$ is with compact support, so that  
\begin{align*}
I^r&=\int_{(0,T)\times \{|x-y| > r\}}\frac{[A(u_\eps(t,x))-A(u_\eps(t,y))][(\varphi \eta'(u_\eps-k)) (t,x)-(\varphi \eta'(u_\eps-k)) (t,y)]}{|x-y|^{d+2\lambda}}\,dx\,dy\,dt 
\\
&=\int_{(0,T)\times \{|x-y| > r\}}\frac{[A(u_\eps(t,x))-A(u_\eps(t,y))](\varphi \eta'(u_\eps-k)) (t,x)}{|x-y|^{d+2\lambda}}\,dx\,dy\,dt
\\&\qquad -
\int_{(0,T)\times \{|x-y| > r\}}\frac{[A(u_\eps(t,x))-A(u_\eps(t,y))](\varphi \eta'(u_\eps-k)) (t,y)}{|x-y|^{d+2\lambda}}\,dx\,dy\,dt
\\
&=\int_{(0,T)\times \D\times \{|z| > r\}}\frac{[A(u_\eps(t,x))-A(u_\eps(t,x+z))](\varphi \eta'(u_\eps-k)) (t,x)}{|z|^{d+2\lambda}}\,dx\,dz\,dt
\\&\qquad -
\int_{(0,T)\times \D\times\{|z| > r\}}\frac{[A(u_\eps(t,y+z))-A(u_\eps(t,y))](\varphi \eta'(u_\eps-k)) (t,y)}{|z|^{d+2\lambda}}\,dz\,dy\,dt
\\
&=2 \int_{Q_T}[\varphi \eta'(u_\eps(t,x)-k)] (t,x)\int_{\{|z| > r\}}\frac{[A(u_\eps(t,x))-A(u_\eps(t,x+z))]}{|z|^{d+2\lambda}}\,dz\,dx\,dt\\
& := 2\int_{Q_T} \mathscr{L}^r_{\lambda}(A(u_\eps(t,x))) \varphi \eta'(u_\ep(t,x)-k)\,dx\,dt.
\end{align*}
Moreover, since for any $(a,b)$, $[A(a)-A(b)]\eta^\prime(a-k) \geq A^\eta_k(a)-A^\eta_k(b)\geq [A(a)-A(b)]\eta^\prime(b-k)$, 
we get, 
\begin{align*}
&[A(u_\eps(t,x))-A(u_\eps(t,y))][(\varphi \eta'(u_\eps-k)) (t,x)-(\varphi \eta'(u_\eps-k)) (t,y)]\\
& \hspace{3cm} \ge [A^\eta_k(u_\eps(t,x))-A^\eta_k(u_\eps(t,y))][\varphi(t,x)-\varphi(t,y)].
\end{align*}
Then, 
\begin{align*}
I_r \geq &\int_{(0,T)\times \{|x-y|\leq r\}}\frac{[A^\eta_k(u_\eps(t,x))-A^\eta_k(u_\eps(t,y))][\varphi(t,x)-\varphi(t,y)]}{|x-y|^{d+2\lambda}}\,dx\,dy\,dt
\\
=& \int_{(0,T)}\lim_{\delta \to 0} \int_{\{\delta<|x-y|\leq r\}} \frac{[A^\eta_k(u_\eps(t,x))-A^\eta_k(u_\eps(t,y))][\varphi(t,x)-\varphi(t,y)]}{|x-y|^{d+2\lambda}}\,dx\,dy\,dt
\\
=& \bigg[ \int_{(0,T)} \lim_{\delta \to 0}\int_{ \{\delta<|x-y|\leq r\}}\frac{A^\eta_k(u_\eps(t,x))[\varphi(t,x)-\varphi(t,y)]}{|x-y|^{d+2\lambda}}\,dx\,dy\,dt\\
&\hspace{3cm} -\int_{(0,T)} \lim_{\delta \to 0} \int_{\{\delta<|x-y|\leq r\}}\frac{A^\eta_k(u_\eps(t,y))[\varphi(t,x)-\varphi(t,y)]}{|x-y|^{d+2\lambda}}\,dx\,dy\,dt\bigg]
\\=&2\int_{(0,T)}\lim_{\delta \to 0} \int_{\D \times \{\delta<|z|\leq r\}}A^\eta_k(u_\eps(t,x))\frac{(\varphi(t,x)-\varphi(t,x+z))}{|z|^{d+2\lambda}}\,dx\,dz\,dt
\\=&2\int_{Q_T} A^\eta_k(u_\eps(t,x)) \text{ PV} \int_{\{|z|\leq r\}}\frac{\varphi(t,x)-\varphi(t,x+z)}{|z|^{d+2\lambda}}\,dz\,dx\,dt\\
& := 2\int_{Q_T} A^\eta_k(u_\eps(t,x)) \mathscr{L}_{\lambda,r}[\varphi(t,\cdot)](x)\,dx\,dt
\end{align*}
since $\varphi$ is with a compact support and thanks to \cite[Section 2]{CifaniJakobsen}.
\smallskip

As the other terms of the equation are dealt in \cite{BaVaWit_2012}, we conclude that a viscous-entropy formulation is 
\begin{align*}
0&\leq \int_\D\eta(u_\ep(0,x)-k)\varphi(0) \,dx + \int_{Q_T} \eta(u_\eps(t,x) - k)\partial_t\varphi(t,x)- \zeta(u_\ep(t,x))\cdot \nabla \varphi(t,x) \,dx\,dt \notag\\
&- \eps \int_{Q_T} \eta'(u_\eps(t,x)-k)\nabla u_\eps(t,x)\nabla \varphi(t,x) \,dx\,dt  \notag\\
&-\int_{Q_T} [\mathscr{L}^r_\lambda (A(u_\eps(t,x))) \varphi(t,x) \eta'(u_\eps(t,x)-k) + A^\eta_k(u_\eps(t,x)) \mathscr{L}_{\lambda,r}(\varphi(t,x))]\,dx\,dt \notag\\
&+ \int_{Q_T} \eta'(u_\eps(t,x) -k) \h(u_\eps(t,x))\varphi(t,x) \,dx \,dW(t) + \frac{1}{2}\int_{Q_T}\mathbb{G}^2(u_\eps)\eta^{\prime\prime}(u_\eps(t,x) - k)\varphi(t,x) \,dx\,dt.
\end{align*}
Clearly, the above inequality is stable under the limit $\eps \rightarrow 0$, if the family ${\lbrace u_\ep \rbrace}_{\ep>0}$ has $L^p_{\mathrm{loc}}$-type stability. Indeed, the above inequality provides us the entropy inequality presented in Definition~\ref{Defi_Entropy_formulation}.

\section*{Acknowledgements}
U.K acknowledges the support of the Department of Atomic Energy,  Government of India, under project no.$12$-R$\&$D-TFR-$5.01$-$0520$, and India SERB Matrics grant MTR/$2017/000002$. Both U.K and G.V are supported in part by Indo-French Centre for Applied Mathematics : UMI-CNRS 3494, IFCAM.

 \end{document}